\documentclass[11pt]{article}
\usepackage{amssymb}
\usepackage{fullpage,subfigure}
\usepackage{hyperref}
\usepackage{graphicx,amsmath,amsfonts,amscd,amssymb,bm,cite,epsfig,epsf,url,color}
\usepackage{amsthm}
\usepackage{array} 
\usepackage{fullpage} 
\usepackage[small,bf]{caption}
\usepackage{bbm}
\usepackage{pgfplots}
\usetikzlibrary{spy}
\newtheorem{theorem}{Theorem}[section]
\newtheorem{corollary}[theorem]{Corollary}
\newtheorem{lemma}[theorem]{Lemma}
\newtheorem{proposition}[theorem]{Proposition}
\newtheorem{definition}[theorem]{Definition}

\newcommand{\R}{\mathbb{R}}
\newcommand{\C}{\mathbb{C}}

\newcommand{\Z}{\mathbb{Z}}
\newcommand{\<}{\langle}
\renewcommand{\>}{\rangle}

\newcommand{\norm}[1]{{\left\lVert{#1}\right\rVert}}

\newcommand{\Id}{\text{\em I}}

\numberwithin{equation}{section}

\newcommand{\normInf}[1]{\left|\left| #1 \right|\right| _{\infty}}
\newcommand{\normTwo}[1]{\left|\left| #1 \right|\right| _{2}}
\newcommand{\normOneTwo}[1]{\left|\left| #1 \right|\right| _{\ell_{1}/\ell_{2}}}
\newcommand{\normOne}[1]{\left|\left| #1 \right|\right| _{1}}

\newcommand{\normInfInf}[1]{\left|\left| #1 \right|\right| _{\infty}}

\newcommand{\normTV}[1]{\left|\left| #1 \right|\right| _{\text{TV}}}

\newcommand{\normgTV}[1]{\left|\left| #1 \right|\right| _{\text{gTV}}}

\newcommand{\abs}[1]{\left| #1 \right|}
\newcommand{\keys}[1]{\left\{ #1 \right\}}
\newcommand{\sqbr}[1]{\left[ #1 \right]}
\newcommand{\brac}[1]{\left( #1 \right) }
\newcommand{\MAT}[1]{\begin{bmatrix} #1 \end{bmatrix}}

\newcommand{\PROD}[2]{\left \langle #1, #2\right \rangle}

\newcommand{\derTwo}[2]{\frac{\text{d}^2#2}{\text{d}#1^2}}

\newcommand{\Real}[1]{\text{Re}\brac{ #1 }}

\newcommand{\Set}{S}
\newcommand{\regpar}{\eta}
\newcommand{\fc}{f_c}
\newcommand{\fct}{f_c  \, t}
\newcommand{\fmin}{f_\text{\text{min}}}
\newcommand{\fmineq}{f_\text{\text{\em min}}}

\newcommand{\Deltamin}{\Delta_{\text{min}}}
\newcommand{\Deltaminth}{\Delta_{\text{\em min}}}
\newcommand{\taumin}{\tau_{\text{min}}}
\newcommand{\tauminth}{\tau_{\text{\em min}}}
\newcommand{\optvalue}{1.26}

\newcommand{\minfc}{10^{3}}
\newcommand{\taumiddle}{0.288316}
\newcommand{\taumiddleSecondDer}{0.110497}

\newcommand{\jnear}{20}
\newcommand{\jnearpone}{21}
\newcommand{\jzero}{400}

\newcommand{\gammaOne}{0.247}
\newcommand{\gammaTwo}{0.339}
\newcommand{\gammaThree}{0.414}
\newcommand{\xLR}{x_{\text{LR}}}

\author{Carlos
  Fernandez-Granda\thanks{Courant Institute of Mathematical Sciences and Center for Data Science,
    NYU, New York City NY}}

\title{Super-Resolution of Point Sources via Convex Programming}

\date{July 2015; Revised December 2015}

\begin{document}

\maketitle


\begin{abstract}
We consider the problem of recovering a signal consisting of a superposition of point sources from low-resolution data with a cut-off frequency $\fc$. If the distance between the sources is under $1/\fc$, this problem is not well posed in the sense that the low-pass data corresponding to two different signals may be practically the same. We show that minimizing a continuous version of the $\ell_1$ norm achieves exact recovery as long as the sources are separated by at least $ \optvalue / \fc$. The proof is based on the construction of a \emph{dual certificate} for the optimization problem, which can be used to establish that the procedure is stable to noise. Finally, we illustrate the flexibility of our optimization-based framework by describing extensions to the demixing of sines and spikes and to the estimation of point sources that share a common support. 


\end{abstract}


\section{Introduction}
\label{sec:intro}
Extracting fine-scale information from low-resolution data is a major challenge in many areas of the applied sciences. In microscopy, astronomy and any other application employing an optical device, spatial resolution is fundamentally limited by diffraction~\cite{superres_survey}. 
Figure~\ref{fig:superres_diagram} illustrates a popular model for the data-acquisition process in such cases: the object of interest is convolved with a point-spread function that blurs the fine-scale details, acting essentially as a low-pass filter. The problem of super-resolution is that of reconstructing the original image from the blurred measurements. 
An analogous challenge often arises in signal-processing: estimating the spectrum of a signal from a finite number of samples. Truncating the signal in the time domain limits the spectral resolution, as shown in the lower half of Figure~\ref{fig:superres_diagram}. \emph{Spectral super-resolution}, or equivalently line-spectra estimation, is the problem of recovering the spectrum of the original signal from the truncated data. 

\begin{figure}
\hspace{1.5cm}
\begin{tabular}{
>{\centering\arraybackslash}m{0.06\linewidth} >{\centering\arraybackslash}m{0.28\linewidth}   >{\centering\arraybackslash}m{0.28\linewidth} }
  && Spectrum (real part) \vspace{0.2cm}\\
 Signal & \begin{tikzpicture}[scale=0.7]
\begin{axis}[ticks=none]
\addplot+[ycomb,mark=none, very thick,blue]
file {data_x.dat};
\addplot[black] coordinates {(-8,0) (10,0)};
\end{axis}
\end{tikzpicture}  
&
\begin{tikzpicture}[scale=0.7]
\begin{axis}[ticks=none]
\addplot[darkgray,thick]
file {data_x_spect_real.dat};
\addplot[black] coordinates{(-2,0) (2,0)};
\end{axis}
\end{tikzpicture}
\\
Data & \begin{tikzpicture}[scale=0.7]
\begin{axis}[ticks=none]
\addplot[blue, thick] file {data_x_lowres.dat};
\addplot[black] coordinates
{(-8,0) (10,0)};
\end{axis}
\end{tikzpicture}
&
\begin{tikzpicture}[scale=0.7]
\begin{axis}[ticks=none]
\addplot[ darkgray, thick]
 file {data_x_spect_lowres_real.dat};
\addplot[black] coordinates{(-2,0) (2,0)};
\end{axis}
\end{tikzpicture}
\end{tabular}
\begin{center}
\textbf{Spatial super-resolution}
\end{center}
\hspace{1.5cm}
\begin{tabular}{
>{\centering\arraybackslash}m{0.06\linewidth} >{\centering\arraybackslash}m{0.28\linewidth}   >{\centering\arraybackslash}m{0.28\linewidth} }
  && Spectrum (magnitude) \vspace{0.2cm}\\
 Signal & 
\begin{tikzpicture}[scale=0.7]
\begin{axis}[ticks=none]
\addplot[blue, thick]
 file {data_long.dat};
\addplot[black] coordinates{(-60,0) (60,0)};
\end{axis}
\end{tikzpicture}
&
\begin{tikzpicture}[scale=0.7]
\begin{axis}[ticks=none]
\addplot+[ycomb,mark=none, very thick,darkgray]
 file {data_spectrum.dat};
\addplot[black] coordinates
{(-0.5,0) (0.5,0)};
\end{axis}
\end{tikzpicture}  
\\
Data &
\begin{tikzpicture}[scale=0.7]
\begin{axis}[ticks=none]
\addplot[blue, thick]
 file {data_trunc.dat};
\addplot[black] coordinates{(-60,0) (60,0)};
\end{axis}
\end{tikzpicture}
&
 \begin{tikzpicture}[scale=0.7]
\begin{axis}[ticks=none]
\addplot[darkgray, thick] file {data_spectrum_trunc.dat};
\addplot[black] coordinates
{(-0.5,0) (0.5,0)};
\end{axis}
\end{tikzpicture}
\end{tabular}
\begin{center}
\textbf{Spectral super-resolution}
\end{center}
\caption{Schematic illustration of spatial and spectral super-resolution.}
\label{fig:superres_diagram}
\end{figure}

By \emph{super-resolution} we mean the inverse problem of estimating a signal from low-resolution measurements, but the term may have other meanings in different contexts. In optics, it often refers to the problem of overcoming the diffraction limit by modifying the data-acquisition mechanism~\cite{superres_survey}. In image-processing, it denotes the problem of upsampling an image onto a finer grid while preserving its edge structure and hallucinating high-frequency textures in a reasonable way~\cite{book_milanfar}. To be clear, in this work we focus on \emph{recovering} the lost fine-scale features \emph{without altering} the low-pass sensing process. 

In order to super-resolve a signal it is necessary to leverage some prior knowledge about its structure. Otherwise the problem is hopelessly ill posed, since the missing spectrum can be filled in arbitrarily to produce estimates that correspond to the data. Here, we consider signals that may be represented as superpositions of point sources, such as celestial bodies in astronomy~\cite{ghez_astronomy}, neuron spikes in neuroscience~\cite{rieke_spikes} or line spectra in signal processing and spectroscopy~\cite{lajunen_spectroscopy, linespectra_astronomy}. In addition, locating pointwise fluorescent probes is a crucial step in some optical super-resolution procedures capable of handling more complicated objects, such as photoactivated localization microscopy (PALM)~\cite{palm,fpalm} or stochastic optical reconstruction microscopy (STORM)~\cite{storm}.

At an abstract level, the deconvolution of point sources or spikes from bandlimited data is an instance of a central question in modern data processing: how to recover a low-dimensional object embedded in a high-dimensional space from incomplete linear measurements. Nonparametric techniques based on convex optimization have had great success in tackling problems of this flavor. Notable examples include sparse regression in high-dimensional settings~\cite{lasso}, compressed sensing~\cite{candesRandProj,cs_donoho} and matrix completion~\cite{mc_candes}. 
The interest of developing optimization-based methods for super-resolution lies in their robustness to noise and in their flexibility to account for different structural assumptions on the signal, noise and measurement model. In recent work, convex programming has been shown to recover a superposition of point sources exactly from bandlimited data, as long as the sources are separated by a minimum distance of $2/\fc$, where $\fc$ is the cut-off frequency of the sensing process~\cite{superres}. Subsequent publications~\cite{robust_sr, support_detection, azais2015spike, tang_minimax, venia_positive} have established that the method is robust to noise in non-asymptotic regimes. 

The goal of the present paper is to further develop this line of research through two main contributions: 
\begin{itemize}
\item We establish that it is possible to super-resolve signals with minimum separations above $ \optvalue /\fc$ via convex programming. Section~\ref{sec:superres_point_sources} provides the context for this result by describing the basic super-resolution problem and our optimization-based approach. The proof, which  is based on the construction of a novel \emph{dual certificate} that also allows to extend previous stability results, is presented in Section~\ref{proof:noiseless}.
\item In Section~\ref{sec:convex_framework} we illustrate the flexibility of our approach by adapting it to two related problems: demixing of sines and spikes and super-resolution of multiple signals that share a common support. In each case we propose a optimization program tailored to the problem, discuss how to solve it, analyze its optimality conditions and provide some numerical simulations.
\end{itemize}

\section{Super-resolution of point sources}
\label{sec:superres_point_sources}
\subsection{Basic model}
\label{sec:noiseless}
We model a superposition of point sources as a sum of weighted Dirac measures supported on a subset $T$ of the unit interval
\begin{equation}
\label{eq:model}
  x := \sum_{t_j \in T} a_j \delta_{t_j}, 
\end{equation}
where $\delta_{\tau}$ is a Dirac measure at $\tau$ and the amplitudes $a_j$ may be complex valued. We study the problem of estimating such a signal from low-resolution measurements which correspond to the convolution between the signal and a low-pass point spread function (PSF) $\phi$,
\begin{align}
\label{eq:xLR}
\xLR \brac{t} := \phi \ast x \brac{t} = \sum_{t_j \in T} a_j \phi\brac{t-t_j},
\end{align}
as illustrated at the top of Figure~\ref{fig:superres_diagram}. If the cut-off frequency of the PSF is equal to $\fc$, in the frequency domain the measurements are of the form
\begin{align*}
\mathcal{F} \xLR = \mathcal{F}  \phi \; \mathcal{F} x  = \widehat{\phi} \; \Pi_{\sqbr{-\fc,\fc}} \, \mathcal{F}  x  ,
\end{align*}
where $\mathcal{F} f $ denotes the Fourier transform of a function or measure $f$ and $\Pi_{\sqbr{-\fc,\fc}}$ is an indicator function that is zero out of the interval $\sqbr{-\fc,\fc}$. For ease of exposition, we assume that the Fourier transform of the PSF is constant over $\sqbr{-\fc,\fc}$, i.e. the PSF is a periodized sinc or Dirichlet kernel, but our results hold for any PSF with a known low-pass spectrum. Since the support of $\xLR$ is restricted to the unit interval, it follows from the sampling theorem that its spectrum is completely determined by the discrete samples
\begin{equation} 
 \label{eq:fourier}
  y(k) = \mathcal{F}  \xLR (k)= \int_0^1 e^{-i2\pi kt} x(\text{d}t)  = \sum_j a_j e^{-i2\pi k
    t_j}, \quad k \in \Z, \, \abs{k}\leq \fc, 
\end{equation}
where we assume for simplicity that $\fc$ is an integer. In a more compact form, the sensing process can be represented as 
\begin{equation}
\label{eq:model_matrix_form}
y = \mathcal{F}_{n} \, x
\end{equation}
where $y\in \C^n$ and $\mathcal{F}_{n}$ is the linear operator that maps a measure or function to its lowest $n := 2\fc + 1$ Fourier coefficients.

If the signal $x$ is used to model a superposition of line spectra, equation~\eqref{eq:model_matrix_form} has a very natural interpretation: the data $y$ correspond to a finite number of samples of the signal in the time domain. As sketched in the lower half of Figure~\ref{fig:superres_diagram}, truncating the signal in the time domain is equivalent to convolving its spectrum with a sinc function. Our model can consequently be applied directly to \emph{spectral super-resolution}, where the aim is to estimate sparse line spectra from time-domain samples.

\subsection{Minimum separation}
\label{sec:minimum_separation}
In contrast to compressed sensing, where randomized measurements preserve the energy of arbitrary sparse signals with high probability (this is commonly known as the \emph{restricted isometry property}~\cite{candes2005decoding}), sparsity is not a strong enough prior to ensure that the super-resolution problem is well posed. Indeed, low-pass filtering may suppress sparse signals almost entirely if their support is too clustered together. As a result, in order to derive meaningful guarantees for super-resolution it is necessary to impose conditions on the signal support. To this end, we define the \emph{minimum separation} of the support of a signal, as introduced in~\cite{superres}.

\begin{definition}[Minimum separation] Let $\mathbb{T}$ be the circle obtained by identifying the endpoints on $[0,1]$. For a family of points $T \subset  \mathbb{T}$, the minimum separation (or minimum distance) is defined as the closest distance between any two elements from $T$,
  \begin{equation}
    \label{eq:min_distance}
    \Delta(T) = \inf_{(t, t') \in T \, : \, t \neq t'} \, \, |t - t'|. 
  \end{equation} 
where $|t - t'|$ is the $\ell_\infty$ distance (maximum deviation in any coordinate). To be clear, this is the wrap-around distance so that the distance between $t = 0$ and $t' = 3/4$ is equal to $1/4$.
\end{definition}

If the minimum distance is too small with respect to the cut-off frequency of the data, it may become impossible to estimate the signal even under very small perturbations to the data. A fundamental limit in this sense is $\lambda_c := 1 / \fc$, the inverse of the cut-off frequency, which also corresponds to the width of the main lobe of the point-spread function $\phi$. The reason is that for minimum separations just below $\lambda_c/ 2$ there exist signals that lie \emph{almost} in the null space of the low-pass operator defined by~\eqref{eq:model_matrix_form}\footnote{$\lambda_c / 2$ is the notorious Rayleigh resolution limit~\cite{dekker_survey}, below which it is challenging to even distinguish two neighboring sources.}. If such a signal $d$ corresponds to the difference between two different signals $s_1$ and $s_2$ so that $s_1 - s_2 = d$, it will be very challenging to distinguish $s_1$ and $s_2$ from low-resolution data. Figure~\ref{fig:superres_illposed} illustrates this: the measurements corresponding to two signals with disjoint supports and a minimum distance of $0.9 \lambda_c$ for $\fc = 10^{3}$ are indeed almost indistinguishable. The phenomenon can be characterized theoretically in an asymptotic setting using Slepian's prolate-spheroidal sequences~\cite{slepian} (see also Section 3.2 in \cite{superres}). More recently, Theorem 1.3 of~\cite{moitra_superres} provides a non-asymptotic analysis.  Finally, other works have obtained lower bounds on the minimum separation necessary for convex-programming approaches to succeed~\cite{tang_resolution,peyreduval}.


\begin{figure}
\hspace{1.5cm}
\begin{tabular}{
>{\centering\arraybackslash}m{0.09\linewidth} >{\centering\arraybackslash}m{0.28\linewidth}   >{\centering\arraybackslash}m{0.28\linewidth} }
  && $\: \quad$ Spectrum (magnitude) \vspace{0.2cm}\\
 Signals & 
\includegraphics{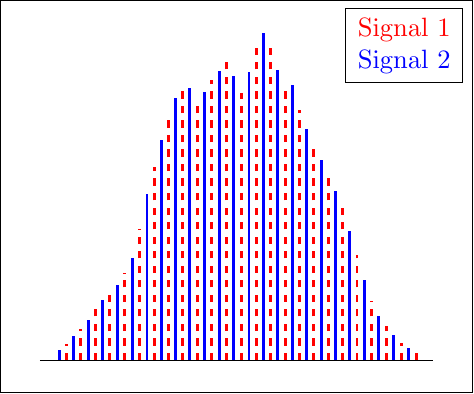}
&
\includegraphics{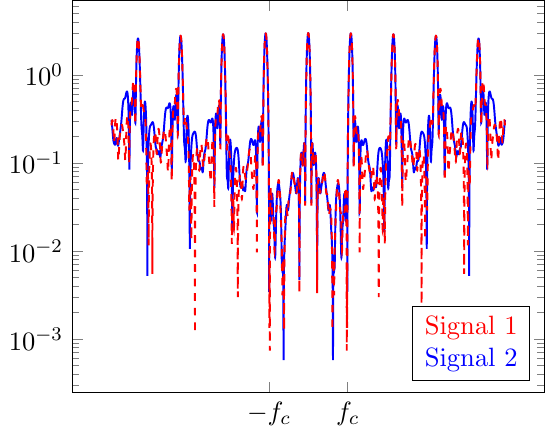}
\\
Difference & 
\includegraphics{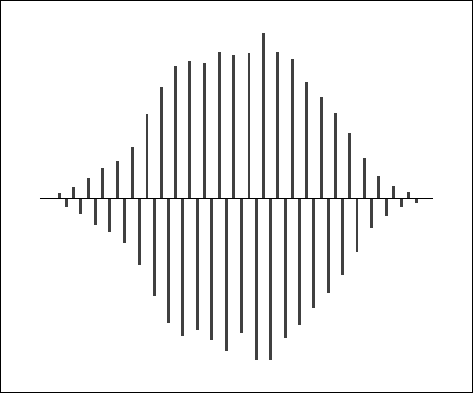}
&
\includegraphics{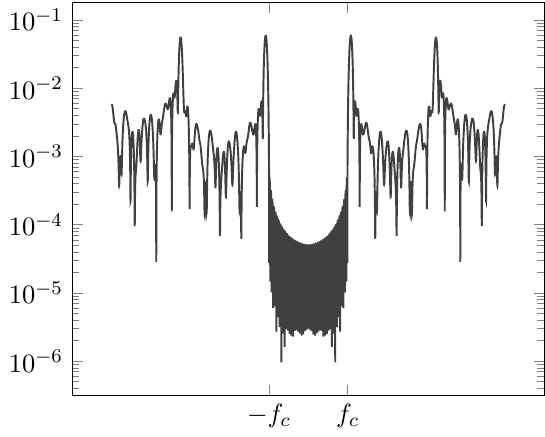}
\\
Data & 
\includegraphics{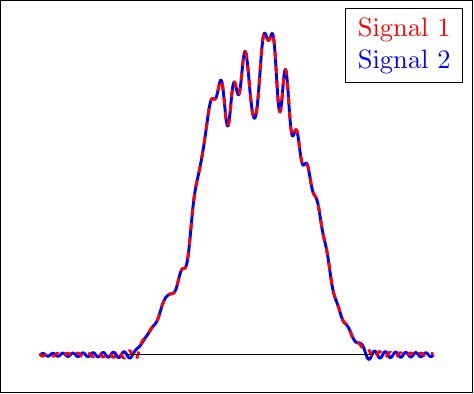}
&
\includegraphics{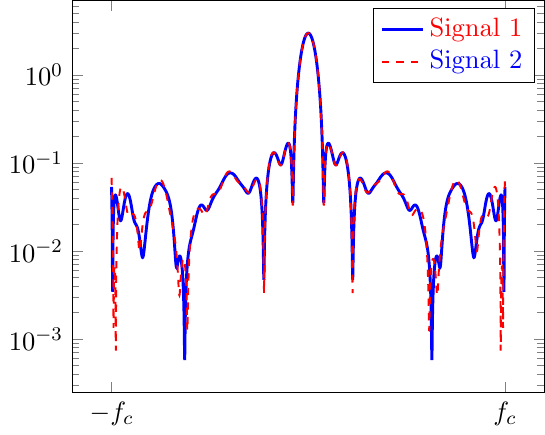}
\end{tabular}
\caption{Two signals with disjoint supports that satisfy the minimum-separation condition for $\Delta \brac{T} = 0.9 \fc$ when $\fc = 10^3$ (top left) and their spectrum (top right). Their difference (center left) has a spectrum that is concentrated away for the low-pass band between $-\fc$ and $\fc$ (center right). As a result, it is very difficult to the data corresponding to the two signals (bottom).}
\label{fig:superres_illposed}
\end{figure}

\subsection{Optimization-based super-resolution}
\label{sec:optimization_based}
Our approach to super-resolution is based on optimization: we estimate the signal by minimizing a sparsity-inducing norm. Since we are interested in point sources that may be supported at arbitrary locations within a continuous interval, we resort to a continuous counterpart of the $\ell_1$ norm known as the total-variation norm\footnote{The term \emph{total variation} may also refer to the $\ell_1$ norm of the discontinuities of a piecewise-constant function, which is a popular regularizer in image processing and other applications~\cite{tv}.}. If we consider the space of measures supported on the unit interval, this norm is dual to the infinity norm (see Section A in the appendix of~\cite{superres} for a different definition), so that for a measure $x$ we have
\begin{align*}
\normTV{x} = \sup_{\normInf{f}\leq 1,f \in C\brac{\mathbb{T}}} \operatorname{Re} \sqbr{  \int_{\mathbb{T}}\overline{f \brac{t}} x\brac{\text{d}t}}.
\end{align*}
For a superposition of Dirac deltas $\sum_{j}a_j\delta_{t_j}$, the total-variation norm is equal to the $\ell_1$ norm of the coefficients, i.e. $\normTV{x}=\sum_{j}\abs{a_j}$. Our super-resolution method consists in minimizing the total-variation norm of the estimate subject to data constraints, as proposed in~\cite{superres},
\begin{align}
\label{eq:TVnormMin}
\min_{\tilde x}\normTV{\tilde x} \quad \text{subject to} \quad
\mathcal{F}_{n} \, \tilde x = y, 
\end{align}
where the minimization is carried out over the set of all finite complex measures $\tilde x$ supported on $[0,1]$. Section~\ref{sec:implementation} discusses how to solve this optimization problem. 

\begin{figure}
\begin{tabular}{
>{\centering\arraybackslash}m{0.3\linewidth} >{\centering\arraybackslash}m{0.3\linewidth}   >{\centering\arraybackslash}m{0.3\linewidth} } 
$\fc = 30$  &  $\fc = 40$ & $ \fc = 50$ \\ 
\includegraphics{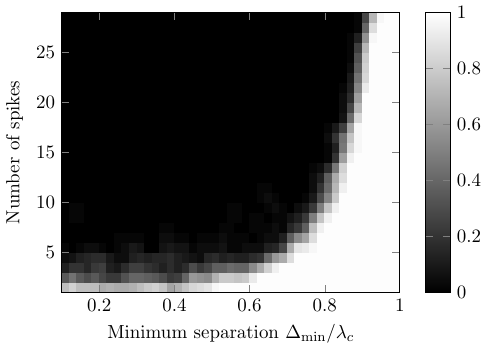} & \includegraphics{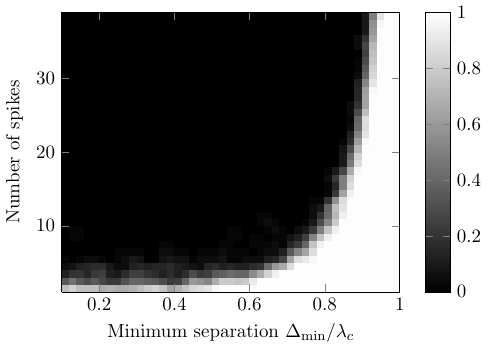} 
& \includegraphics{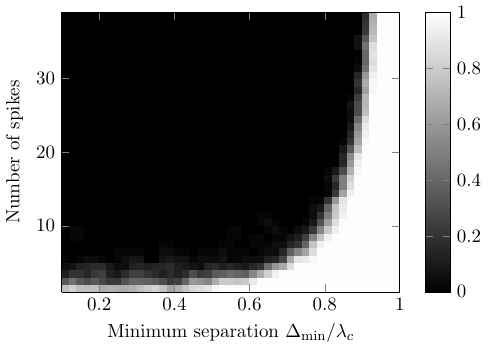} 
\end{tabular}
\caption{Graphs showing the fraction of times Problem~\eqref{eq:TVnormMin} achieves exact recovery over 10 trials with random signs and supports. A phase transition occurs near $\lambda_c$. The simulations are carried out using the high-precision semidefinite-programming solver CVX~\cite{cvx}.}
\label{fig:limits}
\end{figure}

In~\cite{superres} it was established that if the minimum separation of the support of a signal $\Delta(T)$ is greater or equal to $2\lambda_c$, TV-norm minimization achieves exact recovery. However, numerical experiments indicate that the actual limit at which super-resolution via TV-norm regularization may fail is $\lambda_c$. This is shown in Figure~\ref{fig:limits}; see also Section~5 of~\cite{superres}. As explained in Section~\ref{sec:minimum_separation}, $\lambda_c$ is a natural limit in the sense that below that minimum separation the problem may become ill posed. Our main result is that the guarantee for exact recovery can be extended to a minimum separation of just $\optvalue \lambda_c$. 
\begin{theorem}
  \label{theorem:noiseless} 
Let $T = \{t_j\}$ be the support of $x$. If the minimum separation obeys 
\begin{equation}
\label{eq:min-dist}
\Delta(T)  \geq {\optvalue}\, /{\fc} = {\optvalue} {\lambda_c}, 
\end{equation}
then $x$ is the unique solution to~\eqref{eq:TVnormMin}. This holds as long as $f_c \geq \minfc$. 
\end{theorem}
Since the signal is assumed to be supported on the unit interval, the result implies that it is possible to recover a number of point sources that is directly proportional to the cut-off frequency.
\subsection{Dual certificate}
\label{sec:dual_cert}
Theorem~\ref{theorem:noiseless} is a direct consequence of Proposition~\ref{prop:dualcert} below, which establishes the existence of a certain subgradient of the TV norm that is orthogonal to the null space of the measurement operator. Such an object is often referred to as a \emph{dual certificate} in the compressed-sensing literature~\cite{candesFreq} because it certifies that exact recovery occurs and its coefficients are a solution to the dual of Problem~\eqref{eq:TVnormMin}. 
\begin{proposition}
\label{prop:dualcert}
Under the conditions of Theorem~\ref{theorem:noiseless} for any sign pattern $v \in \C^{\abs{T}}$, such that $\abs{v_j}=1$ for all $j$, there exists a low-pass trigonometric polynomial
\begin{equation}
\label{eq:q_cond_lowfreq}
q(t) = \sum_{k = -\fc}^{\fc} c_k e^{i2\pi k t} 
\end{equation}
obeying 
\begin{align}
q(t_j) & = v_j, \qquad  t_j \in T, \label{eq:q_cond_interp}\\
|q(t)| & <1, \qquad  t_j \notin T ,\label{eq:q_far}
\end{align}
In addition, there exist numerical constants $C_0 \in \brac{0,1}$, $C_1$ and $C_2$ such that
\begin{align}
 1-C_1 \fc^2 \brac{t-t_j}^2 \leq  |q(t)| & \leq 1-C_2 \fc^2 \brac{t-t_j}^2 , \qquad \abs{t -t_j} \leq C_0 \lambda_c, \quad  t_j \in T. \label{eq:q_quadratic}
\end{align}

\end{proposition}
By~\eqref{eq:q_cond_interp} and~\eqref{eq:q_far}, the polynomial is a subgradient of the total-variation norm at the original signal $x$. By~\eqref{eq:q_cond_lowfreq} it is also low pass, which means that it is orthogonal to the null space of the measurement operator, as any signal in the null space is high pass. This immediately implies that for any signal $h$ in the null space 
\begin{align*}
\normTV{x+h} \geq \normTV{x} +\PROD{q}{h} = \normTV{x}.
\end{align*}
The bound on the off-support~\eqref{eq:q_far} actually implies that this inequality is strict, so that $x$ is the unique solution to Problem~\ref{eq:TVnormMin}. A complete proof of the fact that Theorem~\ref{theorem:noiseless} follows from Proposition~\ref{eq:q_cond_lowfreq} is provided in Section A of the appendix of~\cite{superres}. The quadratic bound~\eqref{eq:q_quadratic} is key in establishing robustness guarantees, see~\cite{robust_sr, support_detection}. Section~\ref{proof:noiseless} is devoted to proving Proposition~\ref{prop:dualcert}.

In the case of super-resolution, the dual certificate is a trigonometric polynomial with cut-off frequency $\fc$ that interpolates the sign of the original signal on its support and has magnitude strictly bounded by one on the off-support (see Proposition~\ref{prop:dualcert}). In~\cite{superres}, such a dual certificate is constructed using interpolation with a low-pass kernel and its derivative. The proof of Proposition~\ref{prop:dualcert} generalizes this approach, allowing to optimize the choice of the interpolation kernel by using sharp non-asymptotic bounds on the Dirichlet kernel; we defer the details to Section~\ref{proof:noiseless}.

Finally, we would like to emphasize that the structure of the dual polynomial reveals what signals will be more challenging for the optimization-based procedure. If the support of the signal is cluttered together and the sign of its coefficients varies rapidly, it may not be possible to achieve the interpolation with a bounded low-pass polynomial. In Sections~\ref{sec:dualcert_sinesspikes} and~\ref{sec:dual_cert_commonsupport} we show that similar insights arise when we derive dual certificates for other optimization programs designed to tackle extensions of the basic super-resolution problem.

\subsection{Robustness to noise}
\label{sec:robustness}
In any problem involving real data, it is necessary to account for perturbations and model imperfections. In the case of super-resolution, we can adapt Problem~\eqref{eq:TVnormMin} by using an inequality constraint to quantify the uncertainty,
\begin{align}
\label{eq:TVproblem_relaxed}
\min_{\tilde x} \,  \normTV{\tilde x} 
\quad \text{subject to} \quad \normTwo{\mathcal{F}_{n} \tilde x - y}^2 \leq \delta,
\end{align}
where $\delta$ is an estimate of the noise level. Alternatively, we could also consider a Lagrangian formulation of the form
\begin{align}
\label{eq:TVproblem_lag}
\min_{\tilde x} \,  \normTV{\tilde x} +  \gamma \normTwo{\mathcal{F}_{n} \tilde x - y}^2,
\end{align}
where the regularization parameter $\gamma>0$ governs the tradeoff between data fidelity and the sparsity of the estimate. 

Recent works~\cite{robust_sr, support_detection, azais2015spike, tang_minimax} derive non-asymptotic guarantees on the estimation error achieved when solving these problems to perform super-resolution from noisy data. The proofs of these stability guarantees rely in part on the dual certificate constructed in~\cite{superres} and on generalizations of this construction. As a result, the guarantees only hold under the proviso that the minimum separation is greater or equal to $2 \lambda_c$.

The techniques developed to prove Theorem~\ref{theorem:noiseless}, which are presented in Section~\ref{proof:noiseless}, allow to extend these results to minimum separations of just $\optvalue \lambda_c$. In more detail, Lemma~2.7 in~\cite{robust_sr} constructs a low-pass polynomial $q$ such that $q\brac{t_j}=0$ and $\abs{q'\brac{t_j}}=1$ for all $t_j$ belonging to the support of the original signal (i.e. the polynomial is locally linear). The polynomial is built through interpolation, using the same low-pass kernel as in~\cite{superres}. The construction can consequently be adapted by using the interpolation kernel described in Section~\ref{proof:noiseless}, along with the bounds provided in Section~\ref{sec:kernel}. The same holds for the results in~\cite{support_detection}. In this case, a low-pass polynomial that is equal to one on a certain element of the support and to zero on the rest is used to obtain support-detection guarantees. The polynomial is constructed in Lemma~2.2 of~\cite{support_detection}, again using the same interpolation kernel as in~\cite{superres}. The kernel and bounds described in the present work can be leveraged to build such a polynomial for supports with minimum separations above $\optvalue \lambda_c$. Finally, our techniques can also be used to sharpen the analysis in~\cite{venia_positive}, which studies the super-resolution of point sources with positive amplitudes.

\subsection{Extensions}
\label{sec:extensions}
Apart from the extensions discussed in Section~\ref{sec:convex_framework}, there are several interesting extensions to the basic super-resolution problem. In \textit{compressed sensing off the grid} the aim is to super-resolve a sparse signal from a random subset of its low-pass Fourier coefficients. Exact recovery via convex programming has been shown to occur with high probability for a number of measurements that is proportional to the sparsity level (up to logarithmic factors) as long as the support of the signal has a minimum separation of $2 \lambda_c$~\cite{cs_offgrid}.  This bound can be lowered to $\optvalue \lambda_c$ using the dual-certificate construction in Section~\ref{proof:noiseless}. Our construction also allows to sharpen results pertaining to the recovery of non-uniform splines from their projection onto spaces of algebraic polynomials~\cite{bendory_splines, decastro_splines}. Finally, it is straightforward to apply our results to the recovery of piecewise-constant or piecewise-smooth functions from low-pass data. The details can be found in Section~1.10 of~\cite{superres}. 

\subsection{Implementation}
\label{sec:implementation}
Solving Problem~\eqref{eq:TVnormMin} may seem challenging at first because its primal variable is infinite dimensional. This may be overcome by discretizing the unit interval into a grid and solving an $\ell_1$-norm minimization problem. However, it is also possible to solve the problem exactly without resorting to discretization. A strategy to achieve this is to recast the dual problem as a tractable semidefinite program and then decode the support of the primal variable from the dual solution. The dual problem is
\begin{align}
\label{eq:TV_normMin_dual}
\max_{c\in \C^{n}} \; \<y, c\>  \quad \text{subject to}
\quad \normInf{\mathcal{F}_{n}^{\ast} \, c} \leq 1,
\end{align}
where the inner product is defined as $\<y, c\> : = \operatorname{Re}\brac{y^{\ast}c}$. The dual variable $c$ is finite-dimensional but the constraint is infinite dimensional; the magnitude of the trigonometric polynomial $(\mathcal{F}_{n}^{\ast} \, c)(t) = \sum_{|k| \le \fc} c_k e^{i2\pi k t}$ must be bounded by $1$. Fortunately, this condition has a semidefinite representation provided by the following proposition, which is a consequence of the Fej\'er-Riesz Theorem (see Section~\ref{proof:sdp-charact} of the appendix for the proof and Theorem 4.24 in \cite{dumitrescu} for a more general result).
\begin{proposition}
\label{prop:sdp-charact}
Let $C \in \C^{n \times m}$ and let $C_k$ denote the $k$th column of $C$,
\begin{align*}
\sum_{k=1}^{m}\abs{ (\mathcal{F}_{n}^{\ast} \, C_k)(t) }^2 & \leq 1 \quad \text{for all } t \in \mathbb{T}
\end{align*}
if and only if there exists a Hermitian matrix $\Lambda \in \C^{n \times n}$ obeying
 \begin{equation}
\label{eq:sdp-charact}
   \MAT{\Lambda & C \\ C^{\ast} & \Id } \succeq 0, \qquad \mathcal{T}^{\ast}\brac{\Lambda}= e_1,
\end{equation}
where $e_1$ is the first vector of the canonical basis of $\R^{n}$.
\end{proposition}
For any vector $u$ such that $u_1$ is positive and real, $\mathcal{T}\brac{u}$ is a Hermitian Toeplitz matrix whose first row is equal to $u$. The adjoint of $\mathcal{T}$ with respect to the usual matrix inner product $\PROD{M_1}{M_2}=\text{Tr}\brac{M_1^{\ast}M_2}$, extracts the sums of the diagonal and off-diagonal elements of a matrix 
\begin{align*}
\mathcal{T}^{\ast}\brac{M}_j = \sum_{i=1}^{n-j+1}M_{i,i+j-1}.
\end{align*}
Setting $m=1$ in Proposition~\ref{prop:sdp-charact}, Problem~\eqref{eq:TV_normMin_dual} is equivalent to 
\begin{align}
\label{eq:TV_normMin_sdp}
\max_{c \in \C^n, \, \Lambda\in
\C^{n\times n}} \;  \<y, c \> \qquad
\text{subject to} \quad \MAT{\Lambda & c \\ c^{\ast} & 1} \succeq 0, \quad \mathcal{T}^{\ast}\brac{\Lambda}= e_1.
\end{align}
In order to extract an estimate of the support of the primal variable, we construct a \emph{support-locating} polynomial from the solution to~\eqref{eq:TV_normMin_sdp} $\hat{c}$,
\begin{align}
  P_{\hat{c}} \brac{ t } =  (\mathcal{F}_{n}^{\ast} \, \hat{c})(t). \label{polynomial_P}
\end{align}
By strong duality, which holds because the interior of the feasible set of Problem~\eqref{eq:TV_normMin_dual} contains the origin and is
consequently non-empty~\cite{rockafellar1974conjugate}, any solution $\hat{x}$ to \eqref{eq:TVnormMin} obeys
\begin{align*}
 \operatorname{Re} \sqbr{\int_0^1 \overline{ P_{\hat{c}}  \brac{t} }\, \hat{x}(\text{d}t)} & =\<\mathcal{F}_{n}^{\ast} \, \hat{c},x\> =\<\hat{c},\mathcal{F}_{n} \, x\>  =  \<y, \hat{c}\>  = \normTV{\hat{x}}, 
\end{align*}
which implies that $ P_{\hat{c}} $ is equal to the sign of the primal solution $\hat{x}$ at any point where the latter is non-zero. This suggests super-resolving the signal by applying the following scheme:
\begin{enumerate}
\item Solve the finite-dimensional semidefinite program~\eqref{eq:TV_normMin_sdp}.
\item Construct the support-locating polynomial and determine a set of points where its magnitude is equal to one to produce an estimate of the signal support.
\item Estimate the amplitude of the signal by solving the corresponding system of equations.
\end{enumerate}
Figure~\ref{fig:support_locating_pol} shows an example. This approach was proposed in~\cite{superres} and extended to a noisy setting in~\cite{robust_sr}. We refer the reader to Section~4 in~\cite{superres} and Section~3 in~\cite{robust_sr} for numerical simulations and a more detailed discussion (see also~\cite{atomic_norm_denoising} where a related semidefinite program is applied to the denoising of line spectra). Sections~\ref{sec:implementation_sinesspikes}
 and~\ref{sec:implementation_commonsupport} describe similar algorithms for demixing sines and spikes and super-resolving point sources with a common support. Finally, we would like to mention that other strategies to overcome discretization have been proposed~\cite{sparse_inverse_ben,radon_measures}.

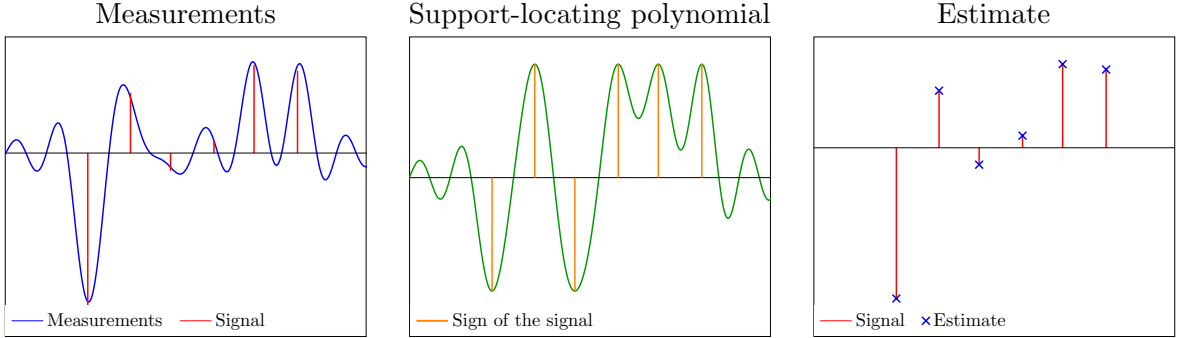
\begin{figure}[t]
\centering
\begin{tabular}{  >{\centering\arraybackslash}m{0.3\linewidth} >{\centering\arraybackslash}m{0.3\linewidth}  >{\centering\arraybackslash}m{0.3\linewidth}  }
Measurements & Support-locating polynomial & Estimate\\
\begin{tikzpicture}[scale=0.7]
\begin{axis}[xmin=0.2,xmax=0.5,ticks=none,legend entries={Measurements,Signal},legend style={font=\small,at={(0,0)},
anchor=south west,legend columns=-1,/tikz/every even column/.append style={column sep=0.25cm},draw=none}]
\addlegendimage{blue};
\addlegendimage{red};
\addplot[blue, thick] file {data_lowpass.dat};
\addplot+[ycomb,red,no marks,thick] file {data_signal.dat};
\addplot[black] coordinates
{(0.2,0) (0.9,0)};
\end{axis}
\end{tikzpicture}  
&
\begin{tikzpicture}[scale=0.7]
\begin{axis}[xmin=0.2,xmax=0.5,ticks=none,legend style={font=\small,at={(0,0)},anchor=south west,legend columns=-1,/tikz/every even column/.append style={column sep=0.25cm},draw=none},ymin=-1.4]
\addplot[black!40!green, thick,forget plot] file {data_dualpoly.dat};
\addplot[black,forget plot] coordinates {(0.2,0) (0.5,0)};
\addplot+[ycomb,orange,no marks,thick,forget plot] file {data_sign.dat};
\addlegendimage{orange,line width=1pt}
\addlegendentry{Sign of the signal}
\end{axis}
\end{tikzpicture}
&
\begin{tikzpicture}[scale=0.7]
\begin{axis}[xmin=0.2,xmax=0.5,ticks=none,
legend entries={Signal,Estimate},legend style={font=\small,at={(0,0)},
anchor=south west,legend columns=-1,/tikz/every even column/.append style={column sep=0.25cm},draw=none},ymin=-38]
\addlegendimage{red}
\addlegendimage{only marks,mark=x,blue,mark size=3pt, thick}
\addplot+[ycomb,red,no marks,thick] file {data_signal.dat};
\addplot[black] coordinates {(0.2,0) (0.5,0)};
\addplot+[ycomb,blue,only marks,mark=x,mark size=3pt, thick] file {data_est.dat};
\end{axis}
\end{tikzpicture}
\end{tabular}
\caption{On the left we have the low-pass measurements (blue) corresponding to a superposition of spikes (red). At the center, we see how the support-locating polynomial obtained from solving Problem~\eqref{eq:TV_normMin_sdp} interpolates the sign of the signal very precisely. This allows to obtain an extremely accurate estimate of the signal, as shown on the right.}
\label{fig:support_locating_pol}
\end{figure}

\subsection{Related work}
\label{sec:related_work}
A popular approach to the estimation of point sources from bandlimited data is to fit an estimate of the point-spread function of the sensing device to each source individually; see~\cite{fpalm} for an example in fluorescence microscopy. This is essentially equivalent to matched-filtering techniques used in communications and radar, as well as linear nonparametric spectral-analysis algorithms such as the periodogram~\cite{stoica_book}. These simple methods perform well as long as the dynamic range of the signal is not very high and the sources are well separated. However, if this is not the case, their performance can be degraded due to aliasing, even in the absence of noise. 

Prony's method is a classical technique that exploits the algebraic structure of the data to tackle the problem of spectral super-resolution~\cite{prony}. The method is able to super-resolve the support of the sparse spectrum of a signal from measurements following~\eqref{eq:fourier} by building a polynomial whose zeros are located exactly on the support. This motivates the use of the term \emph{annihilating filter} in the finite rate of innovation (FRI) framework~\cite{fri,blu_fri,fri_moments}. The FRI framework allows to super-resolve signals with a parametric representation that has a finite number of degrees of freedom. The superpositions of Dirac measures that we consider in spectral super-resolution belong to this class. For this type of signal, FRI reduces exactly to Prony's method (see Section~3 in~\cite{fri} or Section~1 in~\cite{blu_fri}). We would like to mention that algebraic techniques have also been applied to the location of singularities in the reconstruction of piecewise polynomial
functions from a finite number of Fourier coefficients (see
\cite{banerjee_algebraic,batenkov_algebraic,eckhoff_algebraic} and
references therein). The theoretical analysis of these methods proves
their accuracy up to a certain limit related to the number of
measurements. 

In the absence of noise, Prony's method is capable of super-resolving $s$ spikes from $2s$ measurements without any minimum-separation condition. However this is no longer the case when the data are noisy, as small perturbations may produce substantial changes in the position of the zeros of the annihilating polynomial. A popular approach to overcome this issue is to estimate the location of the sources from the eigendecomposition of the empirical covariance matrix of the data. Such techniques, which include MUSIC~\cite{music1,music2} and ESPRIT~\cite{esprit}, have widespread popularity in the signal-processing community. We refer the reader to~\cite{stoica_book} for a detailed description. A variant that is often used by FRI methods fits a low-rank Toeplitz model to a matrix constructed from the data, see Section 2 of~\cite{blu_fri}. This technique is due to Tufts and Kumaresan~\cite{tufts_kumaresan} and is also known as Cadzow denoising~\cite{cadzow}.

Most theoretical work analyzing the robustness of covariance-based methods for spectral super-resolution is based on the asymptotic characterization of the sample covariance matrix under Gaussian noise~\cite{stoica_statistical,clergeot_music}. However, some recent stability analysis has shown that these methods are robust for signals with a minimum separation above $\lambda_c$~\cite{moitra_superres}, see also~\cite{demanet_superres, music_liao}. In addition, \cite{music_liao} shows that numerically the noise level tolerated by MUSIC for small signal separations obeys a power law. In practice, the performance of classical spectral methods such as MUSIC, Cadzow's method or ESPRIT seems to be comparable to the optimization-based approach which is the focus of the present paper, see~\cite{atomic_norm_denoising} for some numerical comparisons.

The idea of applying $\ell_1$-norm minimization and related optimization programs to the super-resolution problem dates back to Beurling~\cite{beurling_1, beurling_2} and researchers in geophysics from the 1970s~\cite{claerbout,levy,santosa}. Theoretical guarantees for positive signals in the absence of noise were obtained in~\cite{fuchs_positive,donoho_positive,supportPursuit} (see also~\cite{pos_gaussian_sr} for an interesting recent result). As explained previously in this section, \cite{superres} proved that exact recovery via convex optimization for measures with arbitrary complex amplitudes satisfying a minimum separation of $2 \lambda_c$ and provided stability guarantees for a discretized version of the problem. Subsequent works then established stability guarantees in a continuous setting~\cite{robust_sr, support_detection, azais2015spike, tang_minimax,peyreduval} and analyzed related problems using similar techniques~\cite{sr_radar, spherical_harmonics_sr,stft_superres}. Finally, we would like to mention an application of super-resolution via convex programming to fluorescence microscopy~~\cite{storm_superres}, another optimization-based approach to super-resolve point sources on a continuous domain~\cite{cbp} and an alternative proof of exact recovery via TV-norm minimization~\cite{thesis_mahdi} that allows to sharpen the result in~\cite{superres} to $1.56 \lambda_c$.
\section{A general framework for super-resolution}
\label{sec:convex_framework}
In practical applications it is often desirable to adapt the signal, noise or measurement model to incorporate prior information. Optimization-based methods are well suited for this purpose: tailoring the cost function and the constraints of the optimization program allows to account for application-specific structure. We illustrate this by deriving optimization-based algorithms to demix sines and spikes and to super-resolve multiple signals with a common support. In both cases we provide an implementation based on semidefinite programming, derive the corresponding dual certificate and report some experimental results.

\subsection{Super-resolution from corrupted data: Demixing of sines and spikes}
\label{sec:sines_spikes}

\begin{figure}[t]
\centering
\begin{tabular}{  >{\centering\arraybackslash}m{0.08\linewidth} >{\centering\arraybackslash}m{0.35\linewidth} >{\centering\arraybackslash}m{0.35\linewidth}  }
 \vspace{0.1cm}& \vspace{0.1cm}& Spectrum\vspace{0.1cm} \\
Sines &
\includegraphics{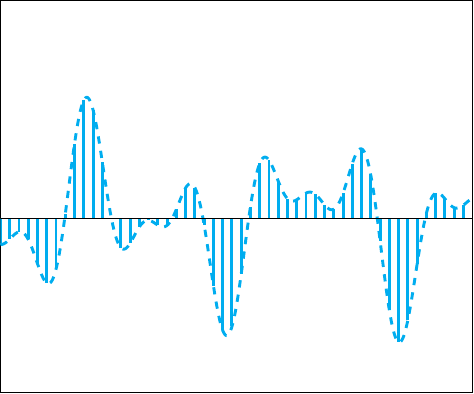}
& 
\includegraphics{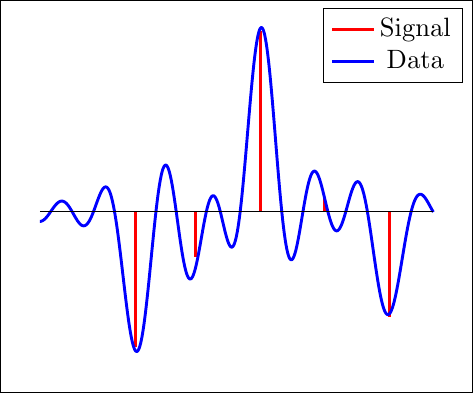}
\\
Spikes &
\includegraphics{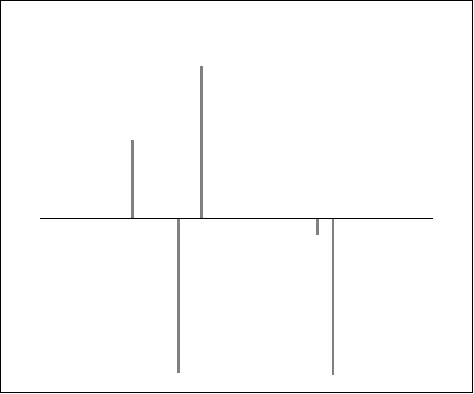}
&
\includegraphics{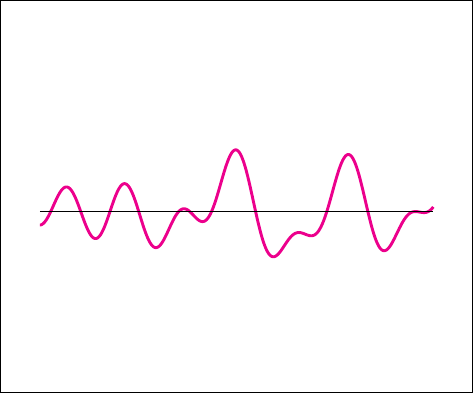}
\\
Sines \hspace{2cm} 
+ 
Spikes 
& 
\includegraphics{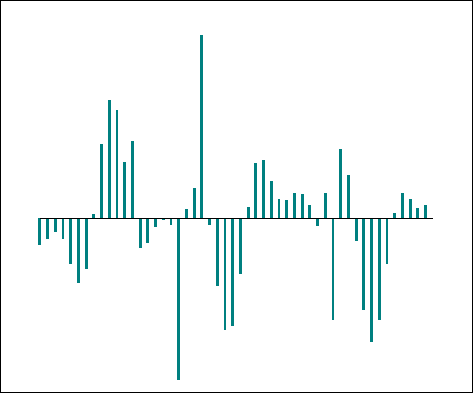}
&
\includegraphics{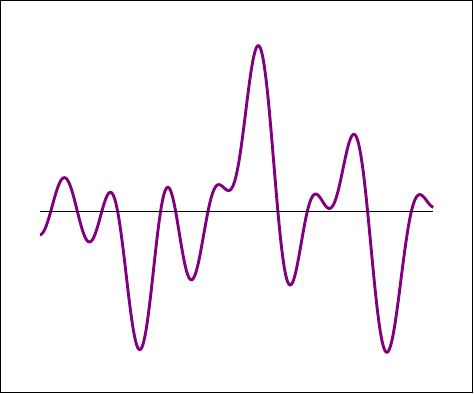}
\end{tabular}
\caption{Mixture of sines and spikes: samples from a multi-sinusoidal signal (top) are combined with impulsive events (middle) to yield the data (bottom). }
\label{fig:sines_spikes}
\end{figure}

As mentioned previously, the super-resolution problem discussed in Section~\ref{sec:noiseless} can be used to model line-spectra estimation. In that case, the measurements correspond to equispaced samples from a multi-sinusoidal signal measured at the Nyquist rate over a limited period of time. We now consider line-spectra estimation when some of the data may be completely corrupted. This could be due to the presence of sparse noise, to the intermittent failure of a sensor or to some other impulsive phenomenon. We model the corruptions as a vector $s \in \C^n$ with support $S$. In the resulting measurement model, the data $y$ are of the form
\begin{align}
 y = \mathcal{F}_{n} \, x + s. \label{eq:sinesspikes_data}
\end{align}
As shown in Figure~\ref{fig:sines_spikes}, impulsive corruptions produce additional aliasing that further masks the line-spectra locations when we visualize the data in the frequency domain. We aim to \emph{demix} the two components: the \emph{sines}, whose spectrum is denoted by $x$, and the \emph{spikes} $s$.

\subsubsection{Estimation via convex programming}
Convex programming has proven very effective in separating incoherent low-dimensional components ~\cite{chandrasekaran2011rank,candes2011robust,li2013compressed,demixing_tropp}. Inspired by these works, we design a simple convex program that allows to achieve effective demixing \emph{without any previous knowledge of the number of spikes or sines}. To separate the two components we penalize sparsity-inducing norms in both domains: the $\ell_1$ norm of the spikes and the TV norm the spectrum of the sines, which is modeled as an atomic measure. These two norms are combined additively. More explicitly, the estimate is computed by solving 
\begin{align}
\label{eq:primal_sinesspikes}
\min_{\tilde{x},\, \tilde{s}\in \C^n} \;  \normTV{\tilde{x}} + \eta \normOne{\tilde{s}} \quad \text{subject to} \quad
\mathcal{F}_{n} \, \tilde{x} + \tilde{s} = y, 
\end{align}
where $\eta > 0 $ is a real-valued parameter that determines the tradeoff between the two terms in the cost function. This formulation was previously proposed in~\cite{tang2014robust}.

\subsubsection{Implementation}
\label{sec:implementation_sinesspikes}
As in the case of our basic model (see Section~\ref{sec:implementation}), we have two options to solve Problem~\eqref{eq:primal_sinesspikes}. The first is to discretize the support of the primal variable that represents the line-spectra estimate and then solve an $\ell_1$-norm minimization problem. The second is to consider the dual of Problem~\eqref{eq:primal_sinesspikes}, provided by the following lemma (see Section~\ref{proof:dual_sinesspikes} of the appendix for the proof).
\begin{lemma}
\label{lemma:dual_sinesspikes}
The dual of Problem~\eqref{eq:primal_sinesspikes} is
\begin{align}
\label{eq:dual_sinesspikes}
\max_{c\in \C^{n}} \;   \<y, c\>  \quad \text{subject to}
\quad \normInf{\mathcal{F}_{n}^{\ast} \, c}  \leq 1 \quad \text{and} \quad \normInf{c}  \leq \eta.
\end{align}
\end{lemma}
Problem~\ref{eq:dual_sinesspikes} can be recast as the semidefinite program 
\begin{align}
\label{eq:TVnormMin_sdp_sinesspikes}
\max_{c \in \C^n, \, \Lambda\in
\C^{n\times n}} \;   \<y, c \> \quad
\text{subject to} \quad  \MAT{\Lambda & c \\ c^{\ast} & 1} \succeq 0, \quad \mathcal{T}^{\ast}\brac{\Lambda}= e_1 \quad \text{and} \quad \normInf{c}  \leq \eta,
\end{align}
by invoking Proposition~\ref{prop:sdp-charact}. After solving~\eqref{eq:TVnormMin_sdp_sinesspikes} we still need to decode the support of the primal solution from the dual solution. This can be done by leveraging the following lemma, proved in Section~\ref{proof:primaldual_sinesspikes} of the appendix.
\begin{lemma}
\label{lemma:primaldual_sinesspikes}
Let 
\begin{align*}
\hat{x}=\sum_{t_j \in \widehat{T}} \hat{a}_{j} \delta_{t_j}=\sum_{t_j \in \widehat{T}} \abs{\hat{a}_{j}}{e^{i\phi_j}} \delta_{t_j} \qquad \text{and} \qquad \hat{s}_l = \abs{\hat{s}_l}e^{i \psi_l}, \;  l \in \widehat{\Set} \subseteq \keys{1,\dots,n},
\end{align*}
be a solution to \eqref{eq:primal_sinesspikes}, such that $\widehat{T}$ and $\widehat{\Set}$ are the nonzero supports of the line spectra $\hat{x}$ and the spikes $\hat{s}$ respectively. If $\hat{c} \in \C^n$ is a corresponding dual solution, then for any $t_j$ in $\widehat{T}$ 
\begin{align*}
 \brac{\mathcal{F}_{n}^{\ast} \, \hat{c}} \brac{t_j} =e^{i\phi_j}
\end{align*}
and for any $l$, $1\leq l \leq n$ in $\widehat{\Set}$
\begin{align*}
\hat{c}_l = \eta \, e^{i \psi_l}.
\end{align*}
\end{lemma}
\begin{figure}[t]
\centering
\begin{tabular}{
>{\centering\arraybackslash}m{0.12\linewidth}>{\centering\arraybackslash}m{0.35\linewidth} >{\centering\arraybackslash}m{0.35\linewidth} }
 \vspace{0.1cm}&  $\mathcal{F}_c^{\ast} \, \hat{c}$  \vspace{0.1cm} & $\hat{c}$ \vspace{0.1cm}\\
Dual solution &
\includegraphics{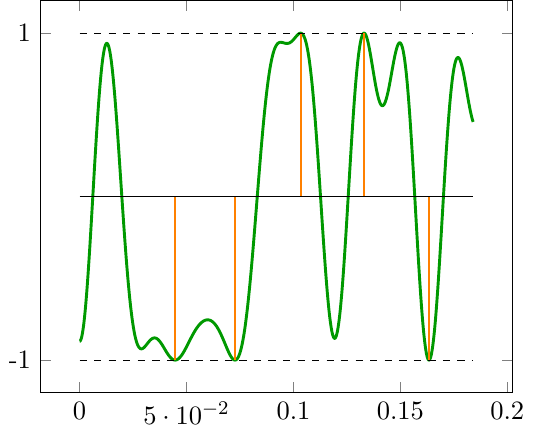}
& 
\includegraphics{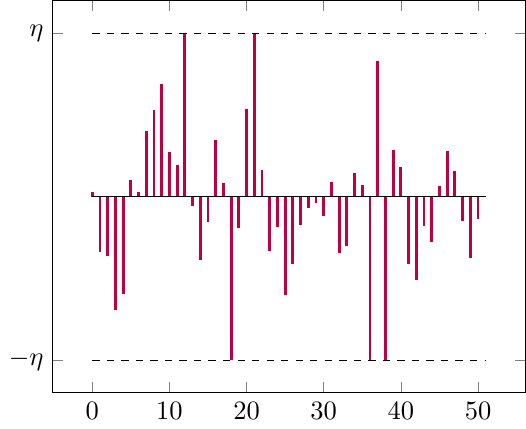}
\\
Estimate
&
\includegraphics{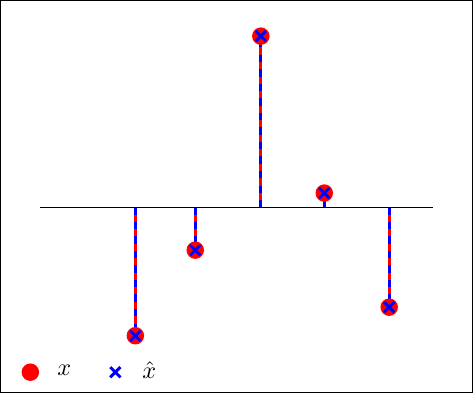}
&
\includegraphics{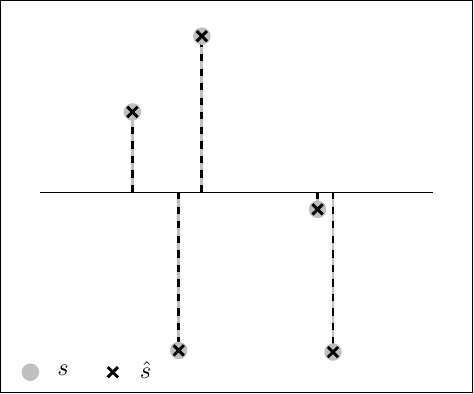}
\vspace{0.01cm}
\end{tabular}
\caption{Demixing of the signal in Figure~\ref{fig:sines_spikes} by semidefinite programming. Top left: the polynomial $\mathcal{F}_c^{\ast} \, \hat{c}$ (green) corresponding to the solution of~\eqref{eq:TVnormMin_sdp_sinesspikes} $\hat{c}$ interpolates the sign of the line spectra of the sines (orange) on their support. Top right: $\eta^{-1} \hat{c}$ interpolates the sign pattern of the spikes on their support. Bottom: locating the support of $x$ and $s$ allows to demix to very high precision.}
\label{fig:sines_spikes_sdp}
\end{figure}

In words, the weighted dual solution $\eta^{-1} \hat{c}$ and the corresponding polynomial $\mathcal{F}_{n}^{\ast} \, \hat{c}$ interpolate the sign patterns of the primal-solution components $\hat{x}$ and $\hat{s}$ on their respective supports. This suggests the following demixing scheme:
\begin{enumerate}
\item Solve the finite-dimensional semidefinite program~\eqref{eq:TVnormMin_sdp_sinesspikes} to obtain a dual solution $\hat{c}$.
\item The support of $\hat{s}$ is given by the points at which $\hat{c}$ equals $\eta$.
\item Construct the support-locating polynomial $\mathcal{F}_c^{\ast} \, \hat{c}$. The support of $\hat{x}$ is given by the points at which its magnitude equals one.
\item Estimate the amplitudes of $\hat{x}$ and $\hat{s}$ by solving the corresponding system of equations.
\end{enumerate}
The resulting $\hat{x}$ and  $\hat{s}$ are our estimates of the line spectra and of the spikes respectively. Figure~\ref{fig:sines_spikes_sdp} shows the results obtained by this method on the data described in Figure~\ref{fig:sines_spikes}: both components are recovered very accurately.~\footnote{A Matlab script reproducing this example is available at \url{http://www.cims.nyu.edu/~cfgranda/scripts/sines_and_spikes.m}.}

\subsubsection{Theoretical analysis: Dual certificate}
\label{sec:dualcert_sinesspikes}
The success or failure of the demixing procedure is determined by the structure of the signal that we aim to recover. As in the case of our basic model, this is reflected by the dual certificate corresponding to Problem~\eqref{eq:primal_sinesspikes}.
\begin{lemma}
\label{lemma:cert_sinesspikes}
Let $y = \mathcal{F}_n x + s$, where 
\begin{align*}
x=\sum_{t_j \in T} a_{j} \delta_{t_j}=\sum_{t_j \in T} \abs{a_{j}}{e^{i\phi_j}} \delta_{t_j}, \quad \text{and} \quad s_l = \abs{w_l}e^{i \psi_l}
\end{align*}
for $s \in \C^n$ with support $\Set \subseteq \keys{1,\dots,n}$. If $\abs{T} + \abs{\Set} \leq n $ and there exists a low-pass polynomial $q: \mathbb{T} \rightarrow \C$
\begin{align}
q (t) =\mathcal{F}_n^{\ast}c = \sum_{l = -\fc}^{\fc} c_{l} e^{i2\pi lt}, \label{eq:opt_lowpass}
\end{align}
where $c \in \C^n$, such that 
\begin{alignat}{2}
  & q(t_j)  = e^{i \phi_j} \quad  t_j \in T, 
 \qquad && \abs{ q(t) }   < 1\quad  t  \notin T,  \label{eq:opt1} \\
 & c_l  =  \regpar \; e^{i \psi_l}\quad  l \in \Set,  
\quad &&  \abs{c_j }   < \regpar \quad  j \notin \Set,  \label{eq:opt2}
\end{alignat}
then $x$ and $s$ are the unique solutions to Problem~\eqref{eq:primal_sinesspikes}.
\end{lemma}
The lemma, proved in Section~\ref{proof:cert_sinesspikes} of the appendix, shows that the dual certificate is a low-pass polynomial which interpolates the sign pattern of the line spectra, just like the dual certificate corresponding to Problem~\eqref{eq:TVnormMin}. As a result, the approach might fail for line spectra with very clustered supports, which makes sense as such problems are challenging even in the absence of corruptions. The dual certificate satisfies two additional conditions that depend on the spikes/corruptions: its coefficients interpolate the sign pattern of the spikes on their support and their magnitude is bounded by $\eta$. If there are too many spikes it may become impossible to build a bounded low-pass polynomial under these constraints that also interpolates the sign of the line spectra. There is consequently a tradeoff between the number of sines and spikes that may be recovered simultaneously. In the following section we provide some experimental results that illustrate this tradeoff. Characterizing it theoretically is an interesting research direction. 

\begin{figure}[t]
\centering
\begin{tabular}{  >{\centering\arraybackslash}m{0.3\linewidth} >{\centering\arraybackslash}m{0.3\linewidth}  >{\centering\arraybackslash}m{0.3\linewidth}  }
 $\Deltamin=1$ &  10 spikes & $\abs{T} = 15$\\
\includegraphics{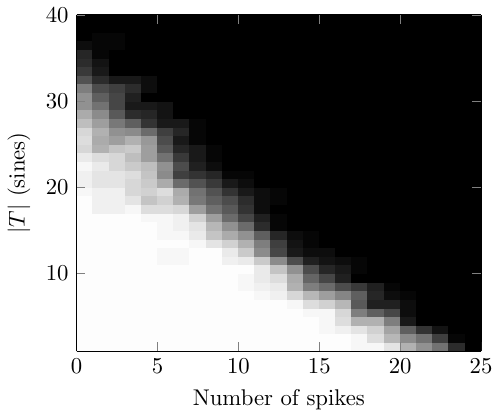}
 &  
\includegraphics{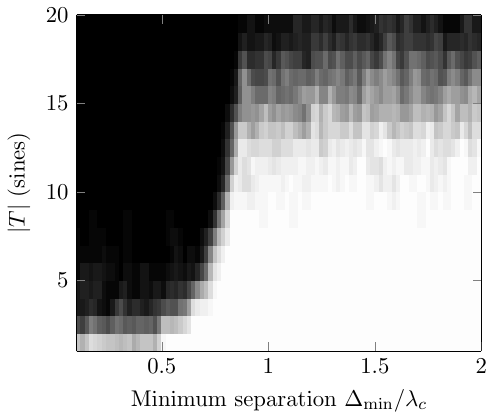}
&
\includegraphics{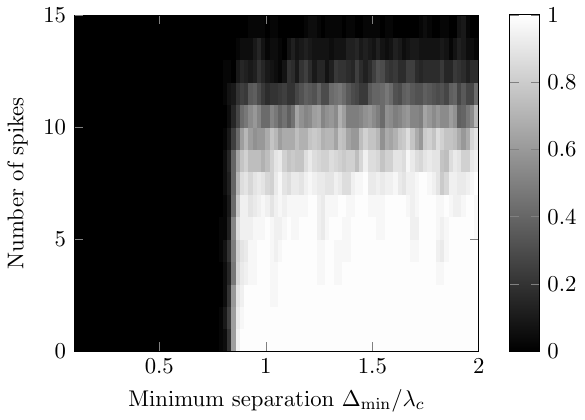}
\end{tabular}
\caption{Graphs showing the fraction of times Problem~\eqref{eq:primal_sinesspikes} achieves exact recovery over 10 trials with random signs and supports for different values of the minimum separation of the line spectra $x$ (left), as well as different cardinalities of the support of the line spectra $x$ denoted by $\abs{T}$ (center) and of the spikes $s$ (right). The simulations are carried out using CVX~\cite{cvx}.}
\label{fig:experiments_sines_spikes}
\end{figure}

\subsubsection{Numerical results}
Figure~\ref{fig:experiments_sines_spikes} shows the result of applying our demixing procedure to signals with different minimum separations of the line spectra $x$, as well as different cardinalities of the supports of $x$ and $s$. In each case we observe a rapid transition between a regime in which exact recovery almost always occurs and a regime in which the method fails. For a fixed minimum separation, the procedure succeeds for a number of spikes which decreases as we increase the number of sine components. In contrast, if we fix the cardinality of $x$, the phase transition occurs at a value of the minimum separation that does not seem to depend on the number of spikes, as long as there are not too many. Similarly, for a fixed number of spikes the minimum separation necessary to achieve exact recovery seems to be independent of the cardinality of $x$ below a certain limit.

\subsection{Point sources with a common support}

In this section we consider the problem of super-resolving $m$ signals consisting of superpositions of point sources, which share a common support. This is a model of interest in fluorescence microscopy, where the data are often well modeled as low-resolution frames where fluorescence probes flicker stochastically~\cite{palm,fpalm, storm}, and in sparse-channel estimation for wireless communications~\cite{barbotin_commonsupport}. 

The signals are modeled as superpositions of Dirac measures with support $T$
\begin{equation}
  \label{eq:model_common_support}
  X_k(t) = \sum_{t_j \in T} A_{jk} \delta_{t_j}, \quad  1\leq k \leq m,
\end{equation}
where $A \in \C^{ \abs{T} \times m}$. The low-resolution measurements consist of the $n = 2\fc + 1$ first Fourier coefficients of each signal. We represent these low-pass data with cut-off frequency $\fc$ as a matrix $Y\in \C^{n\times m}$,
\begin{align}
  Y_{lk} = \int_0^1 e^{-i2\pi lt} X_k(\text{d}t) = \sum_{j=1}^{s} A_{jk} e^{-i2\pi l  t_j} = \brac{ \mathcal{F}_{n} \, X_k }_l , \quad  \abs{l} \leq \fc, \;1\leq k \leq m. \label{eq:measurements_commonsupport}
\end{align}
To be clear, the $k$th column of $Y$ contains the low-pass Fourier coefficients of the signal $ X_k$. 

\subsubsection{Estimation via convex programming}
Note that we could just apply the approach based on TV-norm minimization described in Section~\ref{sec:superres_point_sources} to super-resolve each signal separately. However, our goal is to exploit the diversity of the amplitudes $A_{jk}$ to obtain a better estimate of the common support $T$ by processing all the data simultaneously. This can be achieved by leveraging a generalization of the TV norm, which we call \emph{group total-variation norm} (gTV). Formally, we define the gTV norm as
\begin{align*}
\normgTV{X} := \sup_{F:\mathbb{T}\rightarrow \C^{m},\normTwo{F\brac{t}} \leq 1 \; t\in \mathbb{T}} \sum_{k=1}^{m} \operatorname{Re}\keys{ \int_{\mathbb{T}} \overline{F_k \brac{t}} X_k\brac{\text{d}t} }.
\end{align*}
We would like to stress that the gTV norm is a well-known quantity. In fact, it is commonly used as a definition for the total variation of vector-valued signals~\cite{ambrosio2000functions}. Our intent is \emph{not} to propose a new name for it, but rather to use an abbreviation that distinguishes it from the TV norm while highlighting the connection to group sparsity. Intuitively, the gTV norm is the continuous counterpart of the mixed $\ell_1/\ell_2$ norm, which is used to promote group sparsity~\cite{group_lasso, mmv_tropp}. For an arbitrary matrix $M$
\begin{align*}
\normOneTwo{M} := \sum_j \normTwo{M_{j:}},
\end{align*}
where $M_{j:}$ denotes its $j$th row. If the measures $X_k$ are superpositions of point sources then $\normgTV{X}  = \normOneTwo{A}$. 

Let us provide some intuition as to why the $\ell_1 / \ell_2$ norm promotes joint sparsity. Consider a matrix $M$ which has both zero and nonzero elements. If we choose an element $M_{jk}$ from an empty row $M_{j:}$ and set its value to $\epsilon > 0$, then $\normOneTwo{M}$ increases by $\epsilon$. However, if the row is not empty and $\normTwo{M_{j:}} > 0$, then $\normOneTwo{M}$ increases by just $\epsilon^2/2\normTwo{M_{j:}}$ as long as $\epsilon \ll \normTwo{M_{j:}}$ (this follows from a Taylor expansion around $\epsilon=0$ of $(\normTwo{M_{j:}}^2 + \epsilon^2)^{1/2}$). It is consequently much cheaper to introduce new elements in non-empty rows than in empty ones, which promotes a common sparsity pattern in the columns. 

To super-resolve the signals jointly, we propose minimizing the continuous version of the $\ell_1/\ell_2$ norm subject to data constraints
\begin{align}
\label{eq:gTVnormMin}
\min_{\widetilde{X}} \;  \normgTV{\widetilde{X}} \quad \text{subject to} \quad
 \MAT{\mathcal{F}_{n} \, \widetilde{X}_1 & \mathcal{F}_{n} \, \widetilde{X}_2 & \cdots & \mathcal{F}_{n} \, \widetilde{X}_m } = Y,   
\end{align}
This cost function favors solutions that have a sparse common support, allowing to jointly super-resolve signals that are very challenging to super-resolve individually, as illustrated in Figure~\ref{fig:mm_polynomials}. We would like to note that this approach is equivalent to the one proposed by two recent preprints~\cite{mm_superres_1, mm_superres_2} from the perspective of atomic-norm minimization. 

\begin{figure}[t!]
\centering
\begin{tabular}{  >{\centering\arraybackslash}m{0.29\linewidth} >{\centering\arraybackslash}m{0.29\linewidth} 
>{\centering\arraybackslash}m{0.29\linewidth}  }
\includegraphics{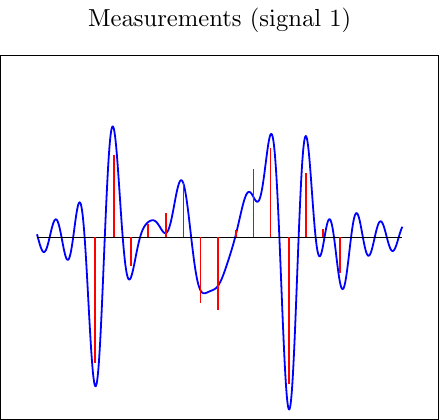}
&
\includegraphics{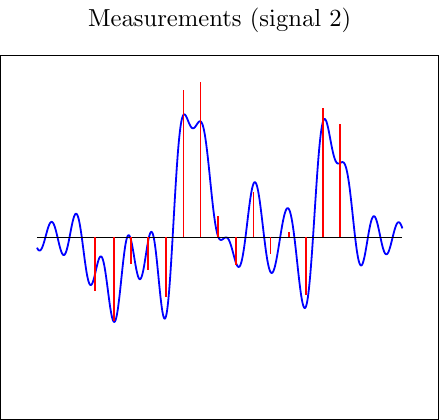}
&
\includegraphics{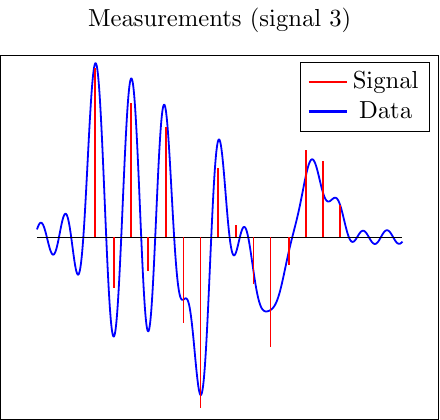}
\\
\includegraphics{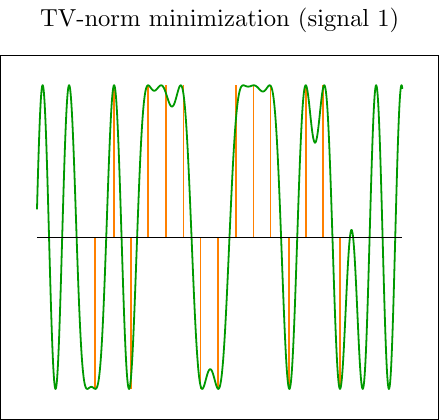}
&
\includegraphics{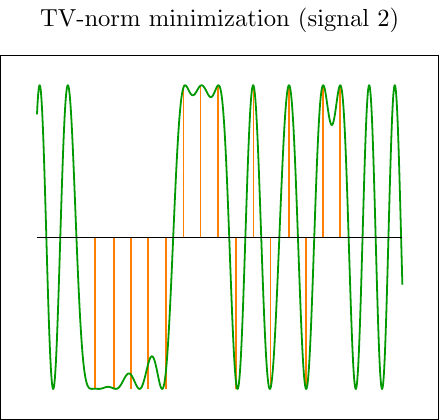}  
&
\includegraphics{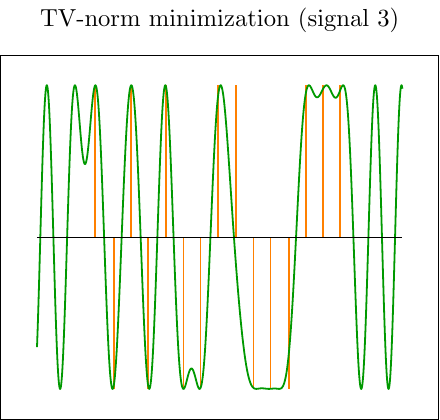}
\end{tabular}

\begin{tabular}{  >{\centering\arraybackslash}m{0.4\linewidth} >{\centering\arraybackslash}m{0.4\linewidth}  }
\includegraphics{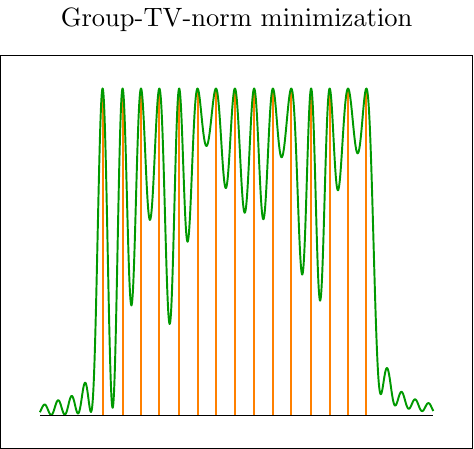}
&  
\includegraphics{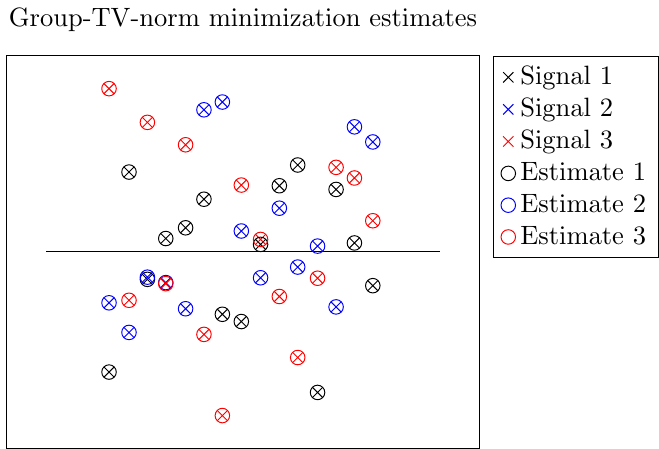}
\end{tabular}
\caption{Top row: Signals and corresponding low-pass data. Middle row: Support-locating polynomials obtained via TV-norm minimization. Bottom row: Support-locating polynomial obtained by gTV-norm minimization (left) and the resulting amplitude estimates (right, the three signals and their estimates are plotted simultaneously).}
\label{fig:mm_polynomials}
\end{figure}
\subsubsection{Implementation}
\label{sec:implementation_commonsupport}
If we discretize the domain, Problem~\eqref{eq:gTVnormMin} reduces to an $\ell_1/\ell_2$-norm minimization problem. However, as in the case of Problems~\eqref{eq:TVnormMin} and~\eqref{eq:primal_sinesspikes}, it is possible to obtain a solution over the continuous domain directly. 

\begin{lemma}
\label{lemma:dual_commonsupport}
The dual of Problem~\eqref{eq:gTVnormMin} is
\begin{align}
\label{eq:gTVnormMin_dual}
\max_{C \in \C^{n \times m}} \;  \<Y, C\>  \quad \text{subject to}
\quad \sup_{t} \sum_{k=1}^{m} \abs{ \brac{\mathcal{F}_{n}^{\ast} \, C_k} (t)}^2 \leq 1,
\end{align}
where $C_k$ denotes column $k$ of $C$. 
\end{lemma}
The proof of the lemma can be found in Section~\ref{proof:dual_commonsupport} of the appendix. By Proposition~\ref{prop:sdp-charact}, we can recast~\eqref{eq:gTVnormMin_dual} as the semidefinite program
\begin{align}
\label{eq:gTV_normMin_sdp}
 \max_{C \in \C^{n \times m}, \, \Lambda \in \C^{n \times n}} \;  \<Y, C\>  \quad \text{subject to} \quad \MAT{\Lambda & C\\ C^{\ast} & \Id_n} \succeq 0 , \quad \mathcal{T}^{\ast}\brac{\Lambda}= e_1,
\end{align}
where $\Id_n$ is the identity matrix of dimensions $n \times n$.

To recover the support of the primal solution from a dual solution $\widehat{C}$ we compute a generalization of the polynomial~\eqref{polynomial_P} defined in Section~\ref{sec:implementation}
\begin{align}
\label{eq:gTV_supportpol}
P_{\widehat{C}}(t)=\sum_{k=1}^m \abs{ \brac{\mathcal{F}_{n}^{\ast} \, \widehat{C}_k} (t)}^2.
\end{align}
The following lemma establishes that we can use this polynomial to locate the support of the primal solution. We defer the proof to Section~\ref{proof:pol_commonsupport} in the appendix. 
\begin{lemma}
\label{lemma:pol_commonsupport}
Let $\widehat{X}$ be a solution to Problem~\eqref{eq:gTVnormMin} such that  $\widehat{X}_{k}(t) = \sum_{t_j \in \widehat{T} } \hat{A}_{jk} \delta_{t_j}$ for $1\leq k \leq m$, where $\hat{A} \in \C^{\abs{\widehat{T}} \times m}$ and $\widehat{T}$ denotes the support of the solution. Then, for any corresponding dual solution $\abs{C}$, the support-locating polynomial satisfies
\begin{align*}
P_{\widehat{C}}(t_j)=1 \quad \text{ for all } t_j \in \widehat{T}.
\end{align*}
\end{lemma}
The careful reader will observe that the lemma does not imply that the support-locating polynomial cannot be equal to one all over the support and hence completely uninformative. However, in the proof of the lemma we show that the polynomial $ \mathcal{F}_{n}^{\ast} \, \widehat{C}_k $, which is restricted to the hypercylinder of radius one by constraint~\eqref{eq:gTVnormMin_dual}, is equal to $A_{jk}/\normTwo{A_{j:}}$ on $\widehat{T}$ (recall that $A_{j:}$ is the $j$th row of $A$). This means that except in pathological cases where the polynomial stays on the surface of the hypercylinder over all of its domain, the support-locating polynomial is informative. Indeed, this is what we observe in numerical simulations. See Section~4 of~\cite{superres} for a related discussion.

Finally, the procedure to estimate the signals is as follows.
\begin{enumerate}
\item Solve the finite-dimensional semidefinite program~\eqref{eq:gTV_normMin_sdp}.
\item Construct the support-locating polynomial $\eqref{eq:gTV_supportpol}$. Set $\widehat{T}$ to the set of points where its magnitude is equal to one.
\item Estimate the amplitudes of the individual signals by solving the corresponding systems of equations.
\end{enumerate}
Figure~\ref{fig:mm_polynomials} shows the polynomial $P_{\widehat{C}}$ for an example where $m=3$, $s=15$, $\fc=40$ and the signal amplitudes are real valued and generated at random. The minimum separation of the common support is $0.7 \lambda_c$. The polynomial $P_{\widehat{C}}$ obtained from~\eqref{eq:gTV_normMin_sdp} allows to locate the support of the signal with great precision. In contrast, the individual support-locating polynomials obtained from solving the individual TV-norm minimization problems do not allow to locate the support: they miss some spikes and also produce false positives.~\footnote{A Matlab script reproducing this example is available at \url{http://www.cims.nyu.edu/~cfgranda/scripts/superres_common_support.m}.}

\subsubsection{Theoretical analysis: Dual certificate}
\label{sec:dual_cert_commonsupport}
As in the case of Problems~\eqref{eq:TVnormMin} and~\eqref{eq:primal_sinesspikes}, the dual certificate corresponding to Problem~\eqref{eq:gTVnormMin} provides some useful intuition as to the performance we can expect from the optimization-based method depending on the structure of the signals that generate the data. 
\begin{lemma}
\label{lemma:cert_common_support}
For $X$ defined by~\eqref{eq:model_common_support} and $Y$ defined by~\eqref{eq:measurements_commonsupport}, if there exists an $m$-dimensional low-pass polynomial $Q: \mathbb{T} \rightarrow \C^{m}$, 
\begin{align*}
Q_k (t) = \brac{\mathcal{F}_n^{\ast} C_k} \brac{t} = \sum_{l = -\fc}^{\fc} C_{lk} e^{i2\pi lt} , \quad 1 \leq k \leq m, \; C = \MAT{C_1 & \cdots & C_m} \in \C^{n \times m},
\end{align*}
such that
\begin{align*}
& Q_k(t_j)  = \frac{A_{jk}}{\sqrt{\sum_{k'=1}^m \abs{A_{jk'}}^2}}, \quad t_j \in T, \\
& \max_{ t \in T^c} \sum_{k=1}^{m} \abs{ Q_k(t) }^2  < 1 ,
\end{align*}
then $X$ is the unique solution to Problem~\eqref{eq:gTVnormMin}.
\end{lemma}
The dual certificate provided by this lemma (see Section~\ref{proof:cert_common_support} in the appendix for the proof) is an $m$-dimensional low-pass polynomial $Q$ that interpolates points from the $m$-dimensional unit hypersphere $\mathcal{S}^m$ on the support of the signal while remaining strictly inside the hypercylinder $\mathbb{T} \times \mathcal{S}^m$ elsewhere. If the amplitudes of the $m$ signals are in a general position, then it is plausible that interpolating these points may become easier as the dimension $m$ of the hypercylinder increases. This would imply that using a larger number of signals may allow to recover common supports with smaller minimum separations. This is confirmed by the numerical simulations in Figure~\ref{fig:mm_phase_transitions}, which show that exact recovery occurs for decreasing values of the minimum distance as we increase $m$.

\subsubsection{Numerical results}

Figure~\ref{fig:mm_phase_transitions} shows the fraction of successes achieved by gTV minimization for different values of the minimum distance, of the number of signals $m$ and of the number of point sources in each signal $\abs{T}$. As in the case of TV minimization for a single signal (see Figure~\ref{fig:limits}), a phase transition occurs at a certain value of the minimum distance. This critical value of the minimum distance decreases as $m$ increases, saturating at $\lambda_c / 2$. Below this value, there exist signals that may lie almost in the null space of the measurement operator (note that this is a stronger statement than the \emph{difference} of two signals being almost in the null space, see Section~\ref{sec:minimum_separation}). Characterizing the minimum separation at which gTV-norm minimization succeeds for amplitudes in a general position is an intriguing research problem.

\begin{figure}[t]
\centering
\begin{tabular}{  >{\centering\arraybackslash}m{0.45\linewidth} >{\centering\arraybackslash}m{0.45\linewidth}  }
 $m$ = 1 (Real amplitudes) & $m$ = 1 (Complex amplitudes)\\
\vspace{0.2cm}
\includegraphics{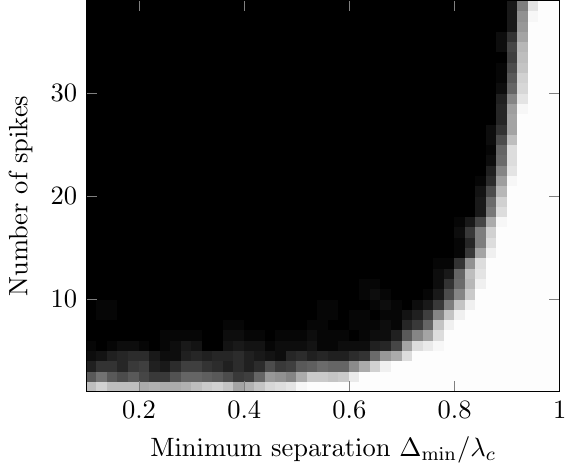}
 &  
\includegraphics{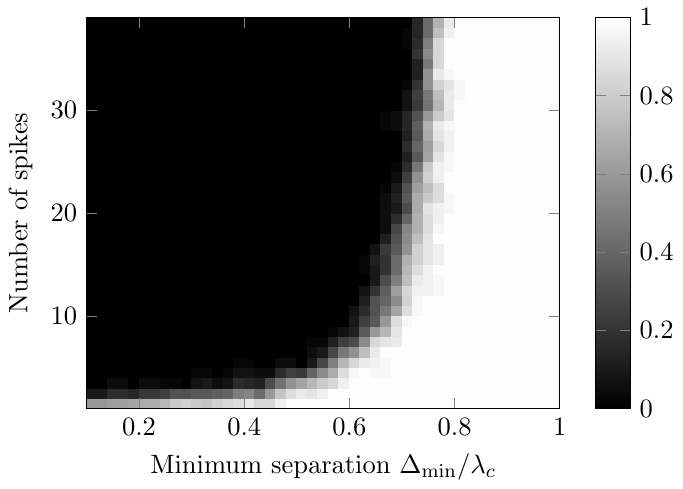}
\\
   $m$ = 2 (Complex amplitudes) & $m$ = 10 (Complex amplitudes)\\
\includegraphics{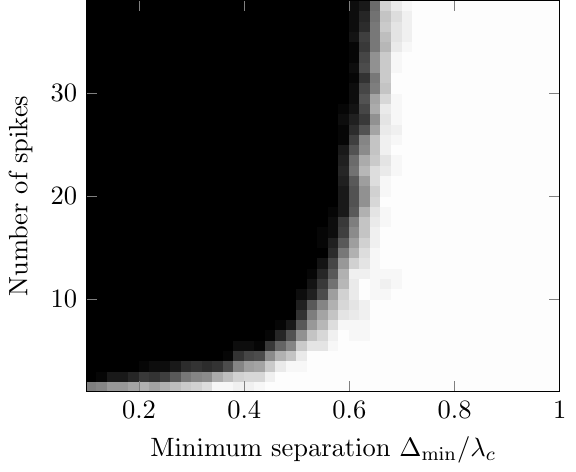}
 &
\includegraphics{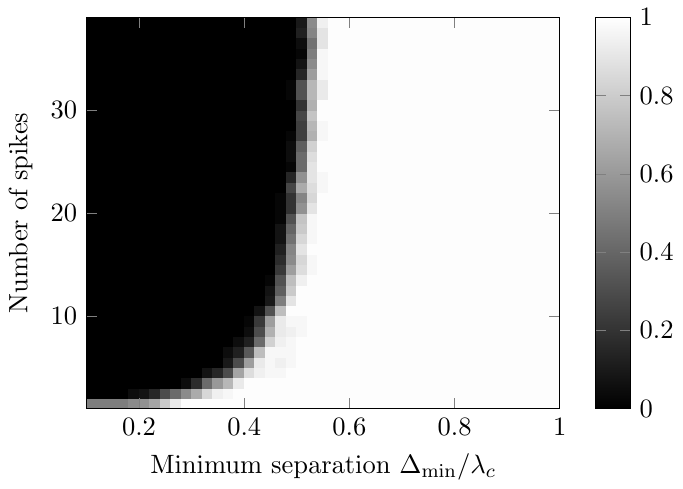}
\end{tabular}
\caption{Graphs showing the fraction of times Problem~\eqref{eq:gTVnormMin} achieves exact recovery over 10 trials with random signs and supports for different values of $m$. As $m$ grows the minimum separation at which exact recovery occurs tends to $\lambda_c / 2$. The simulations are carried out using CVX~\cite{cvx}.}
\label{fig:mm_phase_transitions}
\end{figure}
\section{Proof of Proposition~\ref{prop:dualcert}}
\label{proof:noiseless}

Our aim is to construct a  low pass polynomial that interpolates an arbitrary sign pattern $v \in \C^{\abs{T}}$ on a set of points $T$ separated by a minimum distance of $\optvalue / \fc$. As a first idea, we could consider interpolating the sign pattern with a Dirichlet kernel,
\begin{align*}
K \brac{f,t} := \frac{1}{2f+1} \sum_{k = -f}^{f}  e^{i2\pi k t} 
\end{align*}
which can also be written as
\begin{align}
K \brac{f,t} =
  \begin{cases}
   1 & \text{if } t=0\\
   \frac{\sin \brac{\brac{2 f+1} \pi t}}{\brac{2 f+1}\sin \brac{\pi t}} & \text{otherwise } 
  \text{.}
  \end{cases}
\label{eq:Dirichlet_kernel} 
\end{align}
The result of the interpolation would satisfy~\eqref{eq:q_cond_lowfreq} and~\eqref{eq:q_cond_interp} by construction, but controlling its magnitude on the off-support would be problematic because of the slow decay of the tail of $K$. An alternative is to use a faster-decaying kernel instead, as in~\cite{superres} which uses the fourth power of the Dirichlet kernel, also known as Jackson kernel. Taking the fourth power of a kernel is equivalent to repeatedly convolving it with itself in the frequency domain, which smooths out the discontinuities of its spectrum. The resulting kernel is consequently not as \emph{spiky} at the origin, but has a faster decaying tail. Generalizing this idea, we can take the product of $p$ kernels with different bandwidths $\gamma_1 \fc$, $\gamma_2 \fc$, \dots, $\gamma_p \fc$  to achieve a better tradeoff between the behavior of the resulting kernel at the origin and the decay of its tail. This yields a kernel of the form
\begin{align}
K_{\gamma}\brac{t} := \prod_{j=1}^{p}K \brac{\gamma_j \fc, t} = \sum_{k = -\fc}^{\fc} c_k e^{i2\pi k t} 
\label{eq:interpolation_kernel} 
\end{align}
where $c \in \C^{n}$ is the convolution of the Fourier coefficients of $K\brac{\gamma_1 \fc,t} ,K\brac{\gamma_2 \fc,t}, \dots K\brac{\gamma_p \fc,t} $. As long as $\sum_{j=1}^{p}\gamma_j=1$, the cut-off frequency of the kernel is indeed $\fc$. We fix $p=3$ and $\gamma_1=\gammaOne$, $\gamma_2=\gammaTwo$ and $\gamma_3=\gammaThree$. Figure~\ref{fig:kernel_gamma} shows $K_{\gamma}$ and its spectrum.

\begin{figure}
\begin{tabular}{ >{\centering\arraybackslash}m{0.31\linewidth} >{\centering\arraybackslash}m{0.32\linewidth}   >{\centering\arraybackslash}m{0.32\linewidth} }
\vspace{0.2cm}
\includegraphics{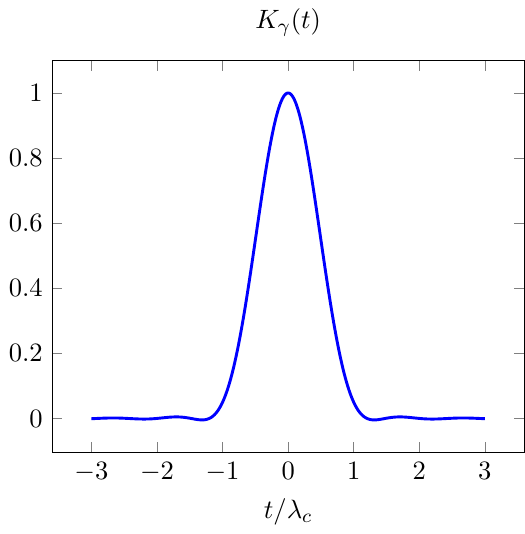}
&
\vspace{0.2cm}
\includegraphics{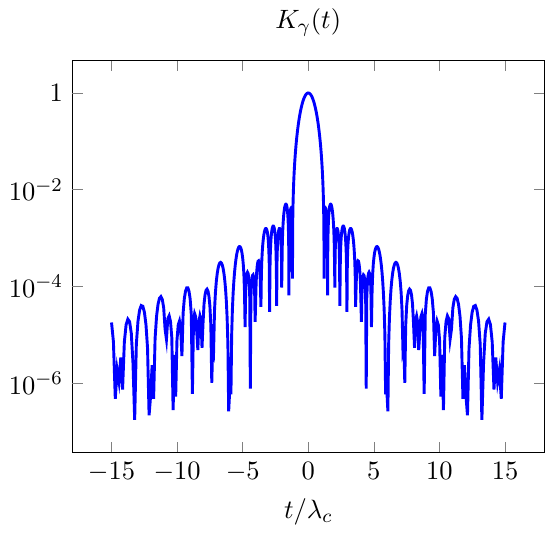}
&
\vspace{-0.21cm}
\includegraphics{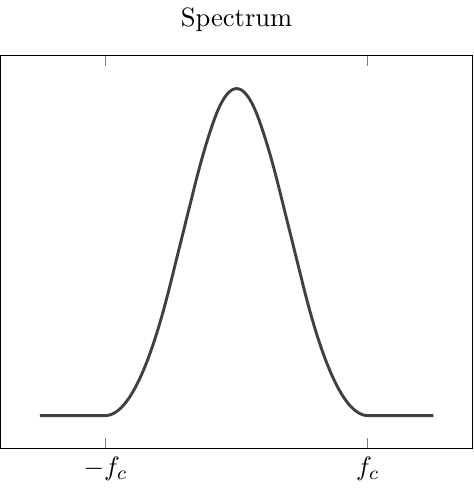}
\end{tabular}
\caption{Interpolating kernel used in the proof (left), along with its asymptotic decay (center) and its low-pass spectrum (right). It is the product of three Dirichlet kernels with cut-off frequencies $\gammaOne \, \fc$, $ \gammaTwo \, \fc$ and $\gammaThree \, \fc$. }
\label{fig:kernel_gamma}
\end{figure}

Interpolating the sign pattern with an adequately selected kernel is not sufficient to produce a valid construction. The reason is that the magnitude of the polynomial will tend to exceed one near the elements of the support $T$. This can be avoided, however, by forcing the derivative of the kernel to be zero at those points, an idea introduced in~\cite{superres}. As a result, the magnitude of the polynomial has a local extremum at the interpolation point and obeys condition~\eqref{eq:q_quadratic} for a sufficiently large minimum separation, as illustrated by Figure~\ref{fig:convexity}. In order to enforce the extra constraint on the derivative of 
$q$, we need more degrees of freedom in the construction. For this purpose, we incorporate the derivative of the kernel, so that the construction is of the form
\begin{equation}
  q(t)  = \sum_{t_j \in T} \alpha_j K_{\gamma}(t-t_j) + \beta_j K_{\gamma}^{\brac{1}}(t-t_j), 
  \label{eq:dual_cert}
\end{equation}
where $\alpha, \beta \in \C^{\abs{T}}$ are coefficient sequences satisfying
\begin{align}
q(t_k) &= \sum_{t_j \in T} \alpha_j K_{\gamma}\brac{t_k-t_j} + \beta_j
K_{\gamma}^{\brac{1}}\brac{t_k-t_j} = v_k, \quad   t_k \in T,   \label{eq:interp1}\\
q_R^{\brac{1}}(t_k) + i q_I^{\brac{1}}(t_k) &= \sum_{t_j \in T} \alpha_j K_{\gamma}^{\brac{1}}\brac{t_k-t_j} + \beta_j
K_{\gamma}^{\brac{2}}\brac{t_k-t_j} = 0, \quad   t_k \in T,  \label{eq:interp2}
\end{align}
and $q_R$ and $q_I$ are the real and imaginary parts of $q$.

\begin{figure}[t]
\begin{tabular}{ >{\centering\arraybackslash}m{0.47\linewidth} >{\centering\arraybackslash}m{0.47\linewidth} }
\includegraphics{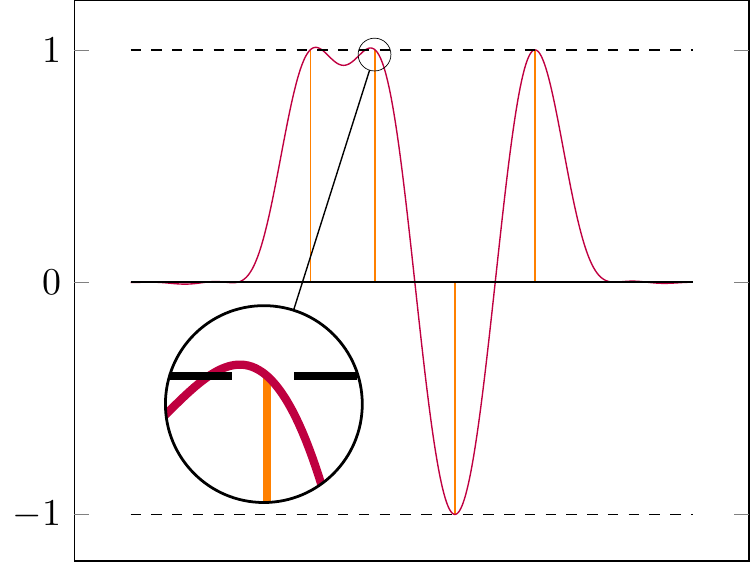}
&
\includegraphics{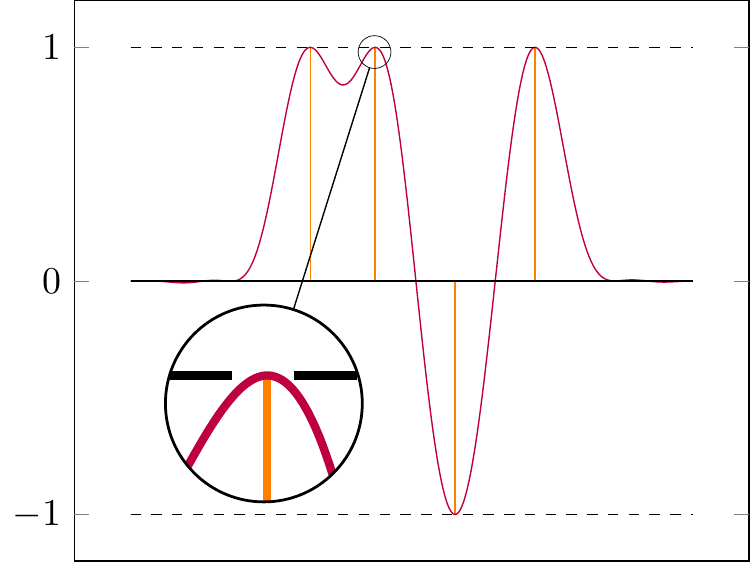}
\end{tabular}
\caption{Dual-polynomial if we only use $K_{\gamma}$ to interpolate the sign pattern (left), as opposed to also incorporating its derivative into the construction and forcing the derivative of the polynomial to be zero at the interpolation points (right).}
\label{fig:convexity}
\end{figure}

This construction trivially satisfies~\eqref{eq:q_cond_lowfreq} because both $K_{\gamma}$ and $K_{\gamma}^{\brac{1}}$ are low pass. The following lemma, proved in Section~\ref{proof:coeffs}, establishes that~\eqref{eq:q_cond_interp} holds and derives some bounds on the interpolation coefficients.
\begin{lemma}
\label{lemma:coeffs}
Under the hypotheses of Theorem~\ref{theorem:noiseless}, there exist coefficient vectors $\alpha$ and $\beta$ obeying~\eqref{eq:interp1} and~\eqref{eq:interp2}, such that
\begin{equation}
 \label{bound_alpha}
\begin{aligned}
  \normInf{\alpha}  & \leq \alpha^{\text{U}} := 1+ 2.370 \; 10^{-2} ,\\
  \normInf{\beta}& \leq \beta^{\text{U}} :=  4.247 \; 10^{-2} \,
  \lambda_c.
\end{aligned}
\end{equation}
In addition, if $v_1=1$,
\begin{equation}
 \label{bound_Re_alpha1}
\begin{aligned}
  \operatorname{Re}{\alpha_1} & \geq \alpha^{\text{L}}_R:= 1-2.370 \; 10^{-2}, \qquad 
  \abs{\operatorname{Im}{\alpha_1}} & \leq \alpha^{\text{U}}_I:=2.370 \; 10^{-2}.
\end{aligned}
\end{equation} 
\end{lemma}
The proof is completed by two lemmas, proved in Sections~\ref{proof:concavity} and~\ref{proof:boundq}, which show that our construction satisfies~\eqref{eq:q_far} and~\eqref{eq:q_quadratic} .
\begin{lemma}
\label{lemma:concavity}
Fix $t_0 \in T$. Under the hypotheses of Theorem~\ref{theorem:noiseless}, $|q(t)| < 1$ for $\abs{t - t_0} \in
(0, \taumiddle \, \lambda_c]$ and there exist numerical constants $C_0 \in \brac{0,1}$, $C_1$ and $C_2$ such that~\eqref{eq:q_quadratic} holds.
\end{lemma}
\begin{lemma}
\label{lemma:boundq}
Fix $t_0 \in T$. Let $t_{+}$ be the nearest element of $T$ such that $t_{+}>t_0$ and $t_{-}$ be the nearest element of $T$ such that $t_{-}<t_0$. Under the hypotheses of Theorem~\ref{theorem:noiseless}, $|q(t)| < 1$ for $t_0 - \frac{t_0-t_{-}}{2} \leq t < t_0$ and $ t_0 < t \leq t_0 + \frac{t_{+}-t_0}{2} $.
\end{lemma}
The proofs of Lemmas~\ref{lemma:coeffs}, \ref{lemma:concavity} and \ref{lemma:boundq} make heavy use of bounds on $K_{\gamma}$ that are presented in Section~\ref{sec:kernel}. 

The proofs of the different lemmas required to establish our main result involve some rather intricate calculations. In order to ease the task of checking the proof, we provide a Matlab script that performs these calculations at \url{http://www.cims.nyu.edu/~cfgranda/scripts/superres_proof.zip}.

\subsection{Bounds on the interpolation kernel}
\label{sec:kernel}
Our results rely on sharp non-asymptotic bounds on the Dirichlet kernel $K$~\eqref{eq:Dirichlet_kernel} and its derivatives, which can in turn be used to establish bounds on the interpolation kernel $K_{\gamma}$~\eqref{eq:interpolation_kernel}  and its derivatives. 
To avoid excessive clutter we defer the precise expressions and their proofs to the appendix. The following lemma, proved in Section~\ref{proof:bounds_near}, provides upper and lower bounds near the origin.
\begin{lemma}
For $\fc \geq \minfc$ and $\ell \in {\keys{0,1,2,3}}$, for any $t$ in the unit interval if $\tau $ is such that $\abs{\fct-\tau }\leq \epsilon$ then
\label{lemma:bounds_near}
\begin{align*}
B_{\gamma,\ell}^{\text{L}} \brac{\tau}-\brac{2 \pi}^{\ell+1} \fc^{\ell} \epsilon \leq K^{\brac{\ell}}_{\gamma}\brac{t} \leq  B_{\gamma,\ell}^{\text{U}} \brac{\tau} +\brac{2 \pi}^{\ell+1} \fc^{\ell} \epsilon
\end{align*}
where $B_{\gamma,\ell}^{\text{L}} $ and $B_{\gamma,\ell}^{\text{U}}$ are defined in Section~\ref{sec:bounds_near} of the appendix.
\end{lemma}
A direct corollary yields a sharp bound on the magnitude of the kernel and its derivatives.
\begin{corollary}
\label{cor:boundinf}
For $\fc \geq \minfc$, $\ell \in {\keys{0,1,2,3}}$ and any $t$ in the unit interval, if $\tau $ is such that $\abs{\fct-\tau }\leq \epsilon$ then
\begin{align*}
\abs{ K^{\brac{\ell}}_{\gamma}\brac{t}} \leq B_{\gamma,\ell}^{ \infty } \brac{ \tau , \epsilon } := \max \keys{\abs{B_{\gamma,\ell}^{\text{L}} \brac{\tau}},\abs{B_{\gamma,\ell}^{\text{U}} \brac{\tau}}  } +\brac{2 \pi }^{\ell+1} \fc^{\ell} \epsilon .
\end{align*}
\end{corollary}
Away from the origin, we establish bounds on the absolute value of $K_{\gamma}$ and its derivatives that are weaker, but decreasing. The following lemma is proved in Section~\ref{proof:tail_bounds} of the appendix.
\begin{lemma}
\label{lemma:tail_bounds}
For $\fc \geq \minfc$, $\ell \in {\keys{0,1,2,3}}$ and $\abs{ t } < \frac{450}{ \fc }$
\begin{align*}
\abs{ K^{\brac{\ell}}_{\gamma}\brac{t}} \leq  b_{\gamma,\ell} \brac{\fc t} 
\end{align*}
where $ b_{\gamma,\ell}$ are decreasing functions defined in Section~\ref{sec:tail_bounds} of the appendix.
\end{lemma}
Corollary~\ref{cor:boundinf} and Lemma~\ref{lemma:tail_bounds} are combined in the next lemma, proved in Section~\ref{proof:kernelsum}, to bound the sum of absolute values of shifted copies of the kernel (and of its derivatives) if they are separated by the minimum distance. The lemma is stated for an interval around the origin, but it holds for any position in the unit interval.
\begin{lemma}
\label{lemma:kernelsum}
Assume $0 \in T$ and let $\tauminth : = \Deltaminth / \lambda_c=\optvalue$ denote the minimum separation of the set $T$, scaled so that $\Delta\brac{T} \geq \tauminth /\fc$. If $\fc \geq \minfc$ and $\gamma=\sqbr{\gammaOne,\gammaTwo,\gammaThree}^T$, for any $0 \leq t \leq \Deltaminth/2$, $
\epsilon>0$ and $\tau$ such that $\tau-\epsilon \leq \fct \leq \tau$  
\begin{align*}
\sum_{t_j \in \,  T\setminus \keys{0} } \abs{K^{\brac{\ell}}_{\gamma} \brac{t-t_j}} & \leq   H_{\ell}\brac{\tau } + H_{\ell}\brac{-\tau}  ,
\end{align*}
where  
\begin{align*}
H_{\ell}\brac{ \tau } & := \sum_{j=1}^{\jnear}  \max \keys{\max_{ u \in \mathcal{G}_{j,\tau} } B_{\gamma,\ell}^{ \infty } \brac{u, \epsilon}, b_{\gamma,\ell} \brac{ \brac{j+4} \tauminth }} + \widetilde{C}_{\ell} ,
\end{align*}
where $\mathcal{G}_{j,\tau}$ is an $\epsilon$-grid of equispaced points covering the interval $\sqbr{j\tauminth-\tau,\brac{j+4}\tauminth}$ and  
\begin{align*}
\widetilde{C}_{\ell} := \sum_{j=\jnearpone}^{\jzero}  b_{\gamma,\ell} \brac{ \brac{j-1/2} \tauminth } + C_{\ell}, 
\end{align*}
$C_{0} = 8.852 \, 10^{-7}$, $C_{1} = 5.562 \, 10^{-6}$, $C_{2} = 3.495 \, 10^{-5}$, $C_{3} = 2.196 \, 10^{-4}$.
\end{lemma}

\subsection{Proof of Lemma~\ref{lemma:coeffs}}
\label{proof:coeffs}
Set
\[
\brac{D_0}_{jk} =K_{\gamma}\brac{t_j - t_k}, \quad \brac{D_1}_{jk} =
K_{\gamma}^{\brac{1}}\brac{t_j - t_k}, \quad \brac{D_2}_{jk} = K_{\gamma}^{\brac{2}}\brac{t_j - t_k},
\]
where $j$ and $k$ range from $1$ to $\abs{T}$. With this,
\eqref{eq:interp1} and \eqref{eq:interp2} become
\[
\begin{bmatrix} D_0 & D_1\\ D_1 & D_2 \end{bmatrix} \begin{bmatrix} \alpha \\
  \beta \end{bmatrix} =\begin{bmatrix} v\\0 \end{bmatrix}.
\]
By standard linear algebra result, this system is
invertible if and only if $D_2$ and its Schur complement $D_0 -
D_1D_2^{-1}D_1$ are both invertible. To prove that this is the case we
can use the fact that a symmetric matrix $M$ is invertible if
\begin{equation}
 \normInfInf{\Id- M} < 1, \label{inf_invertibility}  
\end{equation}
where $\|A \|_\infty$ is the usual infinity norm of a matrix defined as
$\|A\|_\infty = \max_{\|x\|_\infty = 1} \|Ax\|_\infty = \max_i \sum_j
|a_{ij}|$. This follows from $M^{-1} = (I - H)^{-1} = \sum_{k \ge 0}
H^k$, $H = I - M$, where the series is convergent since $
\normInfInf{H} < 1$. In particular,
\begin{equation}
  \normInfInf{M^{-1}}  \leq \frac{1}{1-\normInfInf{\Id-M}}.
\label{infinfnorminv}
\end{equation}
By Lemma~\ref{lemma:kernelsum}, 
\begin{align}
\normInfInf{\Id- D_0} & \le \max_{t_0 \in T}\sum_{t_j \in T\setminus \{t_0\}} \abs{ K_{\gamma} \brac{t_0-t_j}} \leq
  2H_{0}\brac{0} \le 1.855 \, 10^{-2},\label{Id_F_bound}\\
\normInfInf{D_1}   & \le  \max_{t_0 \in T} \sum_{t_j \in T\setminus \{t_0\}} \abs{K_{\gamma}^{\brac{1}}\brac{t_0-t_j}} \leq
 2H_{1}\brac{0}  \leq  0.1110 \; \fc, \label{D1_bound}\\
\normInfInf{ K_{\gamma}^{\brac{2}}\brac{0}\Id-D_2} & \leq \max_{t_0 \in T} \sum_{t_j \in T\setminus \{t_0\}} \abs{K_{\gamma}^{\brac{2}}\brac{t_0-t_j}} \leq 2H_{2}\brac{0} \leq   1.895 \; \fc^2
\label{Id_D2_bound}\text{.}
\end{align}

Note that $D_2$ is symmetric because the second derivative of the
interpolation kernel $K_{\gamma}^{\brac{2}}$ is symmetric. The bound \eqref{Id_D2_bound} implies
\[
\normInfInf{\Id-\frac{D_2}{ K_{\gamma}^{\brac{2}}\brac{0}}}  < 1, 
\]
which implies the invertibility of $D_2$.  
To control the infinity norm of the inverse of $D_2$ we need bounds on the value of $K_{\gamma}^{\brac{2}}$ at the origin. 
\begin{lemma}
\label{lemma:K20}
If $\fc \geq \fmineq$
\begin{align*}
K_{\gamma}^{\brac{2}} \brac{0} \leq B_{\gamma,2}^{0,\text{\em U}} < 0 , 
\quad \text{ where} \quad
B_{\gamma,2}^{0,\text{\em U}} = \sum_{j=1}^{p} B_{2}^{0,\text{\em U}} \brac{\gamma_j}
\quad \text{ and } \quad 
B_{2}^{0,\text{\em U}} \brac{\gamma_j} = - \frac{4 \pi ^2 \gamma_j^2  \fc^2}{3}.
\end{align*}
\end{lemma}
\begin{proof}
The result is a consequence of $K \brac{f,0}=1$, $K^{\brac{1}}\brac{f,0}=0$, $K^{\brac{2}}\brac{f,0}=- 4 \pi ^2 f \brac{1+f}/3$ and~\eqref{eq:derivatives}.
\end{proof}
By the lemma and~\eqref{infinfnorminv}, we have
\begin{align}
\normInfInf{D_2^{-1}} & \leq \frac{1}{\abs{K_{\gamma}^{\brac{2}}\brac{0}} - \normInfInf{ K_{\gamma}^{\brac{2}}\brac{0}\Id-D_2}} \notag \\
& \leq \frac{1}{\abs{B_{\gamma,2}^{0,\text{U}}} - \normInfInf{ K_{\gamma}^{\brac{2}}\brac{0}\Id-D_2}} \leq  0.3738 \; \lambda_c^2. \label{D2_inv_bound}
\end{align}
Combining this with \eqref{Id_F_bound} and \eqref{D1_bound} yields 
\begin{align}
\normInfInf{\Id- \brac{D_0-D_1D_2^{-1}D_1}} & \leq \normInfInf{\Id- D_0} + \normInfInf{D_1}^2\normInfInf{D_2^{-1}}\notag\\
& \leq 2.315 \; 10^{-2} <1 \label{invschur_bound}\text{.}
\end{align}
Note that the Schur complement of $D_2$ is symmetric because $D_0$ and
$D_2$ are both symmetric whereas $D_1^T=-D_1$ since the derivative of
the interpolation kernel is odd. This shows that the Schur complement
of $D_2$ is invertible and, therefore, the coefficient vectors
$\alpha$ and $\beta$ are well defined.

There just remains to bound the interpolation coefficients, which can
be expressed as
\[
  \begin{bmatrix} \alpha \\
    \beta \end{bmatrix} =\begin{bmatrix} \Id \\
    -D_2^{-1}D_1 \end{bmatrix} C^{-1} v, \quad C :=
  D_0-D_1D_2^{-1}D_1,
\]
where $C$ is the Schur complement.  The relationships
\eqref{infinfnorminv} and \eqref{invschur_bound} immediately give a
bound on the magnitude of the entries of $\alpha$
\[
\normInf{\alpha} = \normInf{C^{-1} v} \leq \normInfInf{C^{-1}} \leq 1+
2.370 \; 10^{-2}.
\]
Similarly, \eqref{D1_bound}, \eqref{D2_inv_bound} and
\eqref{invschur_bound} allow to bound the entries of $\beta$:
\begin{align*}
  \normInf{\beta} & \leq \normInfInf{D_2^{-1}D_1 C^{-1}}\\
  & \leq \normInfInf{D_2^{-1}} \normInfInf{D_1}
  \normInfInf{C^{-1}} \leq 4.247 \; 10^{-2} \,
  \lambda_c.
\end{align*}
Finally, with $v_1=1$, we can use \eqref{invschur_bound} to show that
$\alpha_1$ is almost equal to 1. Indeed,  
\begin{align*}
\abs{1 - \alpha_1} & = \abs{1-\brac{C^{-1}v}_1} \\
& = \abs{ \brac{ \Id - C^{-1} v }_1}\\
& \le \normInfInf{\Id - C^{-1}} \\
& = \normInfInf{C^{-1} (\Id - C)}\\
& \leq \normInfInf{C^{-1}} \, \normInfInf{\Id - C} \leq 2.370 \; 10^{-2}.
\end{align*}
This concludes the proof. 

\subsection{Proof of Lemma \ref{lemma:concavity}}
\label{proof:concavity}
Without loss of generality we map the unit circle $\mathbb{T}$ to the interval $\sqbr{-\frac{1}{2},\frac{1}{2}}$ and we set $t_0 = 0$ and $v_0 =
1$. 
In order to bound the magnitude of the dual polynomial locally, we control the second derivative of its square around the origin. In particular, we bound it on an interval of the form $ \fc \, t \in \,  [j \epsilon, \brac{j+1} \epsilon]$ for an integer $j \geq 0$ and $\epsilon>0$. Let us decompose $q=q_R+iq_I$ into its real and imaginary parts. The derivative is equal to
\begin{align*}
\derTwo{t}{\abs{q}^2}\brac{t} &= 2 \brac{ q_R\brac{t}q_R^{\brac{2}}\brac{t}+ q_I\brac{t} q_I^{\brac{2}}\brac{t} + \brac{q_R^{\brac{1}}\brac{t}}^2+\brac{q_I^{\brac{1}}\brac{t}}^2} \\
& = 2 \brac{ q_R\brac{t}q_R^{\brac{2}}\brac{t}+ q_I\brac{t} q_I^{\brac{2}}\brac{t} +\abs{\tilde{q}_1\brac{t}}^2},
\end{align*}
where $\tilde{q}_1\brac{t}:=q_R^{\brac{1}}+iq_I^{\brac{1}}$. To alleviate notation we define
\begin{align}
G_{\ell,j,\epsilon}\brac{a_1,a_2} := a_1B_{\gamma,\ell}^{ \infty } \brac{ j \epsilon , \epsilon }
+ a_2\brac{ H_{\ell}\brac{-\brac{j+1} \epsilon } +   H_{\ell}\brac{\brac{j +1}\epsilon}}, \label{eq:Gdef}
\end{align}
where $B_{\gamma,\ell}^{ \infty }$ is defined in Corollary~\ref{cor:boundinf}. By Lemmas~\ref{lemma:coeffs}, \ref{lemma:bounds_near} and~\ref{lemma:kernelsum} and Corollary~\ref{cor:boundinf} we have
\begin{align*}
  q_R\brac{t} & = \sum_{t_j \in T} \Real{\alpha_j} K_{\gamma}\brac{t-t_j} + \Real{\beta_j} K_{\gamma}^{\brac{1}}\brac{t-t_j}\notag\\
  & \geq \Real{\alpha_0}K_{\gamma}\brac{t} -\normInf{\alpha}\sum_{t_j \in T\setminus\keys{0}}\abs{K_{\gamma}\brac{t-t_j}} - \normInf{\beta}\sum_{t_j \in T} \abs{K_{\gamma}^{\brac{1}}\brac{t-t_j}}\notag\\
  & \geq \alpha^{\text{L}}_R\brac{ B_{\gamma,0}^{\text{L}} \brac{j \epsilon}- 2\pi \epsilon}
  -\alpha^{\text{U}} \brac{ H_{0}\brac{-\brac{j+1} \epsilon }  +   H_{0}\brac{\brac{j +1}\epsilon}} 
  -  G_{1,j,\epsilon}\brac{\beta^{\text{U}},\beta^{\text{U}}}
  \notag\\
&:=J_{0,R}^{\text{L}}\brac{j},
\end{align*}
for $j \epsilon \leq \fc \, t \leq \brac{j+1} \epsilon$ since $B_{\gamma,0}^{\text{L}}$ is positive in the interval of interest. Similarly,
\begin{align*}
\abs{ q_R\brac{t}} & \leq   G_{0,j,\epsilon}\brac{\alpha^{\text{U}},\alpha^{\text{U}}}
+  G_{1,j,\epsilon}\brac{\beta^{\text{U}},\beta^{\text{U}}}
:=J_{0,R}\brac{j} \text{,}\\
\abs{ q_I\brac{t}} & \leq G_{0,j,\epsilon}\brac{ \alpha^{\text{U}}_I ,\alpha^{\text{U}}}
+ G_{1,j,\epsilon}\brac{\beta^{\text{U}},\beta^{\text{U}}}
:=J_{0,I}\brac{j}, \\
\abs{ \tilde{q}_1\brac{t} } & \leq G_{1,j,\epsilon}\brac{\alpha^{\text{U}},\alpha^{\text{U}}}
+ G_{2,j,\epsilon}\brac{\beta^{\text{U}},\beta^{\text{U}}}
:=J_{1}\brac{j}, \\
\abs{q_R^{\brac{2}}\brac{t}} & \leq  G_{2,j,\epsilon}\brac{\alpha^{\text{U}},\alpha^{\text{U}}}
+G_{3,j,\epsilon}\brac{\beta^{\text{U}},\beta^{\text{U}}}
:= J_{2,R}\brac{j}, \\
\abs{q_I^{\brac{2}}\brac{t}} & \leq G_{2,j,\epsilon}\brac{ \alpha^{\text{U}}_I ,\alpha^{\text{U}}}
+ G_{3,j,\epsilon}\brac{\beta^{\text{U}},\beta^{\text{U}}}
:= J_{2,I}\brac{j},
\end{align*}
and if $B_{\gamma,2}^{\text{U}} \brac{j \epsilon}+\brac{2\pi}^3 \epsilon<0$
\begin{align*}
q_R^{\brac{2}}\brac{t} & \leq \alpha^{\text{L}}_R \brac{ B_{\gamma,2}^{\text{L}} \brac{ j \epsilon } +\brac{2\pi}^3 \epsilon}
  +\alpha^{\text{U}} \brac{ H_{2}\brac{-\brac{j+1} \epsilon} +   H_{2}\brac{ \brac{j+1} \epsilon } } 
+ G_{3,j,\epsilon}\brac{\beta^{\text{U}},\beta^{\text{U}}}
:= J_{2,R}^{\text{near}}\brac{j}.
\end{align*}
Combining these expressions we obtain a bound on the second derivative of the square of the polynomial $q$,
\begin{align*}
\derTwo{t}{\abs{q}^2}\brac{t} \leq F_{2,j},
\end{align*}
of the form
\begin{align}
 F_{2,j}:= \begin{cases}
2 \brac{ J_{2,R}^{\text{near}}\brac{j} J^{\text{L}}_{0,R}\brac{j} +  J_{2,I} \brac{j} J_{0,I}\brac{j} +  J_{1}\brac{j}^2}, \quad \text{if } B_{\gamma,2}^{\text{U}} \brac{ j \epsilon}<0 \\
 2 \brac{ J_{2,R}\brac{j} J_{0,R}^{\text{U}}\brac{j} +  J_{2,I} \brac{j} J_{0,I}\brac{j} +  J_{1}\brac{j}^2}, \quad \text{otherwise},
\end{cases}
\label{eq:F2j}
\end{align}
which holds for any $t$ such that $\fc \, t \in \sqbr{j \epsilon,\brac{j+1} \epsilon}$. This bound allows us to control the curvature of the squared magnitude of $\abs{q}$ on a fine grid, which can in turn be leveraged to obtain a bound on $\abs{q}$. This is made precise by the following lemma, proved in Section~\ref{proof:bound_2der} of the appendix.
\begin{lemma}
\label{lemma:bound_2der}
If $F$ is a twice continuously-differentiable real-valued function such that
\begin{align*}
F\brac{0}=1, \qquad F^{\brac{1}}\brac{0}=0, \qquad \sup_{ j \epsilon \leq \tilde{t} \leq \brac{j+1}\epsilon } F^{\brac{2}}\brac{\tilde{t}} \leq F_{2,j} ,
\end{align*}
for a certain positive real number $\epsilon$ and all nonnegative integers $j \geq 0$, then if $k \epsilon \leq t \leq \brac{k+1}\epsilon$ for an integer $k$ we have
\begin{align*}
F\brac{ t }& \leq F^{\infty}_k :=1 + \epsilon^2\brac{  \sum_{j=0}^{k-1 } \brac{\sum_{l=1}^{j-1}  F_{2,l} +\mathbbm{1}_{ F_{2,j}>0} F_{2,j} }   +    \mathbbm{1}_{F_{2}>0}F_{2} },  
\end{align*}
where
\begin{align*}
F_{2} = \sum_{j=0}^{k-1} F_{2,j} +\mathbbm{1}_{ F_{2,k}>0} F_{2,k}  .
\end{align*}
\end{lemma}
We apply the lemma by setting $F=\abs{q}^2$. By Lemma~\ref{lemma:coeffs}, $F\brac{0}=\abs{q\brac{0}}^2=1$ and $F^{\brac{1}}\brac{0}= 2\, q_R\brac{0}q_R^{\brac{1}}\brac{0}+2\, q_I\brac{0} q_I^{\brac{1}}\brac{0}=0$. To obtain a bound on $\abs{q}^2$ we evaluate the bound computationally on a grid with step size $\epsilon = 10^{-6}$ covering the interval $\sqbr{0,0.3}$. The corresponding values of $F^{\infty}_k$ are plotted on the left of Figure~\ref{fig:F2j}. Since $F^{\infty}_k$ lies below $1-6.394 \, 10^{-7}$ between $k=1$ and $k=\taumiddle/\epsilon$, we have
\begin{align*}
\abs{q\brac{t}} \leq 1-3.197 \, 10^{-7}, \qquad \abs{t-t_k} \leq \taumiddle \; \lambda_c, \quad t_k \in T,
\end{align*}
since we fixed $t_0 =0$ without loss of generality. 

For $t$ belonging to the interval $\brac{0,\epsilon}$ we apply the following inequality, which follows from the fundamental theorem of calculus and the facts that $F'\brac{0}=0$ and $F_{2,0} \leq -3.23\,\fc^2$ (the proof of Lemma~\ref{lemma:bound_2der} is based on a similar argument)
\begin{align*}
\sup_{t \in \brac{0,\epsilon}} \abs{q\brac{ t }}^2 & \leq \sup_{t \in \brac{0,\epsilon}} F\brac{ 0 }  + t^2 \sup_{u \in\sqbr{0,\epsilon}} F''\brac{u}\\
& = \sup_{t \in \brac{0,\epsilon}} 1 + F_{2,0}\,t^2 \\
& \leq  \sup_{t \in \brac{0,\epsilon}} 1 -3.23 \,\fc^2 t^2\\
& < 1.
\end{align*}

\begin{figure}[t]
\begin{tabular}{ >{\centering\arraybackslash}m{0.4\linewidth}  >{\centering\arraybackslash}m{0.03\linewidth} >{\centering\arraybackslash}m{0.4\linewidth} }
Bound on $\abs{q}^2$ &&  \begin{center} Bound on $\derTwo{t}{\abs{q}^2}$ \end{center} \\ 
\includegraphics{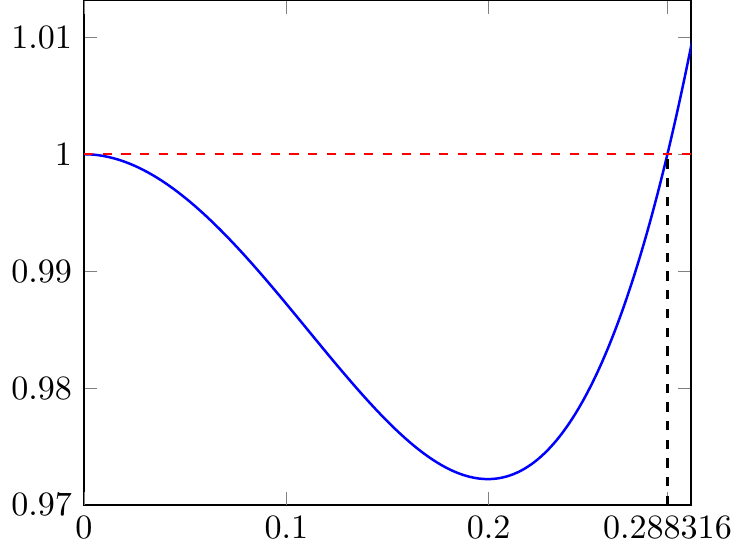}
&&
\includegraphics{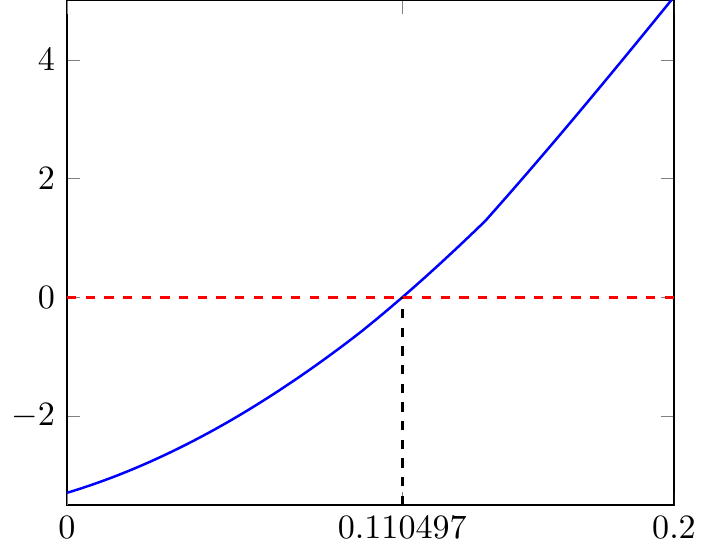}
\end{tabular}
\caption{On the left, plot of $F^{\infty}_k$ obtained by applying Lemma~\ref{lemma:bound_2der} with $F=\abs{q}^2$. The bound is evaluated on a grid with step size $\epsilon = 10^{-6}$ covering the interval $\sqbr{0,0.3}$. Between $10^{-6}$ and $\taumiddle$ it remains strictly below 1.
On the right, plot of $F_{2,j}$~\eqref{eq:F2j} evaluated on a grid with step size $\epsilon = 10^{-6}$ covering the interval $\sqbr{0,0.2}$. Between $10^{-6}$ and $\taumiddleSecondDer$ the bound is strictly negative.}
\label{fig:F2j}
\end{figure}

Finally, to establish the quadratic bound~\eqref{eq:q_quadratic} on the magnitude of $q$ we prove that in the interval the second derivative of $\abs{q}$ lies strictly below zero. In fact, we can make use of the bound~\eqref{eq:F2j}. Since
\begin{align*}
\derTwo{t}{\abs{q}}\brac{t} & = -\frac{\brac{q_R\brac{t}q_R^{\brac{1}}\brac{t}+
    q_I\brac{t}q_I^{\brac{1}}\brac{t}}^2}{\abs{q\brac{t}}^3}+
\frac{q_R\brac{t} q_R^{\brac{2}}\brac{t}+ q_I\brac{t}
  q_I^{\brac{2}}\brac{t} \abs{\tilde{q}_1\brac{t}}^2}{\abs{q\brac{t}}} \\
 & = -\frac{\brac{q_R\brac{t} q_R^{\brac{1}}\brac{t}+
    q_I\brac{t} q_I^{\brac{1}} \brac{t}}^2}{\abs{q\brac{t}}^3}+
\frac{1}{2\abs{q\brac{t}}}\derTwo{t}{\abs{q}^2}\brac{t}
\end{align*}
this second derivative is negative if $\derTwo{t}{\abs{q}^2}$ is negative. To bound this latter quantity we evaluate the bound $F_{2,j}$~\eqref{eq:F2j} on a grid with step size $\epsilon = 10^{-6}$ covering the interval $\sqbr{0,0.2}$. The result is plotted on the right of Figure~\ref{fig:F2j}. Between $0$ and $\taumiddleSecondDer$ the bound is strictly negative, which immediately implies that this is also the case for the second derivative of $\abs{q}$. This establishes the upper bound in~\eqref{eq:q_quadratic}, because $q\brac{t_0}=1$ and the derivative of $\abs{q}$ is zero by Lemma~\ref{lemma:coeffs}. To prove the lower bound we just need to bound the second derivative of the real and imaginary parts of $q-v_j$. Since the argument is almost identical to the one used to bound the second derivative of $\abs{q}^2$ we omit the details.

\subsection{Proof of Lemma~\ref{lemma:boundq}}
\label{proof:boundq}

\begin{figure}[t]
\begin{center}
\includegraphics{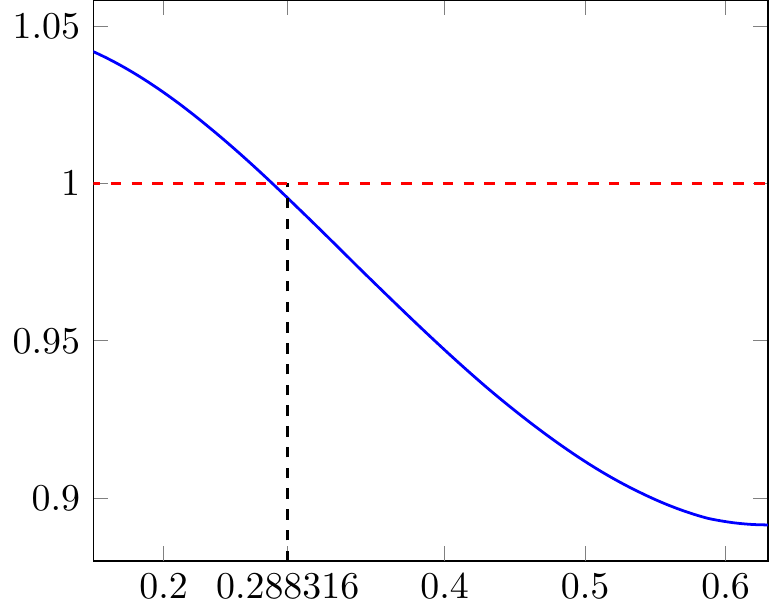}
\end{center}
\caption{The bound~\eqref{eq:qabs} is evaluated on a grid with step size $\epsilon = 10^{-6}$ covering the interval $\sqbr{0,0.3}$. Between $\taumiddle$ and $\taumin/2$ it remains strictly below 1.}
\label{fig:qabs}
\end{figure}

By Lemmas~\ref{lemma:coeffs}, \ref{lemma:bounds_near} and~\ref{lemma:kernelsum} and Corollary~\ref{cor:boundinf}
\begin{align}
  \abs{q\brac{t}} & \leq  G_{0,j,\epsilon}\brac{\alpha^{\text{U}},\alpha^{\text{U}}} +  G_{1,j,\epsilon}\brac{\beta^{\text{U}},\beta^{\text{U}}}, \label{eq:qabs}
\end{align}  
for any $t \leq \Deltamin/2$ in $\sqbr{j \epsilon,\brac{j+1} \epsilon}$, where $G_{\ell,j,\epsilon}$ is defined in~\eqref{eq:Gdef}. We evaluate the bound numerically on a grid with step size $\epsilon = 10^{-6}$ to establish that in the interval $\sqbr{\taumiddle, \taumin/2}$
\begin{align*}
  \abs{q\brac{t}} & \leq 1-4.641 \, 10^{-3}.
\end{align*}
The numerical values are plotted in Figure~\ref{fig:qabs}. 

Together with Lemma~\ref{lemma:concavity}, this completes the proof except for the following case: The first spike to the right of $t_0$, which we denote by $t_1$, might be at a larger distance than $\Deltamin$, so it is still necessary to bound $q$ in the interval $\sqbr{\Deltamin/2,t_0+\frac{t_1-t_0}{2} }$. Let $t$ be a point in this interval, then
\begin{align*}
 \abs{q\brac{t}} & = \abs{\sum_{t_j \in T} \alpha_j K_{\gamma}\brac{t-t_j} + \beta_j K_{\gamma}^{\brac{1}}\brac{t-t_j}}\notag\\
  & \leq \alpha^{\text{U}} \brac{ \abs{K_{\gamma}\brac{t-t_0}} +  \sum_{t_j \in T/\keys{t_0}}  \abs{K_{\gamma}\brac{t-t_j} } } + \beta^{\text{U}}\brac{ \abs{K_{\gamma}^{\brac{1}}\brac{t-t_0}} +  \sum_{t_j \in T/\keys{t_0}}  \abs{K_{\gamma}^{\brac{1}}\brac{t-t_j} } } \\
  & \leq \alpha^{\text{U}} \brac{ \max \keys{\max_{ u \in \mathcal{G}} B_{\gamma,0}^{ \infty } \brac{u, \epsilon}, b_{\gamma,0} \brac{ 2 \taumin }} + H_0\brac{\taumin/2} + H_0\brac{0} }\\
  & \quad + \beta^{\text{U}} \brac{ \max \keys{\max_{ u \in \mathcal{G}} B_{\gamma,1}^{ \infty } \brac{u, \epsilon}, b_{\gamma,1} \brac{ 2 \taumin }} + H_1\brac{\taumin/2} + H_1\brac{0} },
  \end{align*}
where $\mathcal{G}$ is a grid with step size $\epsilon$ covering the interval $\sqbr{\taumin/2,2 \taumin}$. The bound follows from Lemmas~\ref{lemma:coeffs}, \ref{lemma:bounds_near} and~\ref{lemma:kernelsum} and Corollary~\ref{cor:boundinf}. In particular, from the definition of $H_{\ell}$ in Lemma~\ref{lemma:kernelsum} it is not difficult to show that if $\taumin/2 \leq t \leq t_1-\taumin/2$ then
\begin{align*}
\sum_{t_j \in \, T \, \cap  \brac{0,\frac{1}{2}} }  \abs{K_{\gamma}^{\brac{\ell}}\brac{t-t_j} }  \leq H_{\ell}\brac{\taumin/2} \quad \text{and} \quad \sum_{t_j \in \, T \, \cap \left[-\frac{1}{2},0\right)  }  \abs{K_{\gamma}^{\brac{\ell}}\brac{t-t_j} }  \leq H_{\ell}\brac{0} .
  \end{align*}
Finally, we evaluate the bound numerically for a step size of $\epsilon = 10^{-6}$ to prove that
\begin{align*}
 \abs{q\brac{t}} & \leq 0.896
  \end{align*}
in $\sqbr{\Deltamin/2,t_0+\frac{t_1-t_0}{2} }$.

\section{Conclusion and future research}
In this work we provide a reasonably tight characterization of the performance of super-resolution via convex programming when the signal consists of point sources in one dimension that are separated by a minimum distance. Interesting research directions include developing a sharper bound than the one provided by~\cite{superres} for two-dimensional signals, studying other conditions that allow for a small degree of clustering in the support of the signal as in~\cite{venia_positive} and extending the analysis to other signal models such as piecewise-constant signals in two dimensions. We have also derived optimization-based algorithms to tackle the demixing of sines and spikes and the super-resolution of point sources with a common support. In both cases, it would be interesting to derive theoretical guarantees by building the appropriate dual certificates, as outlined in Sections~\ref{sec:dualcert_sinesspikes} and \ref{sec:dual_cert_commonsupport}. Finally, applying optimization-based approaches in real applications will require developing efficient algorithms to solve the corresponding optimization programs on large-scale datasets (see~\cite{sparse_inverse_ben} for recent work in this direction that is relevant to super-resolution).

\subsection*{Acknowledgements}
The author is grateful to Emmanuel Cand\`es for useful discussions and to the anonymous reviewers for helpful suggestions.
\begin{small}
\bibliographystyle{abbrv}
 \bibliography{refs_sr}
\end{small}
\appendix
\section{Proof of auxiliary results in Sections~\ref{sec:superres_point_sources} and~\ref{sec:convex_framework}}
\label{sec:proof_aux}
\subsection{Proof of Proposition~\ref{prop:sdp-charact}}
\label{proof:sdp-charact}
Let us define
\begin{align*}
z_0(t) & := \MAT{e^{ i 2 \pi \fc t} & e^{ i 2 \pi \brac{\fc-1} t} & \dots & e^{- i 2 \pi \fc t}}^T, \\
z_1(t) & :=  \MAT{1 & e^{ i 2 \pi t} & \dots & e^{ i 2 \pi n t} }^T.
 \end{align*}
Note that $\brac{ \mathcal{F}_{n}^{\ast} \, C_k}(t)= z_0(t)^{\ast} C_k $, so that $\sum_{k} \abs{ \mathcal{F}_{n}^{\ast} \, C_k}^2 = z_0(t)^{\ast} C C^{\ast} z_0(t) $.
If~\eqref{eq:sdp-charact} holds then $\Lambda-CC^{\ast} \succeq 0$ and $\mathcal{T}^{\ast}\brac{\Lambda}= e_1$. This implies 
\begin{align*}
\sum_{k} \abs{(\mathcal{F}_{n}^{\ast} \, C_k)(t)}^2 & = z_0(t)^{\ast}CC^{\ast} z_0(t)\\
& \leq z_0(t)^{\ast}\Lambda z_0(t)\\
& = \text{Tr}\brac{ z_0(t)^{\ast} \Lambda z_0(t)}\\
& =\PROD{\Lambda}{ z_0(t) z_0(t)^{\ast}}\\
&= \PROD{\Lambda}{\mathcal{T}\brac{z_1\brac{t}}}\\
&=\PROD{\mathcal{T}^{\ast}\brac{\Lambda}}{ z_1\brac{t} }\\
& =  \brac{z_1(t)}_1 = 1.
 \end{align*}
For the converse, $\sum_{k} \abs{ \mathcal{F}_{n}^{\ast} \, C_k}^2  \leq 1$ implies that the trigonometric polynomial $1-z_0(t)^{\ast}CC^{\ast} z_0(t)$ is non-negative. By the Fej\'er-Riesz Theorem there exists a polynomial $P(t)=\tilde{c}^{\ast} z_0(t)$ such that 
\begin{align}
1-z_0(t)^{\ast}CC^{\ast} z_0(t)=\abs{P(t)}^2=z_0(t)^{\ast}\tilde{c}\tilde{c}^{\ast} z_0(t). \label{eq:dsquare}
 \end{align}
Now let us set $\Lambda=CC^{\ast}+\tilde{c}\tilde{c}^{\ast}$. $\Lambda$ and $\Lambda-CC^{\ast}$ are positive semidefinite by construction, which implies
\begin{align*}
\MAT{\Lambda & C \\ C^{\ast} & \Id} \succeq 0.
 \end{align*}
Furthermore, by~\eqref{eq:dsquare}
\begin{align*}
1 & = z_0(t)^{\ast} \Lambda z_0(t)\\
& = \PROD{\Lambda}{ z_0(t) z_0(t)^{\ast} }\\
& = \PROD{\Lambda}{\mathcal{T}\brac{z_1(t)}}\\
& = \PROD{\mathcal{T}^{\ast}\brac{\Lambda}}{ z_1\brac{t} }
 \end{align*}
 for all $t$, which is only possible if $\mathcal{T}^{\ast}\brac{\Lambda}= e_1$. This completes the proof.
\subsection{Proof of Lemma~\ref{lemma:dual_sinesspikes}}
\label{proof:dual_sinesspikes}
The Lagrangian is equal to
\begin{align*}
\mathcal{L}\brac{\tilde{x},\tilde{s},c} & = \normTV{\tilde{x}} + \eta \normOne{\tilde{s}} + \PROD{c}{y-\mathcal{F}_{n} \, \tilde{x} - \tilde{s} }\\
& = \PROD{c}{y} + \normTV{\tilde{x}}-  \PROD{\mathcal{F}_{n}^{\ast} \,c}{\tilde{x}} + \eta \normOne{\tilde{s}} - \PROD{c}{ \tilde{s}}
\end{align*}
where $c \in \C^{n}$ is the dual variable. The lower bound on the cost function provided by minimizing the Lagrangian over $\tilde{x}$ and $\tilde{s}$ is trivial ($-\infty$) unless $\abs{ \brac{\mathcal{F}_{n}^{\ast} \, c} (t)} \leq 1$ for all $t$ and $\normInf{c} \leq \eta$. Under these constraints, the minimum is attained for $\tilde{x}=0$ and $\tilde{s}=0$ for any value of $c$. Maximizing the remaining expression over $c$ yields the dual problem.

\subsection{Proof of Lemma~\ref{lemma:primaldual_sinesspikes}}
\label{proof:primaldual_sinesspikes}
The interior of the feasible set of Problem~\eqref{eq:TVnormMin_sdp_sinesspikes} contains the origin and is
consequently non empty, so strong duality holds by a generalized
Slater condition \cite{rockafellar1974conjugate} and we have
\begin{align*}
\sum_{t_j \in \widehat{T}} \abs{\hat{a}_{j}} +\regpar \sum_{l \in \widehat{\Set}} \abs{\hat{s}_{l}}  = \normTV{\hat{x}} + \regpar \normOne{\hat{s}} & = \<\hat{c},y \>\\ 
& = \<\hat{c},\mathcal{F}_{n} \, \hat{x}+ \hat{s} \>\\
& = \operatorname{Re}\sqbr{\sum_{t_j \in \widehat{T}} \abs{\hat{a}_{j}}  \overline{\brac{\mathcal{F}_{n}^{\ast} \, \hat{c}}\brac{t_j}} e^{i\phi_j} +\sum_{l \in \widehat{\Set}} \abs{\hat{s}_{l}} \overline{\hat{c}_l}e^{i \psi_l}}.
\end{align*}
By H\"older's inequality and the constraints on $\hat{c}$,
this implies the result (otherwise equality would not be possible).

\subsection{Proof of Lemma~\ref{lemma:cert_sinesspikes}}
\label{proof:cert_sinesspikes}
For any vector $v$ and any atomic measure $\nu$, we denote by $v_{\Omega}$ and $\nu_{\Omega}$ the restriction of $v$ and $\nu$ to the subset of their support indexed by a set $\Omega$ ($v_{\Omega}$ has the same dimension as $v$ and $\nu_{\Omega}$ is still a measure in the unit interval). Let us consider an arbitrary solution $x'$ and $s'$, such that $x'\neq x$ or $s'\neq s$. Due to the constraints of the optimization problem, $x'$ and $s'$ satisfy 
\begin{align}
y=\mathcal{F}_n \, x+s=\mathcal{F}_n \, x'+s'. \label{eq:constraint}
\end{align}
The following lemma proved in Section~\ref{proof:F_T_I} below establishes that $x_{T^c}'$ and $s_{\Set^c}'$ cannot both be equal to zero.
\begin{lemma}
\label{lemma:F_T_I}
If $\keys{x',s'}$ is feasible and $x_{T^c}'$ and $s_{\Set^c}'$ both equal zero, then $x=x'$ and $s=s'$.
\end{lemma}
This implies that $x_{T^c}'$ and $s_{\Set^c}'$ cannot be a solution, since
\begin{align}
  \normTV{x'} +  \regpar \normOne{s'} & = \normTV{x_{T}'} +\normTV{x_{T^c}'} +  \regpar \normOne{s_{\Set}'} +  \regpar \normOne{s_{\Set^c}'}\\
& > \normTV{x_{T}'} +\PROD{Q}{x_{T^c}'} +  \regpar \normOne{s_{\Set}'} +  \PROD{c}{s_{\Set^c}'} \qquad \text{by Lemma~\ref{proof:F_T_I}, \eqref{eq:opt1} and \eqref{eq:opt2}} \notag \\
& \geq \PROD{Q}{x'} +  \PROD{c}{s'} \qquad \text{by \eqref{eq:opt1} and \eqref{eq:opt2}}\\
&=\PROD{\mathcal{F}_n^{\ast}c}{x'} +  \PROD{c}{s'}\\
&=\PROD{c}{\mathcal{F}_n \, x'+s'}\\
&=\PROD{c}{\mathcal{F}_n \, x+s} \qquad \text{by \eqref{eq:constraint}} \\
&=\PROD{\mathcal{F}_n^{\ast}c}{x} +  \PROD{c}{s}\\
&=\PROD{Q}{x} +  \PROD{c}{s}\\
& \geq   \normTV{x} +  \regpar \normOne{s} \qquad \text{by \eqref{eq:opt1} and \eqref{eq:opt2}}.
\end{align}

\subsubsection{Proof of Lemma~\ref{lemma:F_T_I}}
\label{proof:F_T_I}
If $x_{T^c}'$ and $s_{\Set^c}'$ both equal zero, then
\begin{align}
\mathcal{F}_n \, x+s - \mathcal{F}_n \, x_{T}'-s_{\Set}'  & = \mathcal{F}_n \, x'+s' - \mathcal{F}_n \, x_{T}'-s_{\Set}' \qquad \text{by \eqref{eq:constraint}} \\
& = \mathcal{F}_n \, x_{T^c}'+s_{\Set^c}'   \\
& = 0.  \label{eq:zero} 
\end{align}
We index the entries of $\Set := \keys{i_1,i_2, \ldots,i_{\abs{\Set}}}$ and define the matrix $\MAT{F_T & \Id_{\Set}} \in \C^{n \times \brac{\abs{T}+\abs{\Set}}}$, where
\begin{align}
\brac{F_T}_{lj} & = e^{i2\pi l t_j} \quad \text{for } 1\leq l \leq n, \; 1\leq j \leq  \abs{T}, \\
\brac{ \Id_{\Set}}_{lj} & = 
\begin{cases} 1 \quad \text{if } l =i_j \\ 
0 \quad \text{otherwise}\end{cases}\quad \text{for } 1\leq l \leq n, \; 1\leq j \leq  \abs{\Set}.
\end{align}
If $\abs{T}+ \abs{\Set} \leq n$ then $\MAT{F_T & \Id_{\Set}}$ is full rank (this follows from the fact that $F_T$ is a submatrix of a Vandermonde matrix). Equation~\eqref{eq:zero} implies
\begin{align}
\MAT{F_T & \Id_{\Set}} \MAT{a-a' \\ \mathcal{P}_{\Set}s - \mathcal{P}_{\Set} s'} = 0,
\end{align}
where $\mathcal{P}_{\Set} u' \in \C^s$ is the subvector of $u'$ containing the entries indexed by $\Set$ and $a' \in \C^T$ is the vector containing the amplitudes of $x'$ (recall that by assumption $x_{T^c}'=0$). We conclude that $x=x'$ and $s=s'$. 

\subsection{Proof of Lemma~\ref{lemma:dual_commonsupport}}
\label{proof:dual_commonsupport}
The Lagrangian is equal to
\begin{align*}
\mathcal{L}\brac{\widetilde{X}, C} & = \normgTV{\widetilde{X}} + \PROD{C}{Y -\MAT{\mathcal{F}_{n} \, \widetilde{X}_1 & \mathcal{F}_{n} \, \widetilde{X}_2 & \cdots & \mathcal{F}_{n} \, \widetilde{X}_m }} \\
& = \PROD{C}{Y} + \normgTV{\widetilde{X}} - \sum_{k=1}^{m}  \PROD{C_k}{\mathcal{F}_{n} \, \widetilde{X}_k}\\
& = \PROD{C}{Y} + \normgTV{\widetilde{X}} - \sum_{k=1}^{m}  \PROD{\mathcal{F}_{n}^{\ast} \, C_k}{ \widetilde{X}_k}\\
& = \PROD{C}{Y} +  \sup_{\normTwo{F\brac{t}} \leq 1} \sum_{k=1}^{m} \operatorname{Re}\keys{ \int_{\mathbb{T}} \overline{F_k \brac{t}} X_k\brac{\text{d}t} } -  \sum_{k=1}^{m} \operatorname{Re}\keys{ \int_{\mathbb{T}} \overline{\mathcal{F}_{n}^{\ast} \, C_k \brac{t}} X_k\brac{\text{d}t} }
\end{align*}
where $C \in \C^{n\times m}$ is the dual variable. The lower bound on the cost function provided by minimizing the Lagrangian over $\tilde{X}$ is trivial ($-\infty$) unless 
\begin{align*}
\sum_{k=1}^{m} \abs{ \brac{\mathcal{F}_{n}^{\ast} \, C_k} (t)}^2 \leq 1
\end{align*}
for all $t$. Under these constraints, the minimum is attained for $\tilde{X}=0$. Maximizing the remaining expression over $C$ yields the dual problem.

\subsection{Proof of Lemma~\ref{lemma:pol_commonsupport}}
\label{proof:pol_commonsupport}
Let us define the matrix $P \in \C^{\abs{T} \times m}$,
\begin{align*}
P_{jk} := \brac{\mathcal{F}_{n}^{\ast} \, \widehat{C}_k} \brac{t_j}.
\end{align*}
The result is an immediate consequence of the following equality
\begin{align}
\label{eq:P_A}
P_{j:}  = \frac{A_{j:}}{ \norm{A_{j:}}},
\end{align}
where $M_{j:}$ denotes the $j$th row of a matrix $M$. 

By strong duality, which holds by a generalized Slater condition \cite{rockafellar1974conjugate} we have that
\begin{align*} 
\normgTV{\widehat{X}} =  \<Y, \widehat{C} \>.
\end{align*}
On the one hand, 
\begin{align*} 
\normgTV{\widehat{X}} =  \sum_{j=1}^{\abs{T}} \normTwo{A_{j:}}.
\end{align*}
On the other,
\begin{align*} 
\<Y, \widehat{C} \> & =  \sum_{k=1}^{m} \<Y_k, \widehat{C}_k \>\\
&  =  \sum_{k=1}^{m} \<\mathcal{F}_n \widehat{X}_k, \widehat{C}_k \> \\
&  = \sum_{k=1}^{m} \< \widehat{X}_k, \mathcal{F}_n^{\ast} \widehat{C}_k \> \\
& = \sum_{j=1}^{\abs{T}}  \sum_{k=1}^{m} \text{Re} \keys{ \hat{A}_{jk} \; \overline{\mathcal{F}_n^{\ast} \widehat{C}_k \brac{ \hat{t}_j} }} \\
& = \sum_{j=1}^{\abs{T}} \PROD{A_{j:}}{P_{j:}}.
\end{align*}
This combined with the following lemma implies~\eqref{eq:P_A}.
\begin{lemma}
Let $M$, $N$ $\in \C^{n_1 \times n_2}$, such that the rows of $N$ have non-zero norm, if 
\begin{align}
& \sum_{j=1}^{n_1} \PROD{M_{j:}}{N_{j:}}  = \sum_{j=1}^{n_1} \normTwo{N_{j:}}, \label{eq:equality_MN}\\
& \max_{1 \leq j \leq n_1} \normTwo{M_{j:}}  \leq 1 \label{eq:bound_M}
\end{align}
then 
\begin{align*}
M_{jk} & = \frac{N_{jk}}{ \normTwo{N_{j:}}}.
\end{align*}
\end{lemma}
\begin{proof}
By ~\eqref{eq:bound_M} and the Cauchy-Schwarz inequality, $\PROD{M_{j:}}{N_{j:}} \leq  \normTwo{N_{j:}}$ for all $j$. If for any of them $\PROD{M_{j:}}{N_{j:}} <  \normTwo{N_{j:}}$ then~\eqref{eq:equality_MN} cannot hold, so
\begin{align}
\label{eq:equality_MN_rows}
\PROD{M_{j:}}{N_{j:}} =  \normTwo{N_{j:}} \quad \text{for all } j.
\end{align}
This directly implies the desired result because for any $j$
\begin{align*}
\normTwo{M_{j:} - \frac{N_{j:}}{\normTwo{N_{j:}}}} ^2 & = \PROD{M_{j:} - \frac{N_{j:}}{ \normTwo{N_{j:}}}}{M_{j:} - \frac{N_{j:}}{\normTwo{N_{j:}}}}\\
& = \normTwo{M_{j:}} - 2 \frac{\PROD{M_{j:}}{N_{j:}}}{\normTwo{N_{j:}}} + 1\\
& = \normTwo{M_{j:}} - 1 = 0. 
\end{align*}
\end{proof}

\subsection{Proof of Lemma~\ref{lemma:cert_common_support}}
\label{proof:cert_common_support}
Any $m$-dimensional measure that is feasible for Problem~\eqref{eq:gTVnormMin} can be written as $X + H$, where $H$ is an $m$-dimensional measure such that $\mathcal{F}_n H_k = 0$ for $1 \leq k \leq m$. This implies that
\begin{align}
\PROD{Q}{H} = \sum_{k=1}^{m} \PROD{Q_k}{H_k} = \sum_{k=1}^{m} \PROD{\mathcal{F}_n^{\ast}C_k}{H_k} = \sum_{k=1}^{m} \PROD{ C_k}{\mathcal{F}_n H_k} = 0.
\end{align}
For any $m$-dimensional measure $\tilde{X}$ composed by $m$ measures $\tilde{X}_k$ with bounded TV norm supported on the unit interval $\mathbb{T}$, we denote its restriction to a subset $S$ of $\mathbb{T}$ by $\tilde{X}_S$ and its restriction to the complement of $S$ by $\tilde{X}_{S^c}$. Since $\normTwo{Q\brac{t}}\leq 1$ for any $t$, we have $\PROD{Q}{X+H_T} \leq \normgTV{X + H_T}$. Similarly, since $\normTwo{Q\brac{t}}<1$ for any $t$ in $T^c$, $\PROD{Q}{H_{T^c}} < \normgTV{H_{T^c}}$. This combined with the fact that the gTV norm is separable and that by construction $\PROD{Q}{X}  = \normgTV{X}$ allows us to conclude that
\begin{align*}
\normgTV{X + H} & = \normgTV{X + H_T} + \normgTV{H_{T^c}}\\
& > \PROD{Q}{X+H_T} +  \PROD{Q}{H_{T^c}}\\
& = \PROD{Q}{X} +  \PROD{Q}{H}\\
& = \normgTV{X},
\end{align*}
as long as $H_{T^c}$ is non zero. Since the measurements are injective on $T$ (the restriction of the measurement operator is a submatrix of a Vandermonde matrix and thus full rank) $H_{T^c}= 0$ implies $H = 0$ and the proof is complete.
\section{Bounds on the interpolation kernel}
\label{sec:bounds_kernel}
\subsection{Upper and lower bounds near the origin in Lemma~\ref{lemma:bounds_near}}
\label{sec:bounds_near}
We use $\fmin$ to denote a lower bound on the cut-off frequency $\fc$. The letter U denotes an upper bound and the letter L denotes a lower bound. The bounds on $K_{\gamma}$ are
\begin{align}
B_{\gamma,0}^{\text{U}} \brac{\tau}& := \max_{A \in \keys{\text{L},\text{U}}^p} \prod_{j=1}^{p}B_{0}^{A_j} \brac{\gamma_j ,\tau}, \notag\\
B_{\gamma,0}^{\text{L}} \brac{\tau} & := \min_{A \in \keys{\text{L},\text{U}}^p} \prod_{j=1}^{p}B_{0}^{A_j} \brac{\gamma_j ,\tau}, \label{eq:B0gamma}
\end{align}
where $\keys{\text{L},\text{U}}^p$ denotes the set of strings of length $p$ where each character is equal to U or L. $B_{\ell}^{\text{U}}$ and $B_{\ell}^{\text{L}}$ for $\ell=0,1,2,3$ are bounds on the Dirichlet kernel $K$ and its derivatives which are defined below. The bounds on $K_{\gamma}^{\brac{1}}$, $K_{\gamma}^{\brac{2}}$ and $K_{\gamma}^{\brac{3}}$, which follow, also depend on these functions.
\begin{align}
B_{\gamma,1}^{\text{U}} \brac{\tau} & := \sum_{i=1}^p \max_{A \in \keys{\text{L},\text{U}}^p} B_{1}^{A_i} \brac{\gamma_i ,\tau} \prod_{j=1,j\neq i}^{p}B_{0}^{A_j} \brac{\gamma_j ,\tau},\notag \\
B_{\gamma,1}^{\text{L}} \brac{\tau} & := \sum_{i=1}^p \min_{A \in \keys{\text{L},\text{U}}^p} B_{1}^{A_i} \brac{\gamma_i ,\tau} \prod_{j=1,j\neq i}^{p}B_{0}^{A_j} \brac{\gamma_j ,\tau} ,\label{eq:B1gamma}
\end{align}
\begin{align}
B_{\gamma,2}^{\text{U}} \brac{\tau} := & \sum_{i=1}^p \max_{A \in \keys{\text{L},\text{U}}^p} B_{2}^{A_i} \brac{\gamma_i ,\tau} \prod_{j=1,j\neq i}^{p}B_{0}^{A_j} \brac{\gamma_j ,\tau}\notag \\
& + \sum_{i=1}^p \sum_{j=1,j\neq i}^{p} \max_{A \in \keys{\text{L},\text{U}}^p} B_{1}^{A_i} \brac{\gamma_i ,\tau} B_{1}^{A_j} \brac{\gamma_j ,\tau} \prod_{k=1,k\notin \keys{i,j} }^{p}B_{0}^{A_k} \brac{\gamma_k ,\tau},\notag \\
B_{\gamma,2}^{\text{L}} \brac{\tau} := & \sum_{i=1}^p \min_{A \in \keys{\text{L},\text{U}}^p} B_{2}^{A_i} \brac{\gamma_i ,\tau} \prod_{j=1,j\neq i}^{p}B_{0}^{A_j} \brac{\gamma_j ,\tau}\notag \\
& + \sum_{i=1}^p \sum_{j=1,j\neq i}^{p} \min_{A \in \keys{\text{L},\text{U}}^p} B_{1}^{A_i} \brac{\gamma_i ,\tau} B_{1}^{A_j} \brac{\gamma_j ,\tau} \prod_{k=1,k\notin \keys{i,j} }^{p}B_{0}^{A_k} \brac{\gamma_k ,\tau} \label{eq:B2gamma},
\end{align}
\begin{align}
B_{\gamma,3}^{\text{U}} \brac{\tau} := & \sum_{i=1}^p \max_{A \in \keys{\text{L},\text{U}}^p} B_{3}^{A_i} \brac{\gamma_i ,\tau} \prod_{j=1,j\neq i}^{p}B_{0}^{A_j} \brac{\gamma_j ,\tau}\notag \\
& + 3\sum_{i=1}^p \sum_{j=1,j\neq i}^{p} \max_{A \in \keys{\text{L},\text{U}}^p} B_{2}^{A_i} \brac{\gamma_i ,\tau} B_{1}^{A_j} \brac{\gamma_j ,\tau} \prod_{k=1,k\notin \keys{i,j} }^{p}B_{0}^{A_k} \brac{\gamma_k ,\tau}\notag \\
& + \sum_{i=1}^p \sum_{j=1,j\neq i}^{p}\sum_{k=1,k\notin \keys{i,j} }^{p}  \max_{A \in \keys{\text{L},\text{U}}^p} B_{1}^{A_i} \brac{\gamma_i ,\tau} B_{1}^{A_j} \brac{\gamma_j ,\tau}B_{1}^{A_k} \brac{\gamma_k ,\tau}  \prod_{l=1,l\notin \keys{i,j,k} }^{p}B_{0}^{A_l} \brac{\gamma_l ,\tau},\notag \\
B_{\gamma,3}^{\text{L}} \brac{\tau} := & \sum_{i=1}^p \min_{A \in \keys{\text{L},\text{U}}^p} B_{3}^{A_i} \brac{\gamma_i ,\tau} \prod_{j=1,j\neq i}^{p}B_{0}^{A_j} \brac{\gamma_j ,\tau}\notag \\
& + 3\sum_{i=1}^p \sum_{j=1,j\neq i}^{p} \min_{A \in \keys{\text{L},\text{U}}^p} B_{2}^{A_i} \brac{\gamma_i ,\tau} B_{1}^{A_j} \brac{\gamma_j ,\tau} \prod_{k=1,k\notin \keys{i,j} }^{p}B_{0}^{A_k} \brac{\gamma_k ,\tau}\notag \\
& + \sum_{i=1}^p \sum_{j=1,j\neq i}^{p}\sum_{k=1,k\notin \keys{i,j} }^{p}  \min_{A \in \keys{\text{L},\text{U}}^p} B_{1}^{A_i} \brac{\gamma_i ,\tau} B_{1}^{A_j} \brac{\gamma_j ,\tau}B_{1}^{A_k} \brac{\gamma_k ,\tau}  \prod_{l=1,l\notin \keys{i,j,k} }^{p}B_{0}^{A_l} \brac{\gamma_l ,\tau}. \label{eq:B3gamma}
\end{align}
These bounds are plotted for $\gamma=\sqbr{\gammaOne,\gammaTwo,\gammaThree}^T$ in Figure~\ref{fig:kernel_gamma_B_bound}.

We separate the bounds on $K$, $B_{0}^{U}$ and $B_{0}^{L}$, into two intervals.
\begin{align}
B_{0}^{\text{U}} \brac{\gamma_0,\tau} := \begin{cases}
B_{\text{near},0}^{\text{U}} \brac{\gamma_0,\tau} \quad &\text{if } 0 \leq \abs{\tau}  \leq \tau_0\brac{\gamma_0}\\
B_{\text{far},0}^{\text{U}} \brac{\gamma_0,\tau} \quad &\text{if } \abs{\tau}  > \tau_0\brac{\gamma_0},
\end{cases}\label{eq:B0}
\end{align}
and do exactly the same for $B_{0}^{\text{L}}$. The intervals depend on the value of $\gamma_0$. In particular:
$\tau_0\brac{\gammaOne}=0.9112$, $\tau_0\brac{\gammaTwo}=0.6615$,
$\tau_0\brac{\gammaThree}=0.5401$.

The bounds near the origin are of the form
\begin{align}
B_{\text{near},0}^{\text{U}} \brac{\gamma_0, \tau} & : = \frac{1-\frac{2\pi^2  \gamma_0^2 \tau^2}{3}+\frac{2\pi^4 \gamma_0^4 \tau^4}{15} \brac{1+\frac{1}{2 \gamma_0 \fmin}}^4 -\frac{4\pi^6 \gamma_0^6 \tau^6}{315} +\frac{2\pi^8 \gamma_0^8 \tau^8}{2835}\brac{1+\frac{1}{2 \gamma_0 \fmin}}^8 }{\brac{1-\frac{2 \pi^2 \gamma_0^2 \tau^2}{3} +\frac{2\pi^4 \gamma_0^4 \tau^4}{15}  -\frac{4\pi^6 \gamma_0^6 \tau^6}{315}}\brac{1- \frac{\pi^2  \tau^2}{6  \fmin^2}}} \frac{\sin \brac{2 \pi \gamma_0  \tau }}{2 \pi \gamma_0  \tau }, \label{eq:Bnear0U} \\
B_{\text{near},0}^{\text{L}} \brac{\gamma_0,\tau} & : = \frac{1-\frac{2\pi^2 \gamma_0^2 \tau^2}{3}\brac{1+\frac{1}{2 \gamma_0 \fmin}}^2+\frac{2\pi^4 \gamma_0^4 \tau^4}{15}  -\frac{4\pi^6 \gamma_0^6 \tau^6}{315}\brac{1+\frac{1}{2 \gamma_0 \fmin}}^6 }{1-\frac{2\pi^2 \gamma_0^2 \tau^2}{3}+\frac{2\pi^4 \gamma_0^4 \tau^4}{15}  -\frac{4\pi^6 \gamma_0^6 \tau^6}{315} +\frac{2\pi^8 \gamma_0^8 \tau^8}{2835}}  \frac{\sin \brac{2 \pi \gamma_0  \tau }}{2 \pi \gamma_0  \tau }. \label{eq:Bnear0L}
\end{align}
Farther from the origin, 
\begin{align}
B_{\text{far},0}^{\text{U}} \brac{\gamma_0,\tau} & : =\frac{\sin \brac{2 \pi \gamma_0  \tau  }}{2 \pi \gamma_0  \tau  } \frac{\brac{1- \mathbbm{1}_{\sin \brac{2 \pi \gamma_0 \tau} <0} \frac{\pi^2  \tau^2}{2 \fmin^2}} }{\brac{1+\mathbbm{1}_{\sin \brac{2 \pi \gamma_0 \tau} < 0} \frac{1}{2 \gamma_0 \fmin}}  } + \mathbbm{1}_{\cos \brac{2 \pi \gamma_0 \tau} \geq 0}\frac{\cos \brac{2 \pi \gamma_0  \tau  }}{1 +2 \gamma_0 \fmin } , \label{eq:Bfar0U} \\
B_{\text{far},0}^{\text{L}} \brac{\gamma_0,\tau} & : =\frac{\sin \brac{2 \pi \gamma_0  \tau  }}{2 \pi \gamma_0  \tau  } \frac{\brac{1- \mathbbm{1}_{\sin \brac{2 \pi \gamma_0 \tau} \geq 0} \frac{\pi^2  \tau^2}{2 \fmin^2}} }{\brac{1+\mathbbm{1}_{\sin \brac{2 \pi \gamma_0 \tau} \geq 0} \frac{1}{2 \gamma_0 \fmin}}  } + \mathbbm{1}_{\cos \brac{2 \pi \gamma_0 \tau} < 0}\frac{\cos \brac{2 \pi \gamma_0  \tau  }}{1+2 \gamma_0 \fmin }, \label{eq:Bfar0L} 
\end{align}
where $\mathbbm{1}_{S}$ is an indicator function that is equal to one when statement $S$ holds and zero otherwise.

The bounds on $K^{\brac{1}}$, $K^{\brac{2}}$ and $K^{\brac{3}}$ are
\begin{align}
B_{1}^{\text{U}} \brac{\gamma_0,\tau} &: = \frac{  \fc \brac{ \cos \brac{2 \pi \gamma_0 \tau} -  B^{\text{L}}_0 \brac{\gamma_0,\tau}} }{ \tau} \brac{ 1 - \mathbbm{1}_{\cos \brac{2 \pi \gamma_0 \tau} \leq  B^{\text{L}}_0 \brac{\gamma_0,\tau}}\frac{\pi^2 \tau^2}{2 \fmin^2} } \notag\\
& \quad - \mathbbm{1}_{\sin \brac{2 \pi \gamma_0 \tau} < 0} \frac{ \pi \fc \sin \brac{2 \pi \gamma_0 \tau} }{\fmin}, \label{eq:B1U}\\
B_{1}^{\text{L}} \brac{\gamma_0,\tau} &: = \frac{  \fc \brac{ \cos \brac{2 \pi \gamma_0 \tau} -  B^{\text{U}}_0 \brac{\gamma_0,\tau} }}{ \tau} \brac{ 1 - \mathbbm{1}_{\cos \brac{2 \pi \gamma_0 \tau} \geq  B^{\text{U}}_0 \brac{\gamma_0,\tau}}\frac{\pi^2 \tau^2}{2 \fmin^2} }\notag\\
& \quad - \mathbbm{1}_{\sin \brac{2 \pi \gamma_0 \tau} \geq 0} \frac{ \pi  \fc \sin \brac{2 \pi \gamma_0 \tau}}{\fmin}, \label{eq:B1L}
\end{align}
\begin{align}
B_{2}^{\text{U}} \brac{\gamma_0,\tau} &: = \frac{  2  \fc^2 \brac{B^{\text{U}}_0 \brac{\gamma_0,\tau}-\cos \brac{2 \pi \gamma_0 \tau} } }{\tau^2 } \brac{ 1 - \mathbbm{1}_{\cos \brac{2 \pi \gamma_0 \tau} \geq  B^{\text{U}}_0 \brac{\gamma_0,\tau}}\frac{\pi^2 \tau^2}{2 \fmin^2} }^2 \notag \\
&\quad  -4\pi^2 \gamma_0^2  \fc^2 B^{\text{L}}_0 \brac{\gamma_0,\tau} +  \mathbbm{1}_{h_{2}^{\text{U}}\brac{\gamma_0,\tau}\geq 0}   \fc^2 \, h_{2}^{\text{U}}\brac{\gamma_0,\tau}, \label{eq:B2U} \\
B_{2}^{\text{L}} \brac{\gamma_0,\tau} &: = \frac{ 2 \fc^2 \brac{B^{\text{L}}_0 \brac{\gamma_0,\tau}-\cos \brac{2 \pi \gamma_0 \tau} } }{\tau^2 } \brac{ 1 - \mathbbm{1}_{\cos \brac{2 \pi \gamma_0 \tau} <  B^{\text{L}}_0 \brac{\gamma_0,\tau}} \frac{\pi^2 \tau^2}{2 \fmin^2} }^2 \notag\\
&\quad  -4\pi^2 \gamma_0^2 \fc^2 B^{\text{U}}_0 \brac{\gamma_0,\tau} +  \mathbbm{1}_{h_{2}^{\text{L}}\brac{\gamma_0,\tau}< 0}  \fc^2 \, h_{2}^{\text{L}} \brac{\gamma_0,\tau},\label{eq:B2L}
\end{align}
where
\begin{align*}
h_{2}^{\text{U}}\brac{\gamma_0,\tau} & = \frac{4\pi^2\gamma_0}{\fmin}  \brac{  \brac{ 1 - \mathbbm{1}_{\sin \brac{2 \pi \gamma_0 \tau} <0}\frac{\pi^2 \tau^2}{2  \fmin^2} } \frac{\sin \brac{2 \pi \gamma_0 \tau}}{2\pi \gamma_0 \tau} - B^{\text{L}}_0 \brac{\gamma_0,\tau}},\\
h_{2}^{\text{L}}\brac{\gamma_0,\tau} & = \frac{4\pi^2\gamma_0}{\fmin}  \brac{  \brac{ 1 - \mathbbm{1}_{\sin \brac{2 \pi \gamma_0 \tau} \geq 0}\frac{\pi^2 \tau^2}{2 \fmin^2} } \frac{\sin \brac{2 \pi \gamma_0 \tau}}{2\pi \gamma_0 \tau} - B^{\text{U}}_0 \brac{\gamma_0,\tau}},
\end{align*}
\begin{align}
B_{3}^{\text{U}}\brac{\gamma_0,\tau} & :=  -4 \pi^2 \gamma_0^2 \fc^2 B^{\text{L}}_1\brac{\gamma_0,\tau}\brac{1 +\mathbbm{1}_{B^{\text{L}}_1\brac{\gamma_0,\tau}<0} \brac{\frac{1}{ \gamma_0 \fmin}-\frac{1}{2 \gamma_0^2\fmin^2}}} \notag \\
& \quad +\frac{2  \fc \, h_{3}^{\text{U}}\brac{\gamma_0,\tau} }{\tau} \brac{1-\mathbbm{1}_{h_{3}^{\text{U}}\brac{\gamma_0,\tau}< 0} \frac{\pi^2 \tau^2}{2  \fmin^2}}, \label{eq:B3U}
\end{align}
\begin{align}
B_{ 3}^{\text{L}}\brac{\gamma_0,\tau} & :=  -4 \pi^2 \gamma_0^2 \fc^2 B^{\text{U}}_1\brac{\gamma_0,\tau}\brac{1 +\mathbbm{1}_{B^{\text{U}}_1\brac{\gamma_0,\tau}\geq 0} \brac{\frac{1}{ \gamma_0 \fmin}-\frac{1}{2 \gamma_0^2 \fmin^2}} } \notag  \\
& \quad +\frac{2  \fc \, h_{3}^{\text{L}} \brac{\gamma_0,\tau} }{\tau}\brac{1-\mathbbm{1}_{h_{3}^{\text{L}}\brac{\gamma_0,\tau}> 0}\frac{\pi^2 \tau^2}{2  \fmin^2}}, \label{eq:B3L}
\end{align}
where the definition of $h_{3}^{\text{U}}\brac{\gamma_0,\tau}$ and $h_{3}^{\text{L}}\brac{\gamma_0,\tau}$ is separated into two intervals:
\begin{align*}
h_{3}^{U}\brac{\gamma_0,\tau} : = \begin{cases}
h_{\text{near},3,a}^{U} \brac{\gamma_0,\tau} \mathbbm{1}_{ h_{\text{near},3,a}^{U} \brac{\gamma_0,\tau} \geq 0 }-h_{\text{near},3,b}^{U} \brac{\gamma_0,\tau} \quad &\text{if } 0 \leq \abs{\tau}  \leq \tau_3\brac{\gamma_0}, \\
h_{\text{far},3}^{U} \brac{\gamma_0,\tau} \quad &\text{if } \abs{\tau}  > \tau_3\brac{\gamma_0},
\end{cases}\\
h_{3}^{L}\brac{\gamma_0,\tau} : = \begin{cases}
h_{\text{near},3,a}^{L} \brac{\gamma_0,\tau} \mathbbm{1}_{ h_{\text{near},3,a}^{L} \brac{\gamma_0,\tau} < 0}-h_{\text{near},3,b}^{L} \brac{\gamma_0,\tau} \quad &\text{if } 0 \leq \abs{\tau}  \leq \tau_3\brac{\gamma_0}, \\
h_{\text{far},3}^{L} \brac{\gamma_0,\tau} \quad &\text{if } \abs{\tau}  > \tau_3\brac{\gamma_0}.
\end{cases}
\end{align*}
In particular, $\tau_3\brac{\gammaOne}=0.4614$, $\tau_3\brac{\gammaTwo}=0.2346$ and $\tau_3\brac{\gammaThree}=0.3051$. Finally, we have  
\begin{align}
h_{ \text{near},3,a}^{\text{U}}\brac{\gamma_0,\tau} & : = \frac{4  \pi^2 \gamma_0 \fc^2}{ \fmin }\brac{ B^{\text{U}}\brac{\gamma_0 ,\tau} - \frac{\sin \brac{2 \pi \gamma_0 \tau }}{2 \pi \gamma_0 \tau} \brac{1 -\mathbbm{1}_{ \sin \brac{2 \pi \gamma_0 \tau}>0} \frac{\pi^2 \tau^2}{2 \fmineq^2} } } , \notag \\
h_{ \text{near},3,a}^{\text{L}}\brac{\gamma_0,\tau} & := \frac{4  \pi^2 \gamma_0 \fc^2}{ \fmin } \brac{ B^{\text{L}}\brac{\gamma_0,\tau} - \frac{\sin \brac{2 \pi \gamma_0 \tau}}{2 \pi \gamma_0 \tau } \brac{1 -\mathbbm{1}_{ \sin \brac{2 \pi \gamma_0 \tau} < 0} \frac{\pi^2 \tau^2}{2 \fmin^2} } }, \notag \\
 h_{ \text{near},3,b}^{\text{U}}\brac{\gamma_0,\tau} & : =  
 \frac{\fc^2\brac{31.3 \pi^4  \gamma_0^4  \tau^2 - 16 \pi^6 \gamma_0^6  \tau^4 \brac{1+\frac{2}{\gamma_0 \fmin}}} \brac{1- \frac{\pi^2  \tau^2}{2 \fmin^2}}^2  }{15\brac{2+\frac{1}{\gamma_0 \fmin}} }, \notag\\
h_{ \text{near},3,b}^{\text{L}}\brac{\gamma_0,\tau} & := 
\frac{2 \pi^4  \gamma_0^4 \fc^2 \tau^2\brac{8.17+\frac{20}{\gamma_0 \fmin}}}{15 \brac{1-\frac{ \pi^2 \tau^2}{6\fmin^2}}^2}, \notag\\
h_{ \text{far},3}^{\text{U}}\brac{\gamma_0, \tau}& : =  \frac{\fc B^{\text{U}}_1 \brac{\gamma_0, \tau} }{ \tau}\brac{1-\mathbbm{1}_{B^{\text{U}}_1\brac{\gamma_0, \tau}< 0}\frac{\pi^2  \tau^2}{2  \fmin^2}}  - B^{\text{L}}_2\brac{\gamma_0, \tau} , \notag\\
h_{ \text{far},3}^{\text{L}}\brac{\gamma_0, \tau}& : = \frac{\fc B^{\text{L}}_1 \brac{\gamma_0, \tau} }{ \tau }\brac{1-\mathbbm{1}_{B^{\text{L}}_1\brac{\gamma_0, \tau}> 0}\frac{\pi^2 \tau^2}{2  \fmin^2}}  - B^{\text{U}}_2\brac{\gamma_0, \tau} .  \notag\\
\end{align}

\begin{figure}
\begin{tabular}{ >{\centering\arraybackslash}m{0.45\linewidth} >{\centering\arraybackslash}m{0.45\linewidth} }
\begin{tikzpicture}[scale=1,spy using outlines=
	{circle, magnification=10, connect spies}]
\begin{axis}[title = $\ell{=}0$,xlabel= $\tau$]
\addplot[blue,line width=0.3pt] file {data_B_gamma_0_U.dat};
\addplot[red,line width=0.3pt] file {data_B_gamma_0_L.dat};
\addplot+[ycomb,only marks,orange,mark size =0.1pt,mark=o] file {data_K_gamma_0_1.dat};
\addplot+[ycomb,only marks,green,mark size =0.1pt,mark=diamond] file {data_K_gamma_0_2.dat};
\addplot+[ycomb,only marks,black,mark size =0.5pt,mark=x,line width = 0.1pt] file {data_K_gamma_0_3.dat};
\coordinate (spypoint) at (axis cs:1.1,0.02);
\coordinate (magnifyglass) at (axis cs:3,0.6);
\end{axis}
\spy [ size=3cm] on (spypoint)
   in node[fill=white] at (magnifyglass);
\end{tikzpicture}  
& 
\begin{tikzpicture}[scale=1,spy using outlines=
	{circle, magnification=10, connect spies}]
\begin{axis}[title = $\ell {=} 1$,xlabel=$\tau$]
\addplot[blue,line width=0.3pt] file {data_B_gamma_1_U.dat};
\addplot[red,line width=0.3pt] file {data_B_gamma_1_L.dat};
\addplot+[ycomb,only marks,orange,mark size =0.1pt,mark=o] file {data_K_gamma_1_1.dat};
\addplot+[ycomb,only marks,green,mark size =0.1pt,mark=diamond] file {data_K_gamma_1_2.dat};
\addplot+[ycomb,only marks,black,mark size =0.5pt,mark=x,line width = 0.1pt] file {data_K_gamma_1_3.dat};
\coordinate (spypoint) at (axis cs:0.495,-1.37);
\coordinate (magnifyglass) at (axis cs:3,-0.75);
\end{axis}
\spy [ size=3cm] on (spypoint)
   in node[fill=white] at (magnifyglass);
\end{tikzpicture}  
\\
\begin{tikzpicture}[scale=1,spy using outlines=
	{circle, magnification=10, connect spies}]
\begin{axis}[title = $\ell {=} 2$,xlabel=$\tau$]
\addplot[blue,line width=0.3pt] file {data_B_gamma_2_U.dat};
\addplot[red,line width=0.3pt] file {data_B_gamma_2_L.dat};
\addplot+[ycomb,only marks,orange,mark size =0.1pt,mark=o] file {data_K_gamma_2_1.dat};
\addplot+[ycomb,only marks,green,mark size =0.1pt,mark=diamond] file {data_K_gamma_2_2.dat};
\addplot+[ycomb,only marks,black,mark size =0.5pt,mark=x,line width = 0.1pt] file {data_K_gamma_2_3.dat};
\coordinate (spypoint) at (axis cs:0.86,2.39);
\coordinate (magnifyglass) at (axis cs:3,-2.6);
\end{axis}
\spy [ size=3cm] on (spypoint)
   in node[fill=white] at (magnifyglass);
\end{tikzpicture}  
&
\begin{tikzpicture}[scale=1,spy using outlines=
	{circle, magnification=10, connect spies}]
\begin{axis}[title = $\ell {=} 3$,xlabel=$\tau$,legend cell align=left,legend columns=3,legend style={draw=none,/tikz/every even column/.append style={column sep=0.4cm}},
legend to name=named]
\addplot[blue,line width=0.3pt,forget plot] file {data_B_gamma_3_U.dat};
\addplot[red,line width=0.3pt,forget plot] file {data_B_gamma_3_L.dat};
\addplot+[ycomb,only marks,orange,mark size =0.1pt,mark=square,forget plot] file {data_K_gamma_3_1.dat};
\addplot+[ycomb,only marks,green,mark size =0.1pt,mark=diamond,forget plot] file {data_K_gamma_3_2.dat};
\addplot+[ycomb,only marks,black,mark size =0.5pt,mark=x,line width = 0.1pt,forget plot] file {data_K_gamma_3_3.dat};
\coordinate (spypoint) at (axis cs:0.4,13.0);
\coordinate (magnifyglass) at (axis cs:3.25,7.5);
\addlegendimage{only marks,orange,mark size =3pt,mark=square*}
\addlegendentry{\;\;$ K_{\gamma}^{\brac{\ell}}\brac{\fc \tau} /\fc^{\ell}  \;\; (\fc \, {=} \, 10^3)$}
\addlegendimage{only marks,green,mark size =4pt,mark=diamond*}
\addlegendentry{\;\;$ K_{\gamma}^{\brac{\ell}}\brac{\fc \tau}  /\fc^{\ell}  \;\;  (\fc{=} \, 5 \: 10^3)$}
\addlegendimage{only marks,black,mark size =4pt,mark=x,line width=1pt}
\addlegendentry{\;\;$ K_{\gamma}^{\brac{\ell}}\brac{\fc \tau} /\fc^{\ell}  \;\;  (\fc \, {=} \, 10^4)$}
\addlegendimage{blue,line width=3pt}
\addlegendentry{$ \quad B_{\gamma,\ell}^{\text{U}} \brac{\tau} /\fc^{\ell} $}
\addlegendimage{red,line width=3pt}
\addlegendentry{$ \quad B_{\gamma,\ell}^{\text{L}} \brac{\tau} /\fc^{\ell}$}
\end{axis}
\spy [ size=3cm] on (spypoint)
   in node[fill=white] at (magnifyglass);
\end{tikzpicture}  
\end{tabular}
\begin{center}
\ref{named}
\end{center}
\caption{Upper and lower bounds on $K_{\gamma}$ and its derivatives.}
\label{fig:kernel_gamma_B_bound}
\end{figure}
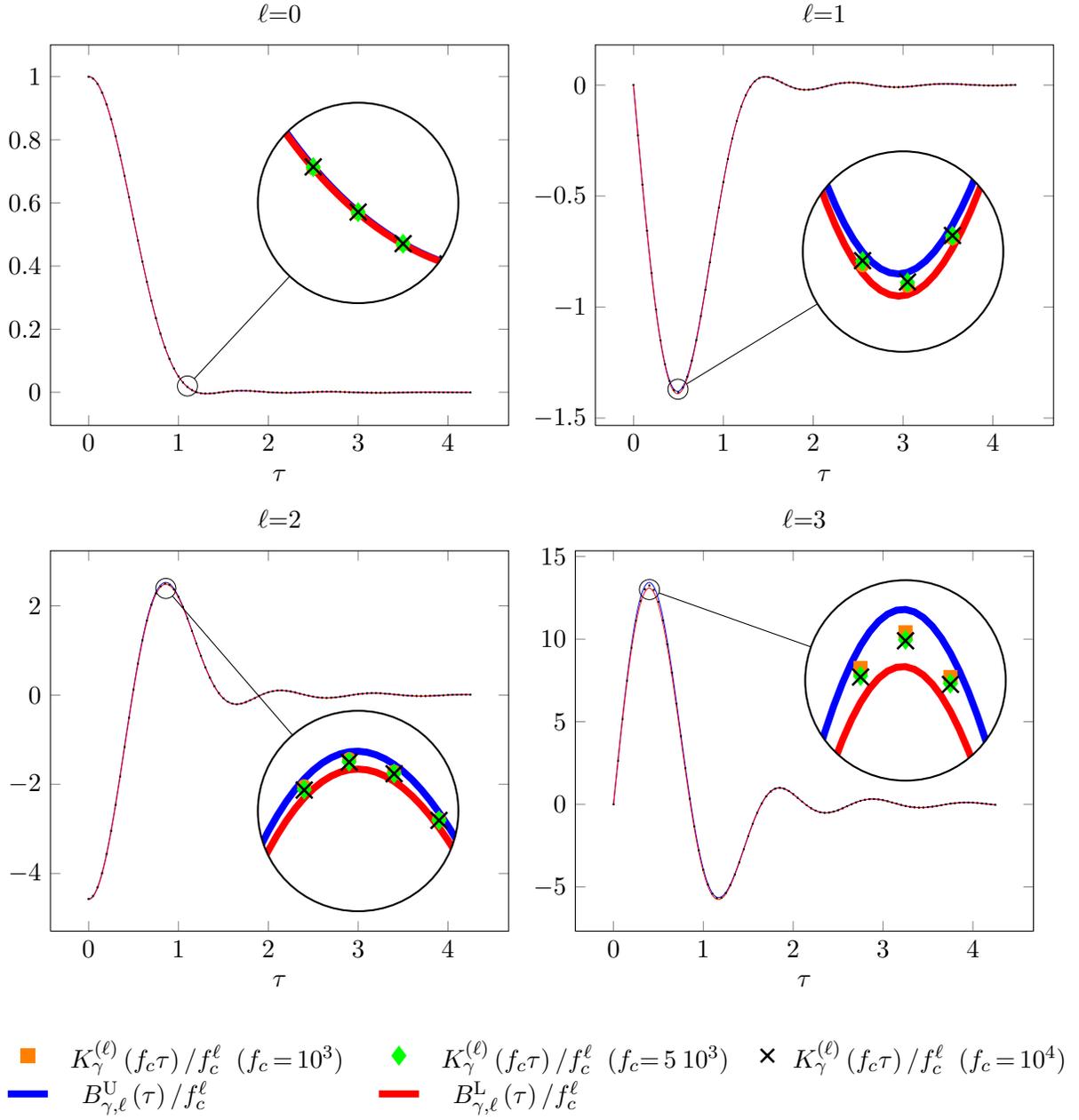

\subsection{Bounds on the tail of the kernel in Lemma~\ref{lemma:tail_bounds}}
\label{sec:tail_bounds}
 The bounds are plotted for $\gamma=\sqbr{\gammaOne,\gammaTwo,\gammaThree}^T$ in Figure~\ref{fig:kernel_gamma_b_bound}. 
\begin{align}
b_{\gamma,0} \brac{\tau}& := \prod_{j=1}^{p}b_{0} \brac{\gamma_j , \tau}, \notag \\
b_{\gamma,1} \brac{\tau} & := \sum_{i=1}^p  b_{1} \brac{\gamma_i ,\tau} \prod_{j=1,j\neq i}^{p}b_{0} \brac{\gamma_j ,\tau}, \notag \\
b_{\gamma,2} \brac{\tau} & := \sum_{i=1}^p b_{2} \brac{\gamma_i ,\tau} \prod_{j=1,j\neq i}^{p}b_{0} \brac{\gamma_j ,\tau} + \sum_{i=1}^p \sum_{j=1,j\neq i}^{p} b_{1} \brac{\gamma_i ,\tau} b_{1} \brac{\gamma_j ,\tau} \prod_{k=1,k\notin \keys{i,j} }^{p}b_{0} \brac{\gamma_k ,\tau}, \notag \\
b_{\gamma,3} \brac{\tau} & :=  \sum_{i=1}^p  b_{3} \brac{\gamma_i ,\tau} \prod_{j=1,j\neq i}^{p}b_{0} \brac{\gamma_j ,\tau} + 3\sum_{i=1}^p \sum_{j=1,j\neq i}^{p} b_{2} \brac{\gamma_i ,\tau} b_{1} \brac{\gamma_j ,\tau} \prod_{k=1,k\notin \keys{i,j} }^{p}b_{0} \brac{\gamma_k ,\tau} \notag\\
& \quad + \sum_{i=1}^p \sum_{j=1,j\neq i}^{p}\sum_{k=1,k\notin \keys{i,j} }^{p}  b_{1}  \brac{\gamma_i ,\tau} b_{1} \brac{\gamma_j ,\tau} b_{1} \brac{\gamma_k ,\tau}  \prod_{l=1,l\notin \keys{i,j} }^{p}b_{0} \brac{\gamma_l ,\tau}, \label{eq:b_bounds}
\end{align}
where
\begin{align*}
b_0 \brac{\gamma_0,\tau} &  := \frac{1 }{2 \pi \gamma_0 \tau  \brac{1-\frac{ \pi^2 \tau^2}{6\fmin^2} }} ,  \\
b_1 \brac{\gamma_0,\tau} &  :=  \frac{\fc \brac{1 +b_0\brac{\gamma_0,\tau}}}{ \tau \brac{1-\frac{ \pi^2  \tau^2}{6  \fmin^2} }},  \\
b_2 \brac{\gamma_0,\tau} &  := 4 \pi^2 \gamma_0^2 \fc^2 b_0\brac{\gamma_0,\tau}\brac{1+\frac{1}{ \gamma_0 \fmin}}+\frac{2 \fc b_1\brac{\gamma_0,\tau} }{ \tau} ,  \\
b_3 \brac{\gamma_0,\tau} &  := 4 \pi^2 \gamma_0^2 \fc^2  b_1 \brac{\gamma_0,\tau} \brac{1+\frac{1}{ \gamma_0 \fmin} } + \frac{2 \fc  }{ \tau} \brac{ b_2 \brac{\gamma_0, \tau}  + \frac{ \fc b_1 \brac{\gamma_0,\tau} }{\tau}} .
\end{align*}
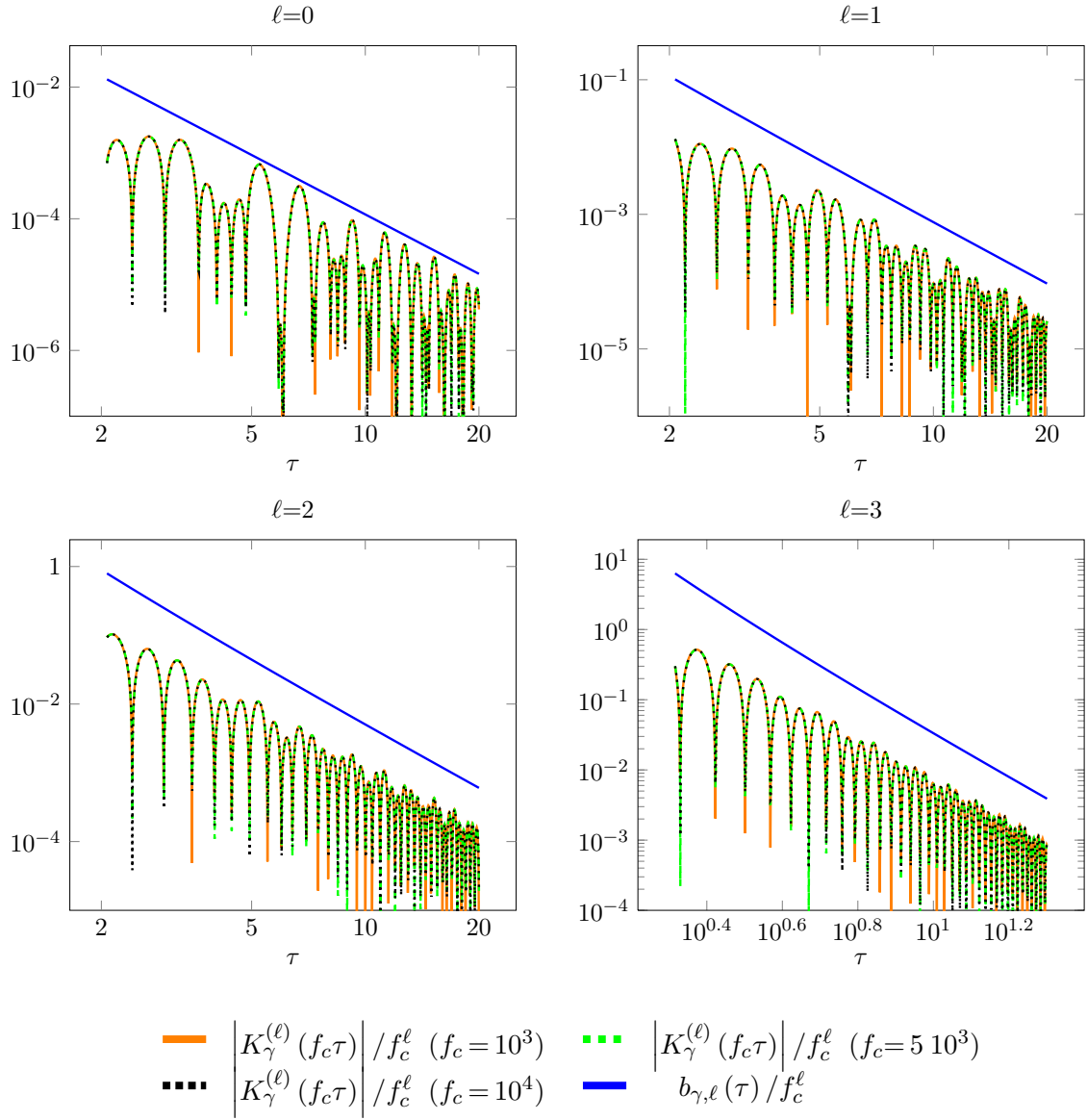
\begin{figure}
\begin{tabular}{ >{\centering\arraybackslash}m{0.45\linewidth} >{\centering\arraybackslash}m{0.45\linewidth} }
\includegraphics{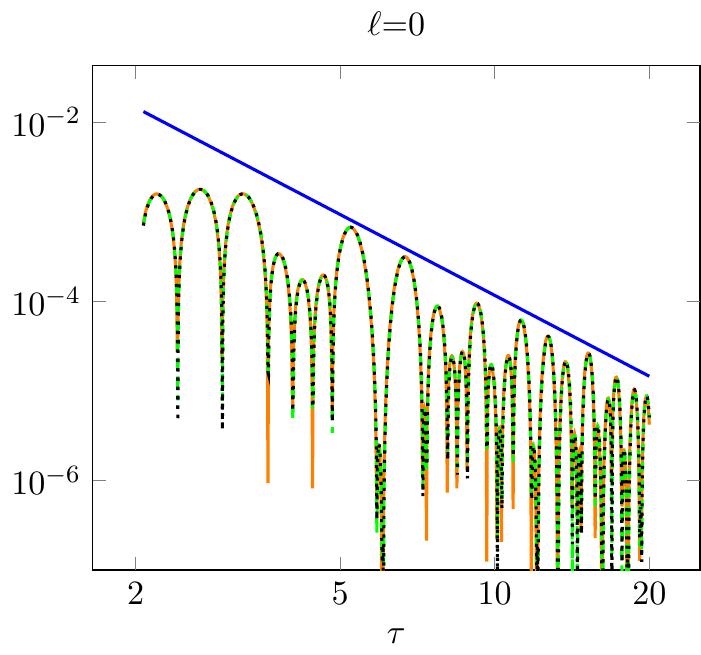}
& 
\includegraphics{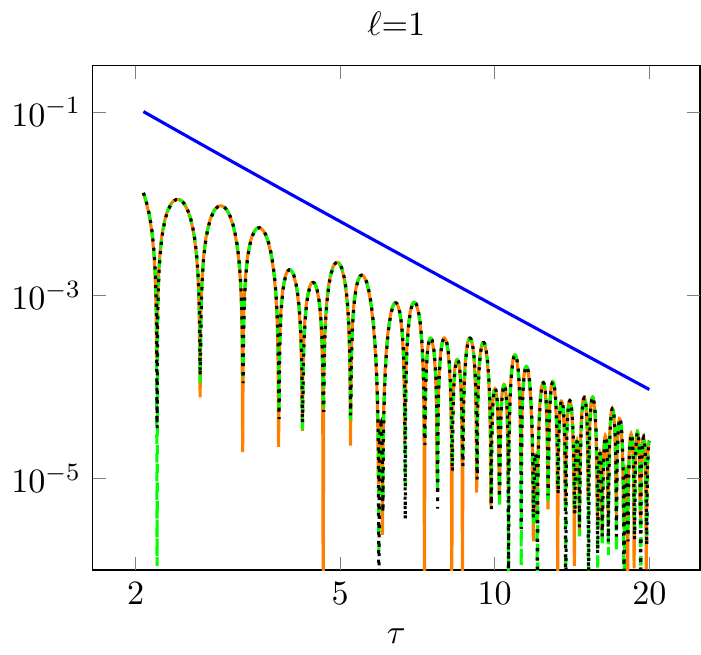}
\\
\includegraphics{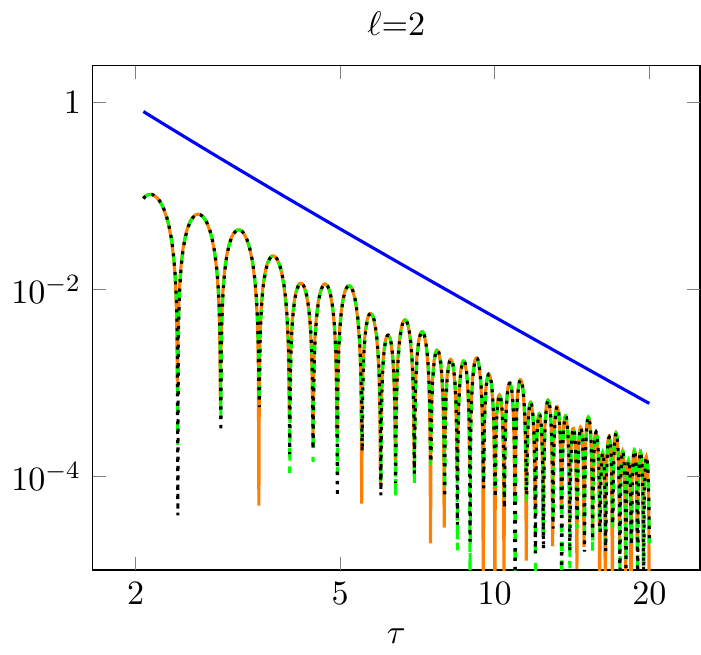}
&
\begin{tikzpicture}[scale=0.9]
\begin{loglogaxis}[title = $\ell {=} 3$,xlabel=$\tau$,legend cell align=left,legend columns=2,legend style={draw=none,/tikz/every even column/.append style={column sep=0.4cm}},
legend to name=named2, ymin=1e-4]
\addplot[blue,line width=1pt,forget plot] file {data_b_gamma_3.dat};
\addplot[orange,line width=1pt,forget plot] file {data_b_K_gamma_3_1.dat};
\addplot[green,dashed,line width=1pt,forget plot] file {data_b_K_gamma_3_2.dat};
\addplot[black,dotted,line width=1pt,forget plot] file {data_b_K_gamma_3_3.dat};
\addlegendimage{orange,line width=3pt}
\addlegendentry{\;\;$\abs{K_{\gamma}^{\brac{\ell}}\brac{\fc \tau}} /\fc^{\ell} \;\; (\fc \, {=} \, 10^3)$}
\addlegendimage{green,dashed,line width=3pt}
\addlegendentry{\;\;$\abs{K_{\gamma}^{\brac{\ell}}\brac{\fc \tau}} /\fc^{\ell} \;\;  (\fc{=} \, 5 \: 10^3)$}
\addlegendimage{black,dotted,line width=3pt}
\addlegendentry{\;\;$\abs{K_{\gamma}^{\brac{\ell}}\brac{\fc \tau}} /\fc^{\ell}  \;\;  (\fc \, {=} \, 10^4)$}
\addlegendimage{blue,line width=3pt}
\addlegendentry{\;\;$ \quad b_{\gamma,\ell} \brac{\tau} /\fc^{\ell} $}
\end{loglogaxis}
\end{tikzpicture}  
\end{tabular}
\begin{center}
\ref{named2}
\end{center}
\caption{Non-increasing upper bounds on the absolute value of $K_{\gamma}$ and its derivatives.}
\label{fig:kernel_gamma_b_bound}
\end{figure}

\section{Proof of the bounds on the interpolation kernel}
\subsection{Proof of Lemma~\ref{lemma:bounds_near}}
\label{proof:bounds_near}
We assume that $t$ is between 0 and 0.5. The bounds for $t$ between -0.5 and 0 follow by symmetry ($K$ and $K^{\brac{2}}$ are even, $K^{\brac{1}}$ and $K^{\brac{3}}$ are odd). The following local bounds, which are derived from Taylor series expansions around the origin, will be used often during the proof. If $u \geq 0$
\begin{align} 
u - \frac{u^3}{6}  & \leq \sin \brac{u }  \leq u ,\label{eq:ineq_sin_1}\\
   u - \frac{u^3}{6} + \frac{u^5}{120} - \frac{u^7}{5040} & \leq \sin \brac{u }  \leq u - \frac{u^3}{6} + \frac{u^5}{120} - \frac{u^7}{5040} +  \frac{u^9}{362880}  ,\label{eq:ineq_sin_2}\\
1 - \frac{u^2}{2}  & \leq \cos \brac{u }  \leq 1 , \label{eq:ineq_cos_1}\\
 1 - \frac{u^2}{2} + \frac{u^4}{24} - \frac{u^6}{720}  &\leq \cos \brac{u }  \leq 1 - \frac{u^2}{2} + \frac{u^4}{24}, \label{eq:ineq_cos_2}
 \end{align}
 and if $0 \leq u \leq \sqrt{2}$
\begin{align} 
u  & \leq \tan \brac{u} \leq \frac{u }{1-\frac{u^2 }{2}}, \label{eq:ineq_tan}\\
u^2  & \leq \sin \brac{u} \tan \brac{u} \leq u^2 +\frac{u^4}{2-u^2}, \label{eq:ineq_sintan}\\ 
u^3 & \leq \sin \brac{u} \tan^2 \brac{u} \leq u^3 +\frac{u^5\brac{4-u^2}}{\brac{2-u^2}^2}. \label{eq:ineq_sinsqtan}
\end{align}

Let us fix a value $0<\gamma_0<1$ and set $f_0=\gamma_0 \, f_c$ to alleviate notation. Henceforth we assume that $\fc \geq \fmin$. By~\eqref{eq:ineq_sin_1} and~\eqref{eq:ineq_tan}
\begin{align}
& \frac{1}{ \pi t }   \leq \frac{1}{\sin \brac{\pi t} }  \leq \frac{1}{ \pi t \brac{1- \frac{\pi^2 f_0^2 t^2}{6 \gamma_0^2\fmin^2}} }, \label{eq:ineq_1oversin}\\
&  \frac{1- \frac{\pi^2 f_0^2 t^2}{2 \gamma_0^2\fmin^2}}{ \pi t }  \leq \frac{1}{\tan \brac{\pi t }} \leq \frac{1 }{\pi t}. \label{eq:ineq_1overtan} 
\end{align}
We begin by deriving bounds on $K\brac{\gamma_0 \,\fc, t}$ near the origin. By~\eqref{eq:ineq_sin_2}
\begin{align*}
\frac{\sin \brac{\brac{2 f_0+1} \pi t}}{2 f_0+1} \frac{2f_0}{\sin \brac{2 \pi f_0 t}} & \leq \frac{1-\frac{\pi^2\brac{2 f_0 + 1}^2 t^2}{6}+\frac{\pi^4\brac{2 f_0 + 1}^4 t^4}{120} -\frac{\pi^6\brac{2 f_0 + 1}^6 t^6}{5040}+\frac{\pi^8\brac{2 f_0 + 1}^8 t^8}{362880}}{1-\frac{2\pi^2f_0^2 t^2}{3}+\frac{2\pi^4f_0^4 t^4}{15}-\frac{4\pi^6f_0^6 t^6}{315}},\\
\frac{\sin \brac{\brac{2 f_0+1} \pi t}}{2 f_0+1} \frac{2f_0}{\sin \brac{2 \pi f_0 t}} & \geq \frac{1-\frac{\pi^2\brac{2 f_0 + 1}^2 t^2}{6}+\frac{\pi^4\brac{2 f_0 + 1}^4 t^4}{120} -\frac{\pi^6\brac{2 f_0 + 1}^6 t^6}{5040}}{1-\frac{2\pi^2f_0^2 t^2}{3}+\frac{2\pi^4f_0^4 t^4}{15}-\frac{4\pi^6f_0^6 t^6}{315}+\frac{2\pi^8f_0^8 t^8}{2835}}.
\end{align*}
Combining this with~\eqref{eq:ineq_1oversin} establishes
\begin{align}
B_{\text{near},0}^{\text{L}} \brac{\gamma_0,\fct} & \leq K(\gamma_0 \,\fc, t)  \leq B_{\text{near},0}^{\text{U}} \brac{\gamma_0,\fct}. \label{eq:B0_bound_near}
\end{align}
To prove
\begin{align}
B_{\text{far},0}^{\text{L}} \brac{\gamma_0,\fct} & \leq K(\gamma_0 \,\fc, t)  \leq B_{\text{far},0}^{\text{U}} \brac{\gamma_0,\fct}, \label{eq:B0_bound_far}
\end{align}
we express the kernel in the following form
\begin{align*}
K\brac{f_0,t}  & = \frac{\sin \brac{2 \pi f_0 t} \cos \brac{\pi t}}{\brac{2 f_0+1} \sin\brac{\pi t}} + \frac{\cos \brac{2 \pi f_0 t}}{\brac{2 f_0+1}},
\end{align*}
and apply~\eqref{eq:ineq_cos_1} and~\eqref{eq:ineq_1overtan}.

Some basic trigonometric identities yield
\begin{align}
K^{(1)}\brac{f_0, t} &= \frac{ \pi }{\sin \brac{\pi t}} \brac{ \cos \brac{\brac{2f_0 + 1} \pi t} - \frac{\sin \brac{\brac{2f_0 + 1} \pi t} \cos \brac{\pi t}}{\brac{2f_0 + 1}\sin \brac{\pi t}}}  \notag \\
  & = \frac{ \pi }{\sin \brac{\pi t}} \brac{ \cos \brac{\brac{2f_0 + 1} \pi t} - K\brac{f_0, t} \cos \brac{\pi t} } \label{eq:ineqK1}\\
   & = \frac{ \pi }{\sin \brac{\pi t}} \brac{ \cos \brac{2 \pi f_0 t}\cos \brac{\pi t}-\sin \brac{2 \pi f_0 t}\sin \brac{\pi t} - K\brac{f_0, t} \cos \brac{\pi t} } \label{eq:ineqK1_2} \\
 & = \pi\brac{ \frac{  \cos \brac{2 \pi f_0 t} -  K\brac{f_0, t} }{\tan \brac{\pi t}} - \sin \brac{2 \pi f_0 t}}. \notag
\end{align}
This implies
\begin{align*}
& K^{(1)}\brac{f_0, t} \leq \pi\brac{ \frac{  \cos \brac{2 \pi f_0 t} - B^{\text{L}}_0 \brac{\gamma_0,f_0 t} }{\tan \brac{\pi t}} - \sin \brac{2 \pi f_0 t}},\\
& K^{(1)}\brac{f_0, t} \geq \pi\brac{ \frac{  \cos \brac{2 \pi f_0 t} -  B^{\text{U}}_0 \brac{\gamma_0,f_0 t} }{\tan \brac{\pi t}} - \sin \brac{2 \pi f_0 t}} .
\end{align*}
Applying~\eqref{eq:ineq_1overtan} then establishes
\begin{align}
B_{1}^{\text{L}} \brac{\gamma_0,\fct} & \leq K^{(1)}(\gamma_0 \,\fc, t)  \leq B_{1}^{\text{U}} \brac{\gamma_0,\fct}. \label{eq:B1_bound}
\end{align}
To bound the second derivative of the kernel we again leverage some trigonometric identities,
\begin{align}
K^{(2)}\brac{f_0, t} &= -\pi^2 K\brac{f_0, t}\brac{\brac{2f_0+1}^2-1} -\frac{2 \pi K^{(1)}\brac{f_0, t} }{\tan \brac{\pi t}} \label{eq:ineqK2} \\
&= -4\pi^2 K\brac{f_0, t}\brac{f_0^2+f_0}-\frac{2 \pi^2 }{\tan \brac{\pi t}}\brac{ \frac{  \cos \brac{2 \pi f_0 t} -  K\brac{f_0, t} }{\tan \brac{\pi t}} - \sin \brac{2 \pi f_0 t}} \notag\\
&= -4\pi^2 f_0^2 K\brac{f_0, t} -\frac{  2 \pi^2 \brac{\cos \brac{2 \pi f_0 t} -  K\brac{f_0, t}} }{\tan^2 \brac{\pi t}} + 4\pi^2 f_0 \brac{ \frac{\sin \brac{2 \pi f_0 t}}{2f_0\tan \brac{\pi t}} - K\brac{f_0, t}}. \notag
\end{align}
This implies
\begin{align*}
K^{(2)}\brac{f_0, t} & \leq -4\pi^2 f_0^2 B^{\text{L}}_0 \brac{\gamma_0,f_0 t} +\frac{  2 \pi^2 \brac{B^{\text{U}}_0 \brac{\gamma_0,f_0 t}-\cos \brac{2 \pi f_0 t} } }{\tan^2 \brac{\pi t}} \\
& \quad+ 4\pi^2 f_0\brac{ \frac{\sin \brac{2 \pi f_0 t}}{2f_0\tan \brac{\pi t}} - B^{\text{L}}_0 \brac{\gamma_0,f_0 t}},\\
K^{(2)}\brac{f_0, t} & \geq -4\pi^2 f_0^2 B^{\text{U}}_0 \brac{\gamma_0,f_0 t} +\frac{  2 \pi^2 \brac{B^{\text{L}}_0 \brac{\gamma_0,f_0 t}-\cos \brac{2 \pi f_0 t} } }{\tan^2 \brac{\pi t}} \\
& \quad+ 4\pi^2 f_0\brac{ \frac{\sin \brac{2 \pi f_0 t}}{2f_0\tan \brac{\pi t}} - B^{\text{U}}_0 \brac{\gamma_0,f_0 t}}.
\end{align*}
Applying~\eqref{eq:ineq_1overtan} yields
\begin{align}
B_{2}^{\text{L}} \brac{\gamma_0,\fct} & \leq K^{(2)}(\gamma_0 \,\fc, t)  \leq B_{2}^{\text{U}} \brac{\gamma_0,\fct}. \label{eq:B2_bound}
\end{align}

Using the fact that the derivative of $1/\tan\brac{u}$ is $-1-1/\tan^2\brac{u}$, we rewrite the third derivative of the kernel as
\begin{align}
K^{(3)}\brac{f_0,t} &= -\pi^2 K^{(1)}\brac{f_0,t}\brac{\brac{2f_0+1}^2-3}-\frac{2\pi }{\tan\brac{\pi t}} \brac{ K^{(2)}\brac{f_0,t} - \frac{\pi  K^{(1)}\brac{f_0,t} }{\tan \brac{\pi t}}}  \label{eq:ineqK3}  \\
&= -2\pi^2 K^{(1)}\brac{f_0,t}\brac{2 f_0^2+2f_0-1} -\frac{2\pi^2 }{\tan\brac{\pi t}} \Bigg( -4\pi K\brac{f_0, t}\brac{f_0^2+f_0}-\frac{3  K^{(1)}\brac{f_0, t} }{\tan \brac{\pi t}} \Bigg) \notag \\
&= -2\pi^2 K^{(1)}\brac{f_0,t}\brac{2 f_0^2+2f_0-1}-\frac{2\pi^3 }{\tan\brac{\pi t}} \Bigg( 4 f_0\brac{ \frac{\sin \brac{2 \pi f_0 t}}{2f_0\tan \brac{\pi t}} - K\brac{f_0,t}} + z(t) \Bigg).\notag
\end{align}
To bound $z(t)$ we perform some algebraic manipulations and then apply~\eqref{eq:ineq_sin_2},~\eqref{eq:ineq_cos_2},~\eqref{eq:ineq_sintan} and~\eqref{eq:ineq_sinsqtan},
\begin{align*}
z(t) & =  -4 f_0^2 K\brac{f_0,t} + \frac{ \sin \brac{2 \pi f_0 t}  }{\tan \brac{\pi t}}  - \frac{ 3  \brac{\cos \brac{2 \pi f_0 t} -  K\brac{f_0,t}} }{\tan^2\brac{\pi t}}\\
& = \frac{ \brac{-4 f_0^2 +2 f_0+1}\sin \brac{2 \pi f_0 t}  }{\brac{2 f_0+1}\tan \brac{\pi t}} -\frac{4 f_0^2 \cos \brac{2\pi f_0 t}}{2 f_0+1} \\
& \quad  - \frac{ 3  \brac{\brac{2 f_0+1}\cos \brac{2 \pi f_0 t} \sin\brac{\pi t} - \sin \brac{\brac{2 f_0+1} \pi t} } }{\brac{2 f_0+1} \sin\brac{\pi t}\tan^2\brac{\pi t}} \\
& = \frac{ \sin \brac{2 \pi f_0 t} \brac{ \brac{-4 f_0^2 +2 f_0+1}\sin\brac{\pi t}\tan\brac{\pi t}    + 3 \cos \brac{\pi t}}}{\brac{2 f_0+1} \sin\brac{\pi t}\tan^2\brac{\pi t}}\\
&\quad + \frac{ \cos \brac{2 \pi f_0 t} \brac{ -4 f_0^2 \sin\brac{\pi t} \tan^2 \brac{\pi t} - 6f_0 \sin\brac{\pi t}  }}{\brac{2 f_0+1} \sin\brac{\pi t}\tan^2\brac{\pi t}}\\
& = \frac{ \brac{2 \pi f_0 t - \frac{4\pi^3 f_0^3 t^3}{3} +\frac{4\pi^5 f_0^5 t^5}{15}-\epsilon_1} \brac{ \brac{-4 f_0^2 +2 f_0+1}\brac{ \pi^2 t^2 + \epsilon_2}   +  3 \brac{1-\frac{\pi^2 t^2}{2} + \epsilon_3 } } }{\brac{2 f_0+1} \sin\brac{\pi t}\tan^2\brac{\pi t}}\\
&\quad + \frac{ \brac{1-2\pi^2 f_0^2t^2 + \frac{2\pi^4 f_0^4 t^4}{3} - \epsilon_4 }\brac{ -4 f_0^2 \brac{\pi^3 t^3 +\epsilon_5} - 6f_0 \brac{\pi t - \frac{\pi^3 t^3}{6} +\epsilon_6 } }  }{\brac{2 f_0+1} \sin\brac{\pi t}\tan^2\brac{\pi t}}\\
& = \frac{4 \pi^5  f_0^5 t^5\brac{8+\frac{20}{f_0}-\frac{5}{8f_0^2} +z_0\brac{f_0,t,\epsilon_1,\epsilon_2,\epsilon_3,\epsilon_4,\epsilon_5,\epsilon_6} } -8 \pi^7  f_0^7 t^7\brac{ 2+ \frac{4}{f_0}-\frac{1}{f_0^2}  } }{15\brac{2 f_0+1} \sin\brac{\pi t}\tan^2\brac{\pi t}},
\end{align*}
where 
\begin{align*}
\begin{aligned}
\abs{ \epsilon_1} & \leq  \frac{8\pi^7 f_0^7 t^7 }{315}, \qquad & 
\abs{ \epsilon_2} & \leq  \pi^4 t^4, & \qquad 
\abs{  \epsilon_3} & \leq \frac{\pi^4 t^4}{24}, \\
\abs{ \epsilon_4} & \leq  \frac{4\pi^6 f_0^6 t^6}{45}, \qquad & 
\abs{\epsilon_5} & \leq 4 \pi^5 t^5, \qquad & 
\abs{\epsilon_6} & \leq \frac{ \pi^5 t^5}{120}. \end{aligned}
\end{align*}
If $\fct \leq \frac{1}{2}$ and $\fmin \geq 10^3$ some straightforward, yet tedious, computations yield
\begin{align*}
\abs{z_0\brac{f_0,t,\epsilon_1,\epsilon_2,\epsilon_3,\epsilon_4,\epsilon_5,\epsilon_6}} & \leq 0.17,
\end{align*}
which together with~\eqref{eq:ineq_1oversin} and~\eqref{eq:ineq_1overtan} implies\footnote{In this step we assume that $\gamma_0\geq 1/32 \fc$, which holds for the values of $\gamma_0$ we are interested in. Otherwise we would have to add a term $- 5 /8f_0^2$ in the first factor of the numerator of the lower bound. For the upper bound, we assume $\gamma_0\geq 1/2 \fc^2$. Similarly, this holds for the values of $\gamma_0$ that we consider, but we could add an extra term to the numerator of the right hand side if this were not the case.}
\begin{align*}
\frac{\brac{31.3 \pi^2  f_0^4 t^2 - 16 \pi^4  f_0^6 t^4 \brac{1+\frac{2}{\gamma_0 \fmin}}} \brac{1- \frac{\pi^2 f_0^2 t^2}{2 \gamma_0^2\fmin^2}}^2  }{15\brac{2+\frac{1}{\gamma_0 \fmin}} 
} \leq z\brac{t}   \leq \frac{2 \pi^2  f_0^4 t^2\brac{8.17+\frac{20}{\gamma_0 \fmin}}}{15 \brac{1-\frac{ \pi^2 f_0^2 t^2}{6\gamma_0^2 \fmin^2}}^2}.
\end{align*}
This establishes
\begin{align}
B_{3}^{\text{L}} \brac{\gamma_0,\fct} & \leq K^{(3)}(\gamma_0 \,\fc, t)  \leq B_{3}^{\text{U}} \brac{\gamma_0,\fct} \label{eq:B3_bound}
\end{align}
for $0 \leq \fct \leq \tau_3\brac{\gamma_0}$. For $\fct \geq \tau_3\brac{\gamma_0}$ the bounds are obtained by combining~\eqref{eq:ineqK3} with~\eqref{eq:ineq_1overtan}.

The bounds~\eqref{eq:B0_bound_near},~\eqref{eq:B0_bound_far},~\eqref{eq:B1_bound},~\eqref{eq:B2_bound} and~\eqref{eq:B3_bound} combined with the identities
\begin{align}
K^{ \brac{ 1} }_{\gamma}\brac{ t} & := \sum_{i=1}^p K^{ \brac{ 1} }\brac{\gamma_i \fc,t}   \prod_{j \neq i} K\brac{\gamma_j \fc,t}, \notag\\  
K^{ \brac{2} }_{\gamma}\brac{ t} & := \sum_{i=1}^p K^{ \brac{ 2} }\brac{\gamma_i \fc,t}   \prod_{j \neq i} K\brac{\gamma_j \fc,t} +\sum_{i=1}^p K^{ \brac{ 1} }\brac{\gamma_i \fc,t} \sum_{j\neq i} K^{ \brac{ 1} }\brac{\gamma_j \fc,t}  \prod_{k \notin \keys{i,j}}^p K\brac{\gamma_k \fc,t},\notag\\  
K^{ \brac{3} }_{\gamma}\brac{ t} & := \sum_{i=1}^p K^{ \brac{3} }\brac{\gamma_i \fc,t}   \prod_{j \neq i} K\brac{\gamma_j \fc,t}  +3\sum_{i=1}^p K^{ \brac{ 2} }\brac{\gamma_i \fc,t} \sum_{j\neq i} K^{ \brac{ 1} }\brac{\gamma_j \fc,t}  \prod_{k \notin \keys{i,j}} K\brac{\gamma_k \fc,t} \notag\\
& \quad +\sum_{i=1}^p K^{ \brac{ 1} }\brac{\gamma_i \fc,t} \sum_{j\neq i} K^{ \brac{ 1} }\brac{\gamma_j \fc,t} \sum_{k \notin \keys{i,j}} K^{ \brac{ 1} }\brac{\gamma_k \fc,t}  \prod_{l \notin \keys{i,j,k}} K\brac{\gamma_l \fc,t} \label{eq:derivatives}
\end{align}
yield
\begin{align}
B_{\gamma,\ell}^{\text{L}} \brac{\fct} & \leq K_{\gamma}^{(\ell)}( t )  \leq B_{\gamma,\ell}^{\text{U}} \brac{\fct} . \label{eq:KB_bound}
\end{align}
Finally, we control the deviation of the kernel and its derivatives from their values on a grid. Bernstein's inequality allows to bound the $\ell$th derivative of a trigonometric polynomial.
\begin{theorem}[Theorem 4 \cite{bernstein_inequality}]
Let $P$ be a trigonometric polynomial of order $N$ 
\begin{align}
P\brac{\theta} = \sum_{j=1}^{N} a_j \cos \brac{j\theta} +  b_j \sin \brac{j\theta}, \label{eq:trig_pol}
\end{align}
such that $\abs{P\brac{\theta}} \leq 1$ for all real $\theta$. Then, 
\begin{align*}
\abs{P^{\brac{\ell}}\brac{\theta}} \leq  N^{\ell} \quad \text{for all real } \theta.
\end{align*}
\end{theorem}
Applying the theorem to $K_{\gamma}( t )$, which can be written as $P\brac{2\pi t}$ for $P$ of the form~\eqref{eq:trig_pol} for $N=\fc$ and has magnitude bounded by 1, we obtain
\begin{align*}
\abs{K_{\gamma}^{(\ell)}( t ) } \leq \brac{2 \pi}^{\ell} \fc^{\ell}.
\end{align*}
As a result, for any $t$ such that $\abs{t-t^{\ast}}\leq \epsilon/\fc$
\begin{align*}
\abs{K_{\gamma}^{(\ell)}( t ) - K_{\gamma}^{(\ell)}( t^{\ast} )} \leq \brac{2 \pi}^{\ell+1} \fc^{\ell} \epsilon
\end{align*}
by the mean-value theorem. Combining this with~\eqref{eq:KB_bound} completes the proof.

\subsection{Proof of Lemma~\ref{lemma:tail_bounds}}
\label{proof:tail_bounds}
By~\eqref{eq:ineq_1oversin}
\begin{align}
\abs{K\brac{f_0,t}}& = \frac{\abs{\sin \brac{\brac{2 f_0+1} \pi t} }}{\brac{2 f_0+1}\abs{\sin \brac{\pi t}}}  \label{eq:tailbound0} \\
& \leq  \frac{1 }{2 \pi f_0 t  \brac{1-\frac{ \pi^2 f_0^2 t^2}{6\gamma_0^2\fmin^2} }} \notag.
\end{align}
As long as $ t < \frac{\sqrt{2} \fmin}{\pi \fc }$ this bound is decreasing because the derivative of the denominator is positive. 
By~\eqref{eq:ineqK1} and~\eqref{eq:ineq_1oversin}, we have
\begin{align}
\abs{K^{(1)}\brac{f_0, t} } & \leq \frac{ \pi  \brac{ \abs{\cos \brac{\brac{2f_0 + 1} \pi t}} + \abs{K\brac{f_0, t} }\abs{\cos \brac{\pi t}} } }{\abs{\sin \brac{\pi t}}} \label{eq:tailbound1} \\
& \leq \frac{1 +b_0\brac{\gamma_0,\fc t}}{ t \brac{1-\frac{ \pi^2 f_0^2 t^2}{6 \gamma_0^2 \fmin^2} }}. \notag
\end{align}
The bound is again decreasing if  $ t < \frac{\sqrt{2} \fmin}{\pi \fc }$ because the derivative of the denominator is positive and $b_0$ is decreasing. By~\eqref{eq:ineqK2} and~\eqref{eq:ineq_1overtan}
\begin{align}
\abs{K^{(2)}\brac{f_0, t} } & \leq \pi^2 \abs{K\brac{f_0, t}}\brac{\brac{2f_0+1}^2-1}+\frac{2 \pi \abs{K^{(1)}\brac{f_0, t} }}{\abs{\tan \brac{\pi t}}}  \label{eq:tailbound2}\\
& \leq  4 \pi^2 f_0^2 b_0\brac{\gamma_0,\fc t}\brac{1+\frac{1}{ \gamma_0 \fmin}}+\frac{2  b_1\brac{\gamma_0,\fc t} }{  t}. \notag
\end{align}
If $ t < \frac{\sqrt{2} \fmin}{\pi \fc }$ the bound is decreasing because so are $b_0$ and $b_1$. The same holds for the next bound, which results from combining~\eqref{eq:ineqK3} and~\eqref{eq:ineq_1overtan},
\begin{align}
\abs{K^{(3)}\brac{f_0, t} } &\leq  \pi^2 \abs{K^{(1)}\brac{f_0,t}}\brac{\brac{2f_0+1}^2-3}+\frac{2\pi }{\tan\brac{\pi t}} \brac{ \abs{K^{(2)}\brac{f_0,t}} + \frac{\pi  \abs{K^{(1)}\brac{f_0,t}} }{\abs{ \tan \brac{\pi t}}}}  \label{eq:tailbound3} \\
 &\leq 
 4 \pi^2 f_0^2  b_1 \brac{\gamma_0,\fc t} \brac{1+\frac{1}{ \gamma_0 \fmin} } + \frac{2  }{ t} \brac{ b_2 \brac{\gamma_0, \fc t}  + \frac{  b_1 \brac{\gamma_0,\fc t} }{ t}}. \notag
\end{align}
These inequalities together with~\eqref{eq:derivatives} establish the bounds in the lemma, which are decreasing in $\tau$ if $ \tau \leq 450 < \frac{\sqrt{2} \fmin}{\pi }$ (which implies $ t < \frac{\sqrt{2} \fmin}{\pi \fc }$ if $\fmin \geq 10^3$) because they are products of positive decreasing functions.
\subsection{Proof of Lemma~\ref{lemma:kernelsum}}
\label{proof:kernelsum}
We bound the first term of the decomposition
\begin{align*}
\sum_{t_j \in \,  T\setminus \{0\}} \abs{K^{\brac{\ell}}_{\gamma} \brac{t-t_j}} & =\sum_{t_j \in \, T \, \cap  \brac{0,\frac{1}{2}} } \abs{K^{\brac{\ell}}_{\gamma} \brac{t-t_j}}+\sum_{t_j \in \, T \, \cap  \left[-\frac{1}{2},0\right) } \abs{K^{\brac{\ell}}_{\gamma} \brac{t-t_j}} ;
\end{align*}
the argument for the second term is almost identical. 

Without loss of generality we assume that the elements of $T \, \cap  \brac{0,\frac{1}{2}} $ are ordered so that $0<t_1 < t_2<\dots<1/2$. 
By Corollary~\ref{cor:boundinf}
\begin{align*}
\abs{K^{\brac{\ell}}_{\gamma} \brac{t-t_j}} & \leq B_{\gamma,\ell}^{ \infty } \brac{u,\epsilon} 
\end{align*}
for any $u$ such that $ \fc\brac{t_j-t}-\epsilon \leq u \leq \fc\brac{t_j-t}+\epsilon$. Maximizing this quantity on an equispaced grid $\mathcal{G}_{j,\tau }$ with separation $\epsilon$ that covers the interval~$\sqbr{j\taumin-\tau,\brac{j+4}\taumin}$ for $\tau - \epsilon \leq \fc t \leq \tau  $ yields a bound that is valid for all values of $t_j$ such that $ t_j-t  \leq \brac{j+4} \Deltamin$, since $ t_j-t \geq j \Deltamin -t  $ due to the minimum-separation condition. If $t_j-t > \brac{j+4} \Deltamin $, then by Lemma~\ref{lemma:tail_bounds}
\begin{align*}
 \abs{K^{\brac{\ell}}_{\gamma} \brac{t-t_j}} & \leq  b_{\gamma,\ell} \brac{\brac{j+4} \taumin } .
\end{align*}
Taking the maximum of the two bounds allows to control the contribution of the kernel (or the $\ell$th derivative of the kernel) centered at $t_j$. We take advantage of this for $1 \leq j \leq \jnear$ and then apply the coarser bound
\begin{align*}
 \abs{K^{\brac{\ell}}_{\gamma} \brac{t-t_j}} & \leq  b_{\gamma,\ell} \brac{ \brac{j-1/2}  \taumin } 
\end{align*}
 for $\jnearpone \leq j \leq \jzero$. This bound follows from Lemma~\ref{lemma:tail_bounds} and the fact that $t_j -t \geq \brac{j-1/2} \Deltamin $ by the minimum-separation condition. The remaining terms can be controlled using the following lemma, which we prove in Section~\ref{proof:tail_bounds_inf}.\begin{lemma}
\label{lemma:tail_bounds_inf}
For $\fc \geq \minfc$, $\ell \in {\keys{0,1,2,3}}$, $\gamma=\sqbr{\gammaOne,\gammaTwo,\gammaThree}^T$ and $ 80 / \fc  \leq \abs{ t } \leq 1/2$
\begin{align*}
\abs{ K^{\brac{\ell}}_{\gamma}\brac{t}} \leq b_{\gamma,\ell}^{\infty}\brac{t} & := \frac{\brac{2\pi \fc}^{\ell}}{ \prod_{i=1}^3 \gamma_i} \brac{\frac{1.1}{4 \fc t }}^3. 
\end{align*}
\end{lemma}
By the lemma,
\begin{align}
 \sum_{t_j \in \, T \, \cap  \brac{\jzero \Deltamin,\frac{1}{2}} } \abs{K^{\brac{\ell}}_{\gamma} \brac{t-t_j}} & \leq \sum_{k=\jzero}^{\infty} b_{\gamma,\ell}^{\infty} \brac{k \taumin} \notag\\ 
 & = \sum_{k=1}^{\infty} b_{\gamma,\ell}^{\infty} \brac{k\taumin} - \sum_{k=1}^{\jzero} b_{\gamma,\ell}^{\infty} \brac{ k\taumin} \notag\\ 
 & = \frac{\brac{2\pi \fc}^{\ell}}{ \prod_{i=1}^p \gamma_i} \brac{\frac{1.1}{4 \taumin }}^p \brac{\sum_{k=1}^{\infty} \frac{1}{k^p}-\sum_{k=1}^{\jzero} \frac{1}{k^p}} :=  C_{\ell}. \label{eq:Cl}
\end{align}
The infinite series $\sum_{k=1}^{\infty} k^{-p}$ is equal to the Riemann zeta function $\zeta\brac{p}$, so we can easily compute $C_{\ell}$ for the desired value of $\gamma$ and conclude that
\begin{align*}
\sum_{t_j \in \, T \, \cap  \brac{0,\frac{1}{2}} } \abs{K^{\brac{\ell}}_{\gamma} \brac{t-t_j}} & \leq H\brac{\tau}.
\end{align*}

\subsection{Proof of Lemma~\ref{lemma:tail_bounds_inf}}
\label{proof:tail_bounds_inf}
In this proof $p:=3$ (the lemma also holds if we convolve a different number of Dirichlet kernels). To obtain bounds that are valid for large $t$ we use a simple lower bound on the sine function that holds between 0 and $\frac{\pi}{2}$
\begin{align*} 
2t \leq \sin \brac{\pi t}. 
\end{align*}
Combining it with~\eqref{eq:ineq_1oversin} and the bounds~\eqref{eq:tailbound0},~\eqref{eq:tailbound1}, \eqref{eq:tailbound2} and~\eqref{eq:tailbound3} on the magnitude of the kernel and its derivatives from Section~\ref{proof:tail_bounds} yields
\begin{align*}
\abs{K\brac{f_0,t}}& \leq \frac{1}{\brac{2 f_0+1} 2t} \leq \frac{1}{4f_0t} , \\
\abs{K^{(1)}\brac{f_0, t} } & \leq \frac{\pi}{2t} \brac{1+\frac{1}{4f_0t}}  , \\
\abs{K^{(2)}\brac{f_0, t} } & \leq  \frac{ \pi^2 f_0}{t}\brac{1+\frac{1}{ \gamma_0 \fmin}}+\frac{\pi}{t^2}\brac{1+\frac{1}{4 f_0 t}} \leq  \frac{ \pi^2 f_0}{t}\brac{1+\frac{1}{ \gamma_0 \fmin} + \frac{1+\frac{1}{4 f_0 t}}{\pi f_0 t}} , \\
\abs{K^{(3)}\brac{f_0, t} } & \leq \frac{2 \pi^3 f_0^2}{t} \Bigg( \brac{1 + \frac{1}{4 f_0 t_0}}\brac{1+\frac{1}{\gamma_0 \fmin}} + \frac{1}{\pi f_0 t} \brac{1+\frac{1}{ \gamma_0 \fmin} + \frac{1+\frac{1}{4 f_0 t}}{\pi f_0 t}}\\
& \quad + \frac{1}{2\pi^2 f_0^2 t^2} \brac{1+\frac{1}{4f_0 t}} \Bigg).
\end{align*}
If $\gamma_0 \geq 0.25 $ then $1/f_0 t \leq 0.05$ for $ t \geq \frac{ 80 }{ \fc }$, which allows to coarsely summarize these inequalities as
\begin{align*}
\abs{K^{(\ell)}\brac{f_0, t} } & \leq \frac{1.1 \, 2^{\ell-2}\pi^\ell f_0^{\ell-1}}{t}
\end{align*}
for $j \in {0,1,2,3}$ and $\fmin \geq 10^3$. Together with~\eqref{eq:derivatives} 
this implies
\begin{align*}
\abs{K_{\gamma}\brac{ t}} & \leq \brac{\frac{1.1}{4 \fc t }}^p \prod_{i=1}^p\frac{1}{ \gamma_i} ,\\
\abs{K^{ \brac{ 1} }_{\gamma}\brac{ t} } & \leq 2\pi \fc \brac{\frac{1.1}{4 \fc t }}^p \sum_{i=1}^p\prod_{j=1,j\neq i}^{p} \frac{1}{\gamma_j},\\
\abs{K^{ \brac{2} }_{\gamma}\brac{ t} }& \leq  \brac{2\pi \fc}^2 \brac{\frac{1.1}{4 \fc t }}^p  \sum_{i=1}^p \brac{   \frac{\gamma_i}{ \prod_{j=1,j\neq i}^p  \gamma_j} + \sum_{j=1,j\neq i}^p  \prod_{k=1,k\neq i,j}^p \frac{1}{ \gamma_k}} ,\\
\abs{K^{ \brac{ 3} }_{\gamma}\brac{ t} } & \leq \brac{2\pi \fc}^3  \brac{\frac{1.1}{4 \fc t }}^p  \sum_{i=1}^p  \brac{   \frac{\gamma_i^2}{\prod_{j=1,j\neq i}^p \gamma_j} + \sum_{j=1,j\neq i}^p \frac{3\gamma_i}{\prod_{k=1,k\neq i,j}^p \gamma_k}+ \sum_{k=1,k\neq i,j}^p \prod_{m=1,m\neq i,j,k}^p\frac{1}{ \gamma_m} }.
\end{align*}
Since by definition $\sum_{j=1}^p \gamma_j=1$, the following identities complete the proof,
\begin{align*}
\sum_{i=1}^p\prod_{j=1,j\neq i}^{p} \frac{1}{\gamma_j}  & = \sum_{i=1}^p \frac{\gamma_i}{\prod_{j=1}^{p} \gamma_j}  =\frac{1}{\prod_{j=1}^{p} \gamma_j},
\end{align*}
\begin{align*}
\sum_{i=1}^p \brac{ \frac{\gamma_i}{ \prod_{j=1,j\neq i}^p  \gamma_j} +  \sum_{j=1,j\neq i}^p \prod_{k=1,k\neq i,j}^p \frac{1}{ \gamma_k} } & = \frac{1}{\prod_{j=1}^{p} \gamma_j} \sum_{i=1}^p \brac{ \gamma_i^2 + \gamma_i \sum_{j=1,j\neq i}^p \gamma_j } \\
&=  \frac{1}{\prod_{j=1}^{p} \gamma_j}\sum_{i=1}^p \gamma_i^2 + \gamma_i \brac{1-\gamma_i}=\frac{1}{\prod_{j=1}^{p} \gamma_j},\\
\end{align*}
\begin{align*}
&\sum_{i=1}^p \brac{ \frac{\gamma_i^2}{\prod_{j=1,j\neq i}^p \gamma_j} + \sum_{j=1,j\neq i}^p \brac{\frac{3\gamma_i}{\prod_{k=1,k\neq i,j}^p \gamma_k}+ \sum_{k=1,k\neq i,j}^p \prod_{m=1,m\neq i,j,k}^p\frac{1}{ \gamma_m}} } \\
&\qquad \qquad  = \sum_{i=1}^p  \brac{ \frac{\gamma_i^2}{\prod_{j=1,j\neq i}^p \gamma_j} + \sum_{j=1,j\neq i}^p  \brac{\frac{2\gamma_i+\gamma_j}{\prod_{k=1,k\neq i,j}^p \gamma_k}+ \sum_{k=1,k\neq i,j}^p \prod_{m=1,m\neq i,j,k}^p\frac{1}{ \gamma_m}}}  \\
&\qquad \qquad  = \frac{1}{\prod_{j=1}^{p} \gamma_j} \sum_{i=1}^p  \brac{ \gamma_i^3 + \sum_{j=1,j\neq i}^p  \brac{ 2\gamma_i^2 \gamma_j+\gamma_i \gamma_j^2 +\sum_{k=1,k\neq i,j}^p \gamma_i \gamma_j \gamma_k }}\\
&\qquad \qquad  = \frac{1}{\prod_{j=1}^{p} \gamma_j} \sum_{i=1}^p  \brac{ \gamma_i^3 + \brac{\gamma_i^2+\gamma_i} \sum_{j=1,j\neq i}^p   \gamma_j }=\frac{1}{\prod_{j=1}^{p} \gamma_j}.
\end{align*}

\subsection{Proof of Lemma~\ref{lemma:bound_2der}}
\label{proof:bound_2der}
By the fundamental theorem of calculus, if $t'$ is such that $ j \epsilon \leq t' \leq \brac{j +1}\epsilon$ for some integer $j$
\begin{align*}
 F'\brac{ t' }& = F'\brac{0} + \int_{0}^{t'} F''\brac{u} \text{d}u\\
 &   =  \sum_{l=0}^{j-1} \int_{l \epsilon}^{\brac{l+1}\epsilon } F''\brac{u} \text{d}u +  \int_{ j \epsilon }^{ t'  } F''\brac{u} \text{d}u. 
\end{align*}
This implies 
\begin{align}
F'\brac{ t' }& \leq \epsilon \brac{\sum_{l=0}^{j-1}  F_{2,l} +\mathbbm{1}_{ F_{2,j}>0} F_{2,j}} . \label{eq:boundFder}
\end{align}
Now, recall that $ k \epsilon \leq t \leq \brac{k \ +1}\epsilon$. By the fundamental theorem of calculus, the assumptions on F and~\eqref{eq:boundFder},
\begin{align*}
 F\brac{ t } & =  F\brac{ 0 }  + \int_{0}^{t} F'\brac{u} \text{d}u\\
 & = 1 + \sum_{j=0}^{k-1} \int_{j \epsilon}^{\brac{j+1}\epsilon } F'\brac{u} \text{d}u +  \int_{ k \epsilon }^{ t  } F'\brac{u} \text{d}u\\
 & = 1 + \epsilon^2\brac{  \sum_{j=0}^{ k-1 }  \brac{\sum_{l=1}^{j-1}  F_{2,l} +\mathbbm{1}_{ F_{2,j}>0} F_{2,j} }  +    \mathbbm{1}_{F_{2}>0}F_{2} }. 
\end{align*}

\end{document}